%% file: aap_final.tex
\documentclass[11pt]{article}
\pdfoutput=1
\usepackage[letterpaper]{geometry}
\RequirePackage{amsthm,amsmath,amsfonts,amssymb}
\RequirePackage[authoryear]{natbib}
\setcitestyle{authoryear,open={(},close={)}} 

\RequirePackage[colorlinks,citecolor=blue,urlcolor=blue]{hyperref}
\RequirePackage{graphicx}

\usepackage{amsmath, amssymb, amscd, amsthm, amsfonts,pbsi,aurical}
\usepackage{amssymb,amsthm,amsfonts,amsbsy,latexsym,amsxtra}
\usepackage{dsfont,todonotes}
\usepackage{graphicx,xcolor,color,bm,authblk}
\usepackage{mathtools,enumitem,todonotes,soul}
\usepackage[T1]{fontenc}
\usepackage[utf8]{inputenc}
\usepackage{float,mdframed}
\usepackage{mathrsfs} 

\usepackage{tikz}

\newcommand{\kb}[1]{{\color{red}\bf[KB: #1]}}
\newcommand{\zy}[1]{{\color{cyan}\bf[ZY: #1]}}
\newcommand{\cb}[1]{{\color{purple}\bf[CB: #1]}}
\newcommand{\WP}[1]{{\color{teal} #1}}
\newcommand{\WPcomment}[1]{{\color{teal}\bf[WP: #1]}}

\newcommand{\cc}[1]{\textcolor{purple}{#1}}

\numberwithin{equation}{section}

\makeatletter
\def\namedlabel#1#2{\begingroup
    #2%
    \def\@currentlabel{#2}%
    \phantomsection\label{#1}\endgroup
}
\makeatother

\newtheorem{Theorem}{Theorem}[section]

\newtheorem{Proposition}{Proposition}[section]
\newtheorem{Remark}{Remark}[section]

\newtheorem{Corollary}{Corollary}[section]
\newtheorem{Definition}{Definition}[section]
\newtheorem{Lemma}{Lemma}[section]

\newcommand{\mbb}{\mathbb}
\newcommand{\mbf}{\mathbf}
\newcommand{\tbf}{\textbf}
\newcommand{\mcal}{\mathcal}

\newcommand{\E}{\mathbf{E}}
\newcommand{\R}{\mathbb{R}}

\newcommand{\N}{\mathbb{N}}

\renewcommand{\emptyset}{\varnothing}

\newcommand{\Nb}{\mathbf{N}}

\newcommand{\lambdad}{\lambda_{d}}
\newcommand{\Png}{\mathcal{P}_{ng}}
\newcommand{\Pngm}{\mathcal{P}_{n\check{g}}}
\newcommand{\Plambda}{\mathcal{P}(\lambda)}
\newcommand{\Plambdam}{\mathcal{P}(\check{\lambda})}

\newcommand{\domD}{\mbf{dom}~D}
\newcommand{\rec}{\text{Rect}}

\DeclareMathOperator{\Var}{Var}
\DeclareMathOperator{\Cov}{Cov}

\allowdisplaybreaks

\begin{document}
\title{\sf \bf Multivariate Gaussian Approximation for Random Forest via Region-based Stabilization}

\author[1]{Zhaoyang Shi}
\author[2]{Chinmoy Bhattacharjee}
\author[1]{\authorcr Krishnakumar Balasubramanian}
\author[1]{Wolfgang Polonik}
\affil[1]{Department of Statistics, University of California, Davis}
\affil[2]{Department of Mathematics, Universität Hamburg}
\affil[1]{\texttt{\{zysshi,kbala,wpolonik\}}@ucdavis.edu}\affil[2]{\texttt{chinmoy.bhattacharjee}@uni-hamburg.de}
\date{}

\maketitle

\begin{abstract}
We derive Gaussian approximation bounds for $k$-Potential Nearest Neighbor ($k$-PNN) based random forest predictions based on a set of training points given by a Poisson process under fairly mild regularity assumptions on the data generating process. Our approach is based on the key observation that $k$-PNN based random forest predictions satisfy a certain geometric property called \emph{region-based stabilization}. We also compare the rates with those of $k$-nearest neighbor-based random forests, highlighting a form of universality in our result. In the process of developing our results, we also establish a probabilistic result on multivariate Gaussian approximation bounds for general functionals of Poisson process that are region-based stabilizing. This general result makes use of the Malliavin-Stein method, and is potentially applicable to various related statistical problems.  
\end{abstract}

\section{Introduction}\label{sec:intro}
A random forest~\citep{amit1997shape,ho1998random,breiman2001random} is an extremely successful general-purpose prediction method. It is based on aggregating randomized tree-based base-learners and has been successfully applied in a wide range of fields including remote sensing \citep{belgiu2016random}, healthcare \citep{khalilia2011predicting,qi2012random} and causal inference~\citep{wager2018estimation}. Different versions of random forest mainly differ by their randomized tree-building processes. Broadly speaking, there are two sources of randomness while constructing the random forest. 
\begin{itemize}[leftmargin=0.2in]
\item \textit{Bagging or sub-sampling:} The first is due to the use of bagging or sub-sampling, where a randomly chosen subset of the entire training data is used to construct the individual (random) trees.   \cite{breiman1996bagging} originally considered the random forest with bagging where the trees are constructed by bootstrap sub-samples from the original sample. It was also shown that for unstable base-learners, the bagging random forest can increase the accuracy; see \cite{breiman1996heuristics} for more details. Moreover, other works including those by \cite{peng2022rates}, \cite{mentch2016quantifying} and \cite{wager2014asymptotic} focus on averaging all possible sub-samples of the training, with a  particular fixed size.
\item \textit{Random tree construction:} The second source depends on the way in which the random-tree based base-learners are constructed. If the randomization in the tree construction only depends on the covariates, then such random forests are called non-adaptive random forests. They are also closely related to the kernel-based prediction techniques \citep{lin2006random,scornet2016random}. From the perspective of nearest neighbor regression methods, non-adaptive random forests could also be viewed as an implicitly adaptive nearest neighbor method. Conversely, if the randomization in the tree construction also depends on the response, such random forests are called adaptive random forests. 
\end{itemize}

Despite the widespread usage, progress in the theoretical understanding of their statistical properties has been rather slow. Motivated by the works of~\cite{breiman2000some,Breiman2004CONSISTENCYFA}, the most theoretically well-studied version of random forest has been the non-bagging and non-adaptive random forest. Specifically~\cite{lin2006random},~\cite{meinshausen2006quantile}, \cite{biau2010layered}, \cite{biau2008consistency}, \cite{biau2012analysis},  \cite{scornet2016asymptotics}, \cite{scornet2016random}, \cite{mourtada2020minimax}, \cite{klusowski2021sharp} and \cite{klusowski2022large} studied consistency properties, with a few of them also establishing minimax rates over various smooth function classes. In ~\cite{biau2016random}, the authors provide an exposition of the above results. Publications ~\cite{wager2014confidence}, \cite{mentch2016quantifying},~\cite{zhou2021v} and~\cite{cattaneo2023inference} studied asymptotic normality and developed asymptotically valid confidence intervals. The focus in the above works are mainly about the fluctuations with respect to the bagging procedure, and are agnostic to the randomization in the base-learners. Due to the asymptotic nature, the above works do not provide any insight into when the Gaussian behavior actually kicks-in, from a finite-sample perspective.

Many practical versions of random forests are truly adaptive~\citep{breiman2001random}. From a theoretical perspective, a condition that bridges adaptive and non-adaptive random forest is that of \emph{honesty} considered in various forms  in~\cite{biau2008consistency,biau2012analysis}. Roughly speaking, an honest (random) tree is defined as a tree that avoids using the same training labels for both selecting split-points for the tree construction process, and for making the predictions. This essentially makes the random forest non-adaptive for all statistical analysis purposes. This condition was further examined in detail in \cite{wager2018estimation}. Specifically, they argue that honesty-type conditions are essentially \emph{necessary} to obtain point-wise asymptotic \emph{Gaussianity} of non-bagging random forest predictions, and provide various examples. However, we would like to remark here that Gaussian limits (and other non-Gaussian limits) might be possible under certain non-standard scalings.

Alternatively,~\cite{scornet2015consistency} and~\cite{chi2022asymptotic} studied consistency of adaptive random forest (without requiring any honesty-type conditions) under additive model assumptions on the truth and with CART splitting criterion, respectively. We are not aware of any further theoretical analysis of statistical properties (e.g., minimax rates and asymptotic normality) of random forests without honesty-type conditions. Apart from studying classical statistical properties, works such as \cite{mentch2020randomization} and \cite{tan2022cautionary} have also looked at explanations for the stronger (or weaker) performance of random forests and related methods over other approaches with similar or comparable statistical properties. 

From a finite-sample inference perspective, it is essential to provide non-asymptotic Gaussian approximation bounds for random forest predictions. To the best of our knowledge, only \cite{peng2022rates} establishes such bounds for sub-sampling based non-adaptive random forest predictions in the \emph{univariate} setting. Their approach was based on directly leveraging standard results on Berry-Esseen bounds for $U$-statistics by Stein's method \citep[see e.g.,][Chapter 10]{chen2011normal}. They also specialized their result to the case when the base-learners are the so-called $k$-nearest neighbors ($k$-NNs) or $k$-potential nearest neighbors ($k$-PNNs); see Section~\ref{sec:kpnn} for details. However, as a consequence of a direct limitation of their proof techniques, they are only able to handle the case of fixed $k$. Handling the case of growing $k$ is highly non-trivial and requires a very different proof technique. In contrast to the above works, our main goal is to obtain multivariate Gaussian approximation bounds for random forest predictions (see Theorem~\ref{thmrfw} and Corollary \ref{thmrf}), focusing on the non-bagging and non-adaptive version in the context of regression with growing $k$. Apart from being applicable to non-bagging and non-adaptive random forests, our results are also applicable to the so-called purely random forest such as the one studied by~\cite{mourtada2020minimax}. 

\subsection{Our Contributions and the need for stabilization-based analysis} 

From a technical point-of-view, we show that random forests satisfy a certain (rectangular) region-based stabilization property (see Section~\ref{sec:connection}), which enables us to develop and leverage tools based on Stein's method to establish a Gaussian approximation result. Concretely, non-adaptive random forests based on $k$-PNNs with uniform weights are shown to have a (multivariate) Gaussian approximation rate $k^{\tau}\log^{-(d-1)/2}n$ for a constant $\tau>0$ and $d$ is the dimension of the predictor variables (see Corollary \ref{thmrf} for details). Note that this rate is only logarithmic in $n$, a behavior that might be surprising. However, this result holds under very weak assumptions. Crucially it uses {\em uniform} weights, meaning that all $k$ potential nearest neighbors get the same (nonzero) weight, and it only assumes continuity of the regression function. The former means that the $k$-PNN based random forest exhibits what might be called {\em local long-range dependence} that is {\em non-isotropic (rectangular)}. And it is this fact that motivates our use of region-based stabilization techniques.

In contrast to that, $k$-NNs exhibit what might be called short-range local dependence (see (\ref{tailprobfromATE})), and this leads to a very different behavior of  $k$-NN based random forests. For instance, assuming a H\"older smooth regression function and additive Gaussian noise, we show that such random forests exhibit a polynomial normal approximation rate of the form $k^{1/2}(k/n)^{\gamma/d}$ (up to a logarithmic term) for the Gaussian approximation, where $\gamma>0$ is the H\"{o}lder smoothness parameter (see Section \ref{sec:kpnn-and-knn} for details). For non-Gaussian errors a rate of approximation involves $1/k^{1/2}$ (using Berry-Esseen). Thus, for appropriate choice of $k$ we have a polynomially decreasing rate of approximation. It turns out that $k$-NN based random forests indeed are a special case of $k$-PNN based random forests, corresponding to a particular choice of weights. Moreover, we also show that, analogously, even $k$-PNN based random forests (not restricted to purely $k$-NNs) with hard-thresholding weights that are enforcing short-range local dependence also enjoy polynomial normal approximation rates (see also Section \ref{sec:kpnn-and-knn}). 

These findings, that will be further discussed in the following sections, reveal the relation between $k$-NN and $k$-PNN based random forests. One of the key differences is the local dependence structure: short-range and isotropic for $k$-NN based random forests and {\em long-range non-isotropic} for $k$-PNN based random forests. The latter motivates the use of region based stabilization techniques.

Stabilization-based approaches, in combination with Stein's method and second-order Poincar\'e inequalities are used to establish Gaussian approximation for functionals of Poisson and binomial point processes; see, e.g.,~\cite{last2016normal, lachieze2019normal, lachieze2020quantitative, schulte2019multivariate, shi2022flexible,schulte2023rates} and references therein for details. Recently~\cite{bhattacharjee2022gaussian} developed the notion of region-based stabilization which strictly generalizes standard stabilization,  and is more widely applicable. In this work, we extend the univariate results in~\cite{bhattacharjee2022gaussian} to the multivariate setting; Theorems~\ref{d2d3} and~\ref{dconvex} are widely applicable to a class of multivariate functionals of a Poisson process whose score functions satisfy the region-based stabilization property. Specializing these results to random forests, we obtain the multivariate Gaussian approximation bounds highlighted above; specifically, see Theorem~\ref{thmrfw} and Corollary \ref{thmrf}. 

All our results for the multivariate settings are under the Poisson sampling setting. Using our results for handling the widely studied case of independent and identically distributed (i.i.d.) observations, requires the use of a de-Poissonization technique. We use the Poisson sampling setting due to the technical reason that there is no natural multivariate second-order Poincar\'e inequality for the case of a binomial point processes; see Remark~\ref{toiid} for more details. However, we would like to highlight that the proof techniques in this paper are potentially valid, with some appropriate modifications, when a univariate normal approximation of the random forest is considered under a binomial sampling regime (i.e., i.i.d.\ samples). Furthermore, despite the fact that we focus on a non-adaptive random forest as an example in Theorem \ref{thmrfw} and Corollary \ref{thmrf}, our approach of using region-based stabilization theory and Stein's method to establish Gaussian approximation results, as stated in Theorems \ref{d2d3} and \ref{dconvex}, is potentially widely applicable for adaptive random forests and other non-parametric regression problems such as Nadaraya-Watson and wavelets-type in which case, one would need to work with appropriate regions (depending on the procedure) and then apply our general results in Theorems~\ref{d2d3} and~\ref{dconvex}.



The rest of the paper is organized as follows. In Section~\ref{sec:kpnn}, we formally introduce the statistical model we work with, and summarize the concept of $k$-potential nearest neighbors ($k$-PNN) and its connection to random forests. In Section~\ref{sec:Main}, we provide our multivariate Gaussian approximation bounds for random forest predictions as well as a bound on its bias. In Section~\ref{generalregion}, we introduce our main probabilistic result on multivariate Gaussian approximation bounds for functionals of Poisson processes, which are used to prove the results in Section~\ref{sec:Main}. For the convenience of the reader, we provide a list of frequently used mathematical notations in Section~\ref{sec:notation}. All the proofs are provided in the supplementary material (see Appendices \ref{Appendix A} -- \ref{Appendix D}).

\section{Random forests and $k$-potential nearest neighbors}\label{sec:kpnn}

We consider the following regression model:
\begin{align}\label{regressionmodel}
    \bm{y}=r(\bm{x},\bm{\varepsilon}),\quad (\bm{x},\bm{y})\in\mbb{R}^{d}\times \mbb{R},\;\bm{\varepsilon} \in \R,
\end{align}
where $\bm{x} \sim\mbb{Q}$ with a.e.\ continuous density $g$ on $\mbb{R}^{d}$, $d \in \N$ and the noise $\bm{\varepsilon} \sim P_{\bm{\varepsilon}}$ independent of $\bm{x}$. We define the true regression function as $    r_0(\cdot):=\mbb{E}_{P_{\bm{\varepsilon}}}[r(\cdot,\bm{\varepsilon})]=\mbb{E}_{P_{\bm{\varepsilon}}}[r(\bm{x},\bm{\varepsilon})|\bm{x}=\cdot].$ Note that we do not use any structural assumptions on $r$, such as additivity. 
We further define $\sigma^2(\cdot)\coloneqq\Var_{P_{\bm{\varepsilon}}}[r(\cdot,\bm{\varepsilon})]$, and highlight that we allow for heteroscedastic variance in the regression model.

 We model the distribution of the training samples $\{(\bm{x}_{i},\bm{y}_{i})\}_{i=1}^{N}$, where $\bm{y}_{i}:=\bm{y}_{\bm{x}_{i}} = r(\bm{x}_i,\bm{\varepsilon}_{\bm{x}_i})$ for $1 \le i \le N$, by assuming that the pairs $\{(\bm{x}_{i},\bm{\varepsilon}_{\bm{x}_i})\}_{i=1}^{N}$ are being drawn from an underlying marked Poisson process $\Pngm$ with intensity measure $n(\mbb{Q}\otimes P_{\bm{\varepsilon}})$
 (see Section \ref{generalregion} for definition and additional details). Here, $N$ is a Poisson random variable with mean $n$, and $\bm{\varepsilon}_{\bm{x}}$'s are independent marks associated to each point $\bm{x}$ in the sample. 
 We refer to this sampling as the Poisson sampling setting. 
This sample, as a collection of points in the product space $\mbb{R}^{d}\times \mbb{R}$, can be thought of as a mixed binomial point process $\sum_{i=1}^{N}\delta_{(\bm{x}_{i},\bm{\varepsilon}_i)}$, where $\delta_{(\bm{x}_{i},\bm{\varepsilon}_i)}$ is the Dirac measure at $(\bm{x}_{i},\bm{\varepsilon}_i)$. Here $N$ is Poisson with mean $n$ and $\{(\bm{x}_{i},\bm{\varepsilon}_i)\}_{i=1}^\infty$ are i.i.d.\ from $\mbb{Q} \otimes P_{\bm{\varepsilon}}$, independent of $N$. So $\bm{x}_i \sim \mbb{Q}$ and the noise $\bm{\varepsilon}_i \sim P_{\bm{\varepsilon}}$ are independent for $i \in \N$. 
Furthermore, we denote by $\Png$ the process obtained by projecting the marked Poisson process $\Pngm$ on $\R^d$ consisting of the Poisson sample $\{\bm{x}_{i}\}_{i=1}^{N}$.

Before we introduce the specific form of the random forest we study in this work, we introduce a geometric concept, the so-called $k$-Potential Nearest Neighbors ($k$-PNNs), which can be interpreted as a generalization of the classical $k$-nearest neighbors ($k$-NNs). The $k$-PNNs share a close connection with random forest as we explain subsequently.

For any $x_1=(x_1^{(1)},\ldots,x_1^{(d)}), x_2=(x_2^{(1)},\ldots,x_2^{(d)})\in\mbb{R}^{d}$, we define the hyperrectangle $\rec(x_1,x_2)$ defined by $x_1,x_2$, and its volume respectively as 
\begin{align*}
\rec(x_1,x_2):=\prod_{i=1}^{d}[x_1^{(i)}\wedge x_2^{(i)},x_1^{(i)}\vee x_2^{(i)}],\qquad\text{and}\qquad|x_1-x_2|:=\prod_{i=1}^{d}|x_1^{(i)}-x_2^{(i)}|.
\end{align*}

\begin{Definition}[$k$-PNN]\label{def:kpnn}
    Given a target point $x_0$, and a locally finite point configuration $\mu$ in $\mbb{R}^{d}$, a point $x \in \mu$ is said to be a $k$-PNN to $x_{0}$ (with respect to $\mu$) if there are fewer than $k$ points from $\mu \setminus \{x\}$ in $\rec(x,x_0)$. 
\end{Definition}

The number of $k$-PNNs to a target point $x_0$ is always larger than or equal to $k$, provided that the underlying configuration $\mu$ has at least $k$ points. Figure~\ref{fig:2PNN} illustrates an example of 2-PNNs to a point $x_0$ in a given configuration. One can also interpret $k$-PNNs in terms of monotone metrics. A metric $\mathsf{d}$ on $\mbb{R}^d$ is said to be monotone if for any two points $x_{1},x_{2}$, and any point $x$ in $\rec(x_{1},x_{2})$, one has $\mathsf{d}(x_1,x) \le \mathsf{d}(x_1,x_2)$. For instance, the Euclidean distance in $\mbb{R}^{d}$ is one such metric. Then, given a collection of points $\mu$, a point $x \in \mu$ is a $k$-PNN of a target point $x_0$, if and only if there exists a monotone metric under which $x$ is among the $k$ closest points in $\mu$ from $x_0$. Obviously, the classical $k$-NN is a special case of $k$-PNN with some chosen monotone metric. The case $k=1$ is a special case, and $1$-PNNs are also called layered nearest neighbors (LNNs). It has been observed that nearest neighbor methods with adaptively chosen metrics that are monotone demonstrate good empirical performance \citep{hastie1995discriminant,domeniconi2002locally}.

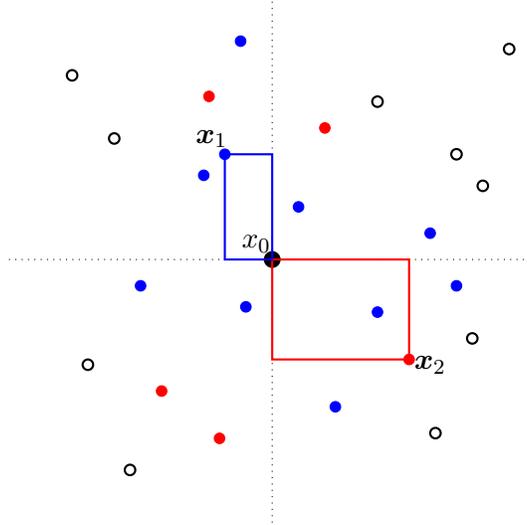
\begin{figure}[t!]
	\begin{center}
		\begin{tikzpicture}[scale = 7]
			\draw[dotted] (-.5, 0) to (.5,0);
			\draw[dotted] (0, -.5) to (0,.5);
			\foreach \Point in {(0,0)}{
				\draw[fill] \Point  circle [radius = .015];};
					\node at (-.03,.03) {$x_0$};
					
			\draw[thick,red](0,0) rectangle (.26,-.19);
					\node at (.3,-.2) {$\bm{x}_2$};
			\draw[blue,thick](0,0) rectangle (-.09,.2);
					\node at (-.115,.23) {$\bm{x}_1$};
				\foreach \Point in {(.45,.4),(.2,.3),(.35,.2),(.4,.14),
					(.38,-.15),(.31,-.33),
					(-.3,.23),(-.38,.35),
					(-.35,-.2),(-.27,-.4)}{
				\draw[thick] \Point  circle [radius = .01];};
			
			\foreach \Point in {(.05,.1),(.3,.05),
			(.2,-.1),(.35,-.05),(.12,-.28),
			(-.09,.2),(-.13,.16),(-.06,.415),
			(-.05,-.09),(-.25,-.05)}{
				\draw[fill,blue] \Point  circle [radius = .01];};
			\foreach \Point in {(.1,.25),
				(.26,-.19),
				(-.12,.31),
				(-.21,-.25),(-.1,-.34)}{
				\draw[fill,red] \Point  circle [radius = .01];};
			
		\end{tikzpicture}
		\caption{The set of 2-PNNs around a point $x_0 \in \R^2$. The point configuration includes all points in the figure except $x_0$. The \textcolor{blue}{blue} and \textcolor{red}{red} points \textit{together} are the 2-PNNs to $x_0$. The \textcolor{red}{red} ones such as $\bm{x}_2$ has exactly 1 point in its corresponding rectangle. The \textcolor{blue}{blue} ones, such as $\bm{x}_1$, are also a 1-PNN, or LNN, with no other point in the rectangle formed by $x_0$ and those points. \label{fig:2PNN}}
	\end{center}
\end{figure}

The notion of $k$-PNN is also intrinsically linked with the notions of ``dominance'' and ``number of maximal/minimal points''; see, e.g.,~\cite{bai2005maxima,bai2006rooted} and references therein. A point $x_1=(x_1^{(1)},\ldots,x_1^{(d)})\in\mbb{R}^{d}$ is said to dominate a point $x_2=(x_2^{(1)},\ldots,x_2^{(d)})\in\mbb{R}^{d}$ if $x_1-x_2\in\mbb{R}_{+}^{d}\backslash\{0\}$, i.e., $x_1^{(i)}> x_2^{(i)}$ for all $i\in[d]$, represented by the binary relations $x_1\succ x_2$ or $x_2\prec x_1$. Furthermore, points in the sample not dominating any other points are called minimal (or Pareto optimal) points of the sample, and points that are not dominated by any other sample points are called maximal. Thus, LNNs to a point $x_0 \in \R^d$ can be thought of as a collection consisting of $2^d$ independent copies (one copy for each quadrant) of the classical minimal points w.r.t.\ $x_0$.

With this background, we now describe the  non-bagging and non-adaptive random forest predictors that we analyze in the present work. For a given target point, all its $k$-PNNs in the training set are also called  its \emph{voting points}. The prediction for that target point is then expressed as a (randomly) weighted linear combination of the labels corresponding to the voting points; see, e.g.,~\cite{lin2006random,biau2010layered}.
The  non-adaptiveness comes from the fact that both the weights and the randomized splitting scheme used to construct the base decision trees of the random forest do not depend on the labels. Furthermore, as discussed by \cite{lin2006random}, regardless of the tree generating schemes, as long as the terminal nodes of each randomized tree define rectangular areas, voting points are all $k$-PNNs to $x_0$ and all $k$-PNNs to $x_{0}$ can be voting points. Particularly, if $k=1$, the above procedure is also called as Layered Nearest Neighbor based prediction~\citep{biau2010layered,wager2014asymptotic}.

For a given test point $x_0\in\mathbb{R}^d$, the random forest type estimator studied in this paper is of the form
\begin{align}\label{rnk}
    r_{n,k,w}(x_0):=\sum_{(\bm{x},\bm{\varepsilon}_{\bm{x}})\in \mcal{P}(\check{\lambda})}W_{n\bm{x}}(x_0)\bm{y}_{\bm{x}}.
\end{align}
Here the weights 
$W_{n\bm{x}}(x_0) \equiv W_{n\bm{x}}(x_0,\Plambdam)$ are nonnegative Borel measurable functions of $x_0$, of the  samples $\{\bm{x}_{i}\}_{i=1}^{N}$ and of the random variables used to generate randomized trees independent of the sample. Note here that we assume the weights to not depend on the marks ${\varepsilon}_{\bm{x}}$; see Remark~\ref{rem:markremoval} for discussions on removing this assumption. The subscript $n$ indicates the dependence of the weights on the given configuration of the Poisson process, which is such that $W_{n\bm{x}}(x_0)=0$ if $\bm{x}\notin \mcal{L}_{n,k}(x_0)$, where $\mcal{L}_{n,k}(x_{0}) \equiv \mcal{L}_{n,k}(x_{0},\Plambdam)$ is the set of all $k$-PNNs to $x_0$ in the Poisson sample, and 
\begin{align}\label{eq:wprob}
    \sum_{(\bm{x},\bm{\varepsilon}_{\bm{x}})\in \mcal{P}(\check{\lambda})}W_{n\bm{x}}(x_0)=\sum_{(\bm{x},\bm{\varepsilon}_{\bm{x}})\in \mcal{P}(\check{\lambda})}W_{n\bm{x}}(x_0)\mathds{1}_{\bm{x}\in\mcal{L}_{n,k}(x_{0})}=1.
\end{align}
In other words, one can view the weights $\{W_{n\bm{x}}(x_0)\}_{\bm{x}\in \mcal{L}_{n,k}(x_0)}$ as a probability mass function of a distribution over all $k$-PNNs of $x_0$. If this distribution is uniform, we have the following so-called $k$-PNN estimator:
\begin{align}\label{rnkunif}
    r_{n,k}(x_0):=\sum_{(\bm{x},\bm{\varepsilon}_{\bm{x}})\in \mcal{P}(\check{\lambda})}\frac{\mathds{1}_{\bm{x}\in\mcal{L}_{n,k}(x_{0})}}{L_{n,k}(x_0)}\bm{y}_{\bm{x}},
\end{align}
where 
$L_{n,k}(x_{0})\equiv L_{n,k}(x_{0},\Plambdam)=|\mcal{L}_{n,k}(x_0,\Plambdam)|$. Here by convention, the sum is zero when $\mcal{P}(\check{\lambda})$ is empty, which by properties of Poisson processes happens with the exponentially small probability $e^{- \check{\lambda}(\R^d \times \R)}=e^{-n}$.
Unlike $k$-NNs, the number of $k$-PNNs is usually larger than $k$ and actually, it is increasing both in $k$ and in $n$. For instance, \cite{lin2006random} shows that if the density $g$ of the distribution $\mbb{Q}$ is  bounded from above and below on $[0,1]^{d}$, $\mbb{E}L_{n,k}(x_{0})$ is of the order $k\log^{d-1}n$. 

\section{Main results: Rates of multivariate Gaussian approximation}\label{sec:Main}
\subsection{Probability metrics}\label{sec:metric}
We now introduce the integral probability metrics that we use in the present work to quantify the error in Gaussian approximations.

Let $\bm{z}_1=(\bm{z}_1^{(1)},\ldots,\bm{z}_1^{(d)})$ and $\bm{z}_2=(\bm{z}_2^{(1)},\ldots,\bm{z}_2^{(d)})$ be two $d$-dimensional random vectors. Denote by $\mcal{H}_{d}^{(2)}$ the class of all $C^{2}$-functions $h:\mbb{R}^{d}\rightarrow \mbb{R}$ such that 
\begin{align*}
    |h(x_1)-h(x_2)|\le \|x_1-x_2\|, \quad x_1,x_2 \in \mbb{R}^{d}, \quad \text{and} \quad \underset{x\in\mbb{R}^{d}}{\sup}~\|\text{Hess}~h(x)\|_{\text{op}}\le 1,
\end{align*}
where Hess $h$ is the Hessian of $h$, and let $\mcal{H}_{d}^{(3)}$ be the class of all $C^{3}$-functions such that the absolute values of the second and third derivatives are bounded by 1. The $\mathsf{d}_2$- and $\mathsf{d}_3$-distances between the laws of $\bm{z}_1$ and $\bm{z}_2$ are given respectively by 
\begin{align*}
    \mathsf{d}_2(\bm{z}_1,\bm{z}_2)&\coloneqq\underset{ h\in \mcal{H}_{d}^{(2)}}{\sup}~|\mbb{E}(h(\bm{z}_1))-\mbb{E}(h(\bm{z}_2))|,\\
    \mathsf{d}_3(\bm{z}_1,\bm{z}_2)&\coloneqq\underset{h\in \mcal{H}_{d}^{(3)}}{\sup}~|\mbb{E}(h(\bm{z}_1))-\mbb{E}(h(\bm{z}_2))|.
\end{align*}
Here as well as for the metric $\mathsf{d}_{\textsf{cvx}}$ below, to simplify notation, with a slight abuse of notation we write the distances between the random vectors $\bm{z}_1$ and $\bm{z}_2$, while they are indeed distances between their laws. The distances $\mathsf{d}_2$ and $\mathsf{d}_3$ are well-defined for random vectors $\bm{z}_1$ and $\bm{z}_2$ that satisfy $\mbb{E}(\|\bm{z}_1\|),\mbb{E}(\|\bm{z}_2\|)<\infty$, and $\mbb{E}(\|\bm{z}_1\|^{2}),\mbb{E}(\|\bm{z}_2\|^{2})<\infty$ respectively. We also use the following non-smooth integral probability metric given by
\begin{align*}
    \mathsf{d}_{\textsf{cvx}}(\bm{z}_1,\bm{z}_2):=\underset{h\in \mcal{I}}{\sup}~|\mbb{E}(h(\bm{z}_1))-\mbb{E}(h(\bm{z}_2))|,
\end{align*}
where $\mcal{I}$ is the set of all indicators of measurable convex sets in $\mbb{R}^{d}$. 
If we restrict the set $\mcal{I}$ to sets of the form $\mathds{1}_{\Pi_{i=1}^{d}(-\infty,t_i]}$ for all $t_{i}\in\mbb{R}$, $i\ge 1$, where $\Pi$ here means the Cartesian product, the distance becomes the so-called Kolmogorov distance given by
\begin{align*}
    \mathsf{d}_{K}(\bm{z}_1,\bm{z}_2):=\underset{(t_1,\ldots,t_{d})\in\mbb{R}^{d}}{\sup}~|\mbb{P}(\bm{z}_1^{(1)}\le t_1,\ldots,\bm{z}_1^{(d)}\le t_d)-\mbb{P}(\bm{z}_2^{(1)}\le t_1,\ldots,\bm{z}_2^{(d)}\le t_d)|.
\end{align*}
Note here that we trivially have $\mathsf{d}_{K} \le \mathsf{d}_{\textsf{cvx}}$. Thus, any bound on the $\mathsf{d}_{\textsf{cvx}}$ distance also holds for the Kolmogorov distance. The above probability metrics are widely used in the literature on quantitative bounds for Gaussian and non-Gaussian approximations.


\subsection{Gaussian approximation bounds for random forests}

We now present our main result providing rates for the multivariate Gaussian approximation of the random forest type estimator given by \eqref{rnk} for multiple test points $x_{0,1},x_{0,2},\ldots,x_{0,m}\in\mbb{R}^{d}$ for some $m\in \N$. For notational convenience, when $m=1$, we simply refer to $x_{0,1}$ as $x_{0}$. 

Recall the regression model in \eqref{regressionmodel}, and the Poisson processes $\Plambda$ and $\Plambdam$ in Section \ref{sec:kpnn}. For $m \in \N$ and $x_{0,i} \in \mbb{R}^d, i = 1,\ldots,m,$ let
\begin{align}\label{eq:rnkw}
\mbf{r}_{n,k,w}:=(r_{n,k,w}(x_{0,1}),\ldots,r_{n,k,w}(x_{0,m}))^{T}
\end{align}
denote the vector of corresponding random forest predictions as defined in \eqref{rnk}. Below, we write $\mathcal{P}_{x,\eta}:=\Pngm+\delta_{(x,\bm{\varepsilon}_x)}+\eta$ for the marked Poisson process $\Pngm$ with additional point $(x,\bm{\varepsilon}_x)$ and an additional finite collection of points $\eta \subset \mbb{R}^d \times \mbb{R}$. 

\begin{Theorem}\label{thmrfw}  Assume there exist $p>0$ and $\sigma^2 > 0$ such that
\begin{align*}
    \mbb{E}(|r(\bm{x},\bm{\varepsilon})|^{6+p}) < \infty\qquad\text{and}\qquad \sigma^{2}:=\inf_{x}~\sigma^2(x)>0.
\end{align*}
For $m \in \N$ and $x_{0,i} \in \mbb{R}^d, i = 1,\ldots,m,$ 
let $\mbf{r}_{n,k,w}$ be as in~\eqref{eq:rnkw}, with covariance matrix $\Sigma_{m}$.
%
Then, 
for $d, n\ge 2$ and $k=\mcal{O}(n^{\alpha})$ for $0<\alpha<1$, there exists $c_{g}>0$ depending on $d$, $\sigma^{2}$, $g$, $\alpha$ and $p>0$, such that for $\mbb{Q}^{m}$-almost all $(x_{0,1},\ldots,x_{0,m})$,
\begin{align*}
    \mathsf{d}\left(\Sigma_{m}^{-1/2}(\mbf{r}_{n,k,w}-\mbb{E}(\mbf{r}_{n,k,w})),\mathcal{N}\right)\le c_{g} m^{43/6} k^{\tau+1}\max_{j\in\{1,4\}}\big\{W(n,k)^{1/2+1/j}\log^{(d-1)/j}n\big\}
\end{align*}
for $\mathsf{d}\in \{\mathsf{d}_{2},\mathsf{d}_3,\mathsf{d}_{\textsf{cvx}},\mathsf{d}_{K}\}$, where, 
\begin{align}\label{eq:wnk}
W(n,k)\coloneqq\frac{\underset{i=1,\ldots,m}{\max}\left(\sup_{x,|\eta| \le 9}~\|W_{nx}(x_{0,i},\mathcal{P}_{x,\eta})\|_{L_{6+p}}\right)^{2}}{\underset{i=1,\ldots,m}{\min}\mbb{E}\left(\sum_{\bm{x} \in \Plambda}W_{n\bm{x}}(x_{0,i})^{2}\right)},
\end{align}
and 
\begin{align}\label{eq:tau}
\tau:=6\zeta\beta+6\beta+1/2+\lceil 21(1+\zeta)/(6+p/2)\rceil,
\end{align}
with 
\begin{itemize}[itemsep=0.1in]
    \item $\beta:=\frac{p}{32+4p}$, and $\zeta:=\frac{p}{40+10p}$,\ for $\mathsf{d}\in \{\mathsf{d}_{2}, \mathsf{d}_{3}\}$, resulting in $\tau\in 
    \big(\frac{63}{20},\frac{9}{2}\big)$.
    \item $\beta:=\frac{p}{72+6p}$, and $\zeta:=\frac{p}{84+14p}$,\ for $\mathsf{d}\in \{\mathsf{d}_{\textsf{cvx}},\mathsf{d}_{K}\}$, resulting in $\tau\in 
    \big(\frac{18}{7},\frac{9}{2}\big)$.
\end{itemize}
In both cases, $\tau$ is a decreasing function of $p$. 
\end{Theorem}

The main ingredient of the proof of Theorem~\ref{thmrfw} is a general multivariate Gaussian approximation bound (see Theorems~\ref{d2d3} and~\ref{dconvex}) for certain classes of functionals of Poisson process. Specifically the classes of functionals are expressed as sums of score functions (as in~\eqref{sumofscore}), with the scores themselves satisfying the so-called region-based stabilization property (see Definition~\ref{regionstab} and Section~\ref{sec:connection}). Intuitively speaking, bounds on the region of stabilization control the level of dependency in the statistic, thereby enabling the statistic to converge to a Gaussian limit. We show that the random forest statistic satisfies region-based stabilization and leverage  Theorems~\ref{d2d3} and~\ref{dconvex} to prove Theorem~\ref{thmrfw}. We defer the proof of the result to Appendix \ref{secofmainpf}.

\begin{Remark}
    The ceiling function in the exponent $\tau$ is due to the proof technique required in Lemmas \ref{lemma8}, \ref{lemma9} and \ref{lemma10}. While we believe the exponent could be improved by removing the ceiling functions, it may require additional tedious arguments in the proof, which we do not pursue for the sake avoiding a more complicated exposition.
\end{Remark}

\begin{Remark}[De-localization of weights]
The factor $W(n,k)$ is a function of $n$ and $k$, and it depends on the distribution of the random weights used to weigh the set of $k$-PNNs for a given test point. To have a meaningful normal approximation bound, it is required to decay to zero, as suggested by Theorem \ref{thmrfw}. In effect, this entails that the weight distribution needs to be de-localized. Indeed, even in the case of CLTs for weighted sums of general random variable, the weights need to be sufficiently de-localized to obtain Gaussian limits.  

Note also that using the Cauchy-Schwarz inequality followed by \eqref{eq:wprob} and Jensen's inequality, for the denominator term (considering $m=1$ for simplicity), we have that
\begin{align*}
    \resizebox{0.99\hsize}{!}{$\mbb{E}\left(\sum_{\bm{x} \in \Plambda}W_{n\bm{x}}(x_0)^{2}\right)\ge \mbb{E}\left(\frac{\big(\sum_{\bm{x}\in \Plambda}W_{n\bm{x}}(x_0)\big)^{2}}{L_{n,k}(x_0)}\right)\ge \left(\mbb{E}[L_{n,k}(x_{0})]\right)^{-1}\gtrsim k^{-1}\log^{-(d-1)}n$},
\end{align*}
where the last inequality is due to Lemma \ref{lemma11}. From the definition of $W(n,k)$, this implies that a smaller supremum of the weights in the numerator (meaning the weights are more equally distributed) will result in a tighter upper bound, see Corollary \ref{thmrf}. 
\end{Remark}

\begin{Remark}
    The fractional powers $1/j$, for $j=1,4$, in the bound corresponds to the fractional powers in $\Gamma_1 - \Gamma_6$ from Theorem \ref{dconvex}, which is used to prove Theorem \ref{thmrfw}. These in turn come from the use of the multivariate second order Poincar\'{e} inequality, see Theorem \ref{2ndpoincare}.
\end{Remark}

As a corollary to Theorem \ref{thmrfw}, we obtain the rates of convergence for multivariate Gaussian approximation in the case of uniform weights, i.e., for the $k$-PNN estimator given by \eqref{rnkunif}. See Appendix \ref{secofmainpf} for its proof.
\begin{Corollary}\label{thmrf}
Under the setting of Theorem \ref{thmrfw}, let $r_{n,k}$ be as in \eqref{rnkunif} with $k \ge 11$ with the uniform weights. Then, there exists $c_{u}>0$ depending on $d$, $\sigma^{2}$, $\alpha$ and $p>0$ such that for $\mbb{Q}^{m}$-almost all $(x_{0,1},\ldots,x_{0,m})$,
\begin{align}\label{eq:unifrates}
    \mathsf{d}\left(\Sigma_{m}^{-1/2}(\mbf{r}_{n,k}-\mbb{E}(\mbf{r}_{n,k})),\mcal{N}\right)\le c_{u}\, \frac{m^{43/6}\,k^{\tau}}{\log^{(d-1)/2}n},\qquad \mathsf{d}\in\{\mathsf{d}_2,\mathsf{d}_3,\mathsf{d}_{\textsf{cvx}},\mathsf{d}_{K}\},
\end{align}
where $\Sigma_{m}$ is the covariance matrix of $\mbf{r}_{n,k}$ and $\tau$ is as in~\eqref{eq:tau}. In particular, for $m$ fixed, if $k=o\big(\log^{(d-1)/(2\tau)}n\big)$, then $\mbf{r}_{n,k}$ is asymptotically normal.
\end{Corollary}
\begin{Remark}
    As a corollary of Theorem \ref{thmrf}, one could just replace the general weights in \eqref{rnk} by the uniform weights given in \eqref{rnkunif} to compute the quantity $W(n,k)$ in \eqref{eq:wnk}. However, simply doing so will result in a slightly worse bound in the power of $k$ as Theorem \ref{thmrf} is proved without knowing any additional specific information on the weights thus $W(n,k)$. We therefore adopted a more refined proof argument for Corollary \ref{thmrf} with the uniform weights.
\end{Remark}

We now make a few additional remarks pertaining to both Theorem \ref{thmrfw} and Corollary \ref{thmrf}.

 \begin{Remark}[Moment condition]\label{mc4+p} The assumption of $(6+p)$ moments in Theorem \ref{thmrfw} is only needed for a multivariate normal approximation result. It can be relaxed to a $(4+p)$-moment condition when considering a univariate normal approximation, utilizing results by \cite{bhattacharjee2022gaussian}.
\end{Remark}

 \begin{Remark}[Binomial Point Processes]\label{toiid}
Apart from the Poisson setting considered in the present work, the case of i.i.d.\ sampling (i.e., binomial point processes) is also of interest. According to \cite{bhattacharjee2022gaussian}, for the i.i.d.\ case, a univariate normal approximation result for region-based stabilization can be achieved by adapting the scheme elaborated by \citet[Theorem 4.3]{lachieze2019normal} and bounding the required terms similarly as done in the proof of Theorem \ref{dconvex}. Furthermore, replacing the Poisson cumulative distribution function (c.d.f.)  in \eqref{regionofstab} by a binomial c.d.f., one can follow the subsequent line of argument to derive a univariate version of rates of normal approximation paralleling our Theorems \ref{thmrfw} and Corollary \ref{thmrf} for the i.i.d.\ sampling case. Specifically, if we let $m=1$ and consider the uniform weights based predictor $r_{n,k}(x_0)$ as in \eqref{rnkunif}, we obtain
\begin{align*}
    \mathsf{d}_{K}\bigg(\frac{r_{n,k}(x_0)-\mbb{E}(r_{n,k}(x_0))}{\sqrt{\Var (r_{n,k}(x_0)})},\mcal{N}\bigg)\le c_{u}\frac{k^{\tau}}{\log^{(d-1)/2}n},\qquad \text{for }\mbb{Q}-\text{almost all}\ x_0\in\mbb{R}^{d}.
\end{align*}
According to \cite{schulte2023rates}, second-order Poincar\'{e} inequalities for the multivariate normal approximation of Poisson functionals have no available counterparts for binomial point processes. Thus, there are no immediate versions of multivariate (i.e., for $m>1$) normal approximations (i.e., analogs of Theorem~\ref{d2d3} and Theorem~\ref{dconvex}) of region-based stabilizing functionals under i.i.d.\ samples. This remains an open problem, with applications beyond the scope of the current work.
\end{Remark}

\begin{Remark}[Comparison to MSE rates]
It has been emphasized, for instance, by \cite{lin2006random}, that as $k$ increases the mean squared error (MSE) of the random forest estimator $r_{n,k,w}$ is at least of the order $k^{-1}\log^{-(d-1)}n$. Hence, by picking $k$ appropriately, one can obtain (near) optimal rates of convergence for the MSE in regression problems under various assumptions. This is in contrast to the dependence on $k$ in Corollary~\ref{thmrf}; see also the following remarks. 
\end{Remark}

\begin{Remark}[Dependence on $k$]
Recall the random forest predictor with uniform weights as in \eqref{rnkunif}. As mentioned previously, we show later in Section \ref{generalregion} that the statistic in \eqref{rnkunif} is a sum of certain score functions that satisfies a region-based stabilization property. In particular, as $k$ increases, the region of stabilization for each summand becomes large resulting in increased dependency between the scores at different points, deviating much further away from a i.i.d. setup, which has a negative effect for obtaining Gaussian limiting distributions.

%


In particular, a main part of our proof of Theorem~\ref{thmrfw} and Corollary~\ref{thmrf} is related to bounding the cumulative distribution function (in terms of $k$) of the Poisson probability that $x$ is in the set of $k$-PNNs to $x_0$ denoted by $\mcal{L}_{n,k}(x_0)$. This plays a key role in controlling the size of the region of stabilization and results in the $k^\tau$ term in the numerator of~\eqref{eq:unifrates}. In Lemma \ref{lowertightk}, we establish a lower bound in $k$ matching our upper bound. This is due to the fact that, as $k$ increases, for points $x$ in a set of substantial measure, the probability that $x$ is in the region of stabilization given by $\mbb{P}(\textup{Poi}(\varpi)<k)$, where $\varpi=n\int_{\rec(x_0,x)}g(z)dz$, is close to 1,  meaning that the region of stabilization is substantially large; 
see Lemma \ref{lowertightk} for more details. Hence, the $k^\tau$ term cannot be further improved using the current proof technique (i.e., using region-based stabilization and Stein's method). Resolving this question of optimal $k$ dependency, either by demonstrating that the order of $k$ is necessary or by improving the $k$ dependency is thus an important open question.  



\end{Remark}

\begin{Remark}[Optimality in $n$]
In terms of $n$, our bounds are presumably optimal. This can be noted, by considering  the response $\bm{y}_{\bm{x}}$ to be integer-valued.  \citet[Lemma 4.1]{pekoz2013degree} prove that for integer-valued functionals \eqref{sumofscore}, it leads to a lower bound matching the classical Berry-Esseen upper bound mentioned above in terms of $n$, we then have the claimed optimality result. 
Moreover, under additional assumptions of the weights and the smoothness of the regression function, we can have an improved polynomial decaying rate, which is similar to the $k$NN based random forest. See the discussion in Section \ref{sec:kpnn-and-knn} for details.
\end{Remark}

\begin{Remark}[Extension to bagging random forests]
As mentioned in Section \ref{sec:intro},~\cite{wager2014asymptotic}, \cite{mentch2016quantifying} and \cite{peng2022rates} studied the bagging random forest based on sub-sampling from the entire training data to construct the base-learners. Particularly, this random forest is in essence expressible in the form of a U-statistic, given by
\begin{align*}
    U_{n,s}=\binom{n}{s}^{-1}\sum_{(n,s)}h(x_{i1},\ldots,x_{is};w_{i}),
\end{align*}
where $\{x_{i1},\ldots,x_{is}\}$ are subsamples of size $s$ from $n$ i.i.d. samples $\{x_{1},\ldots,x_{n}\}$ according to some randomness $w_{i}$ and $h$ is an estimator that is permutation invariant in its arguments. Typically, the estimators $h$ are tree-style base-learners such as $k$-NN estimators and $k$-PNN estimators discussed in Section~\ref{sec:kpnn}. 

Among the aforementioned works, only~\cite{peng2022rates} derived Gaussian approximation bounds in the univariate setting, albeit for fixed $k$. An approach to improve the rate in $k$ in our results is to combine the $U$-statistics based sub-sampling approach with the region-based stabilization proof technique that we introduce in this work. We believe this is an intricate problem and requires further non-trivial efforts. 
\end{Remark}

\begin{Remark}[Generalization to metric-valued data]
    Although Theorem \ref{thmrf} is stated in the context of the input data taking values in the Euclidean space, the proof techniques and the concept of $k$-PNNs actually do not rely on the geometry or topology of Euclidean spaces as it only requires a monotone metric. Therefore, the above result could potentially be generalized to other metric spaces of the inputs $\bm{x}$ by considering $k$-PNNs under other metrics~\citep{haghiri2018comparison}. Indeed, our main probabilistic results (see Theorem~\ref{dconvex} and Theorem~\ref{d2d3}) used to prove Theorem~\ref{thmrfw} are derived for general metric spaces. 
\end{Remark}

\begin{Remark}[Extension to adaptive random forest]\label{rem:markremoval}
In Theorem \ref{thmrf} we consider non-adaptive random forest, i.e., the weights $W_{n\bm{x}}(x_0)$ in \eqref{rnk} are not depending on the response $\bm{y}_{\bm{x}}$. Indeed, the independence between $W_{n\bm{x}}(x_0)$ and $\bm{y}_{\bm{x}}$ due to non-adaptivity is used in~\eqref{eq:handlingadaptiveweights}. However, we would like to highlight that our general result, stated later in Theorem \ref{dconvex}, can be used to derive multivariate Gaussian approximation of the adaptive random forest.  Indeed, it might be possible to use honesty-type assumptions to directly bound the left-hand side of \eqref{eq:handlingadaptiveweights} resulting in a more complicated analysis. A detailed examination of this is left as future work. 
\end{Remark}

\subsubsection{Comparison and Connection to $k$-Nearest Neighbor Case}\label{sec:kpnn-and-knn}
In practice, $k$-NN based random forests are also commonly considered. Importantly, the differences between random forests constructed over  $k$-NNs and $k$-PNNs, respectively, become increasingly significant in high dimensions (in dimension 1 the two coincide). The $k$-NNs exhibit a ball-shaped local dependency structure with its radius determined by the $k$-nearest neighbor distance, while $k$-PNNs show a rectangular-shaped long range anisotropic local dependence.

Particularly, let us consider a special case of the $k$-PNN-based random forest in \eqref{rnk} with general weights \eqref{eq:wprob}, when it becomes a $k$-NN estimator. To this end, choose the weights in \eqref{eq:wprob} as $W_{n\bm{x}}(x_0)=k^{-1}\mathds{1}_{\bm{x}\in \mcal{K}(x_0)}$, where $\mcal{K}(x_0)$ is the set of all $k$NNs of $x_0$. By definition of $k$-PNNs, if a point $\bm{x}$ is a $k$-NN of $x_0$, it must also be a $k$-PNN of $x_0$. Therefore, our choice of $W_{n\bm{x}}$ satisfies \eqref{eq:wprob} and indeed makes the $k$-NN estimator a special case of the $k$-PNN estimator. For simplicity, let us focus on the univariate case for the Gaussian approximation, i.e., consider a single base point $x_0$ with a compact support for $\bm{x}\in\mbb{X}$, and assume an additive model with Gaussian noise $\bm{\varepsilon} \sim \mcal{N}(0,\sigma^2)$. The corresponding $k$-NN estimator is 
\begin{align*}
    r_{n,k}(x_0)^{\text{knn}}=\frac{1}{k}\sum_{i=1}^{n}\mathds{1}_{\bm{x}_i\in \mcal{K}(x_0)}(r_0(\bm{x}_i)+\bm{\varepsilon}_i).
\end{align*}
Then, it holds that
\begin{align*}
    &r_{n,k}(x_0)^{\text{knn}}-\mbb{E}r_{n,k}(x_0)^{\text{knn}}=\frac{1}{k}\sum_{i=1}^n\mathds{1}_{\bm{x}_i\in\mcal{K}(x_0)}\bm{\varepsilon}_{i}\\
    &\hspace*{4cm}+\Bigg[\frac{1}{k}\sum_{i=1}^n\mathds{1}_{\bm{x}_i\in\mcal{K}(x_0)}r_0(\bm{x}_i)-\mbb{E}\left(\frac{1}{k}\sum_{i=1}^n\mathds{1}_{\bm{x}_i\in\mcal{K}(x_0)}r_0(\bm{x}_i)\right)\Bigg].
\end{align*}
It is seen that the first term in the sum is simply an average of exactly (w.p.\ $1$ when the point set is large enough) $k$ i.i.d. $\bm{\varepsilon}_i$'s due to the independent noise assumption. We denote it by $\bar{\bm{\varepsilon}}$. Furthermore, we can write the second term as
\begin{align*}
    &\frac{1}{k}\sum_{i=1}^n\mathds{1}_{\bm{x}_i\in\mcal{K}(x_0)}r_0(\bm{x}_i)-\mbb{E}\left(\frac{1}{k}\sum_{i=1}^n\mathds{1}_{\bm{x}_i\in\mcal{K}(x_0)}r_0(\bm{x}_i)\right)\\
    &=\frac{1}{k}\sum_{i=1}^n\mathds{1}_{\bm{x}_i\in\mcal{K}(x_0)}(r_0(\bm{x}_i)-r_0(x_0))-\mbb{E}\left(\frac{1}{k}\sum_{i=1}^n\mathds{1}_{\bm{x}_i\in\mcal{K}(x_0)}(r_0(\bm{x}_i)-r_0(x_0))\right)\\
    &=:O_1-\mbb{E} O_1.
\end{align*}
According to the tail bound for the $k$-NN distance in \cite[Lemma A.1]{shi2024gaussian}, it holds that for any $r>0$, there exists $C \in (0,\infty)$ such that
\begin{align}\label{tailprobfromATE}
    \mbb{P}(\max_{\bm{x}\in\mcal{K}(x_0)}\|\bm{x}-x_0\|\ge r)\le Ce^{-\frac{C}{k}nr^d}.
\end{align}
Writing
\begin{align*}
    O_1&=\frac{1}{k}\sum_{i=1}^n\mathds{1}_{\bm{x}_i\in\mcal{K}(x_0),\max_{i}\|\bm{x}_i-x_0\|< r}(r_0(\bm{x}_i)-r_0(x_0))\\
    &\quad+\frac{1}{k}\sum_{i=1}^n\mathds{1}_{\bm{x}_i\in\mcal{K}(x_0),\max_{i}\|\bm{x}_i-x_0\|\ge r}(r_0(\bm{x}_i)-r_0(x_0)),
\end{align*}
and taking $r=(C^{-1}kn^{-1}\log n)^{1/d}$ with $C>0$ as in \eqref{tailprobfromATE}, we thus obtain that conditional on the event $\max_{i}\|\bm{x}_i-x_0\|< r$ which occurs with probability at least $1-Cn^{-1}$, we have
\begin{align*}
    \left|O_1\right|\le \max_{i} |r_0(\bm{x}_i)-r_0(x_0)|.
\end{align*}
Furthermore, if $r_0$ is assumed to be H\"{o}lder continuous with exponent $\gamma>0$, it holds that conditional on the same event occurring with probability at least $1-Cn^{-1}$,
\begin{align}\label{concofO1}
    \left|O_1\right|\lesssim \left(\frac{k\log n}{n}\right)^{\frac{\gamma}{d}}.
\end{align}
Thus, noting that $\|r_0\|_{\infty}<\infty$, we have
$$
\mathbb{E}|O_1| \lesssim \left(\frac{k\log n}{n}\right)^{\frac{\gamma}{d}}+\frac{1}{n},
$$
and
$$
    \Var O_1\le \mbb{E}O_1^2\lesssim \left(\frac{k\log n}{n}\right)^{\frac{2\gamma}{d}}+\frac{1}{n}.
$$
Finally, note that
\begin{align*}
    &\mathsf{d}_K\left(\frac{r_{n,k}(x_0)^{\text{knn}}-\mbb{E}r_{n,k}(x_0)^{\text{knn}}}{\sqrt{\Var r_{n,k}(x_0)^{\text{knn}}}},\mcal{N}\right)\\
    &=\mathsf{d}_K\left(\frac{\bar{\bm{\varepsilon}}}{\sqrt{\Var r_{n,k}(x_0)^{\text{knn}}}}+\frac{O_1-\mbb{E}O_1}{\sqrt{\Var r_{n,k}(x_0)^{\text{knn}}}},\mcal{N}\right),
\end{align*}
where the first summand above is exactly normal since
\begin{align*}
    \frac{\bar{\bm{\varepsilon}}}{\sqrt{\Var r_{n,k}(x_0)^{\text{knn}}}}\sim \mcal{N}\left(0,\frac{\sigma^2}{k\Var r_{n,k}(x_0)^{\text{knn}}}\right),
\end{align*}
and the second term $\frac{O_1-\mbb{E}O_1}{\sqrt{\Var r_{n,k}(x_0)^{\text{knn}}}}$ can be viewed as a `bias', independent of the first term. We have the lower bound 
\begin{align*}
    \Var r_{n,k}(x_0)^{\text{knn}}=\frac{\sigma^2}{k}+\Var O_1\gtrsim k^{-1}.
\end{align*}
Hence the `bias' term can be bounded in expectation as
$$
\mathbb{E}\left|\frac{O_1-\mbb{E}O_1}{\sqrt{\Var r_{n,k}(x_0)^{\text{knn}}}}\right| \lesssim k^{1/2} \mathbb{E} |O_1| \lesssim k^{\frac{1}{2}}\left[\left(\frac{k\log n}{n}\right)^{\frac{\gamma}{d}} + \frac{1}{n}\right].
$$

Next, we will proceed by noting the fact that for $a,b > 0$,
\begin{equation}\label{linearnormaldistance}
    \mathsf{d}_{K}(a \mcal{N}+b,\mcal{N})\le b+ a\vee a^{-1}-1.
\end{equation} 
Note from above that $\Var r_{n,k}(x_0)^{\text{knn}} \ge \frac{\sigma^2}{k}$, and that
\begin{align*}
\sqrt{\frac{\Var r_{n,k}(x_0)^{\text{knn}}}{k^{-1}\sigma^2}}-1&\le \frac{\Var r_{n,k}(x_0)^{\text{knn}} - k^{-1}\sigma^2}{k^{-1}\sigma^2}= \frac{\Var O_1}{k^{-1}\sigma^2}\\
    &\lesssim k\left(\frac{1}{n}+\left(\frac{k\log n}{n}\right)^{\frac{2\gamma}{d}}\right).
\end{align*}
Thus (\ref{linearnormaldistance}) yields that
\begin{align*}
    &\mathsf{d}_K\left(\frac{r_{n,k}(x_0)^{\text{knn}}-\mbb{E}r_{n,k}(x_0)^{\text{knn}}}{\sqrt{\Var r_{n,k}(x_0)^{\text{knn}}}},\mcal{N}\right)\\
    &\lesssim \mathbb{E} \left|\frac{O_1-\mbb{E}O_1}{\sqrt{\Var r_{n,k}(x_0)^{\text{knn}}}}\right| + k\left(\frac{1}{n}+\left(\frac{k\log n}{n}\right)^{\frac{2\gamma}{d}}\right)\\
    &\lesssim k^{\frac{1}{2}}\left[\frac{1}{n}+ \left(\frac{k\log n}{n}\right)^{\frac{\gamma}{d}}\right]+k\left(\frac{1}{n}+\left(\frac{k\log n}{n}\right)^{\frac{2\gamma}{d}}\right).
\end{align*}
{\color{black}In the non-Gaussian noise case, the first term $\bar{\bm{\varepsilon}}$ is purely an average of non-Gaussian i.i.d.\ random variables, and then the Berry-Esseen theorem yields a Gaussian approximation rate of the order $k^{-1/2}$. Under the relatively weak condition that $k\asymp n^{a}$ for some $0<a<1$, this then will also be the rate of approximation for the $k$-NN.

The above arguments provide the intuition for the following ingredients to be important in the derivation of faster normal approximation rates: an additive noise regression model, isotropic smoothness (H\"{o}lder continuity) of the regression function, and the `short'-range isotropic local dependence, expressed by the magnitude of the distance of the $k$-NNs from $x_0$ (see \eqref{tailprobfromATE}). None of these types of assumptions are being made in our main results (i.e., Theorem~\ref{thmrfw}, and nonetheless we are able to derive a normal approximation rate for the $k$-PNN based random forest, highlighting a form of universality.


We next argue that under similar assumptions, also a $k$-PNN-based random forest can be expected to have improved rates of Gaussian approximation. For simplicity, let us consider $\bm{x}\in \mbb{X}=[-1,1]^d$ and $x_0=\bm{0}$. We study the following $k$-PNN estimator with weights that enforce a similar short-range dependence as for the $k$NN:}
\begin{align*}
    r_{n,k}^{\text{trun}}( \bm{0}) = \sum_{i=1}^{n}\frac{\mathds{1}_{\bm{x}_i\in\mcal{L}_{n,k}\cap H}}{L_{n,k}^{\text{trun}}}(r_0(\bm{x}_i)+\bm{\varepsilon}_i),
\end{align*}
where $H:=\rec(-h,h)$ is a hyperrectangle with $$h=\left(\left(\frac{k}{n}\right)^{\frac{1}{d}},\ldots,\left(\frac{k}{n}\right)^{\frac{1}{d}}\right)\in\mbb{R}^d$$ and $L_{n,k}^{\text{trun}}=|\mcal{L}_{n,k}\cap H|$. This is a still a $k$-PNN estimator with uniform weights. However, compared to the original one in \eqref{rnkunif}, we truncate the $k$-PNN points within a radius that is exactly the order of $k$-NN distance from \eqref{tailprobfromATE}. Then, its asymptotic behavior is similar to $r_{n,k}^{\text{knn}}(x_0)$. Moreover, now the `effective' sample size after truncation is $k$, according to Lemma \ref{lemma11}, the expected number of $k$-PNNs, $\mbb{E}L_{n,k}^{\text{trun}}$, is of order $k\log^{d-1}k$. Up to a logarithmic factor, the `effective' sample size is $k$, the same as $k$-NNs. Under the assumption that $r_0$ is H\"{o}lder continuous with exponent $\gamma>0$, following the proof for $r_{n,k}^{\text{knn}}(x_0)$, we again write
\begin{align*}
    \frac{r_{n,k}^{\text{trun}}( \bm{0})-\mbb{E}r_{n,k}^{\text{trun}}( \bm{0})}{\sqrt{\Var r_{n,k}^{\text{trun}}( \bm{0})}}=\frac{1}{L_{n,k}^{\text{trun}}}\sum_{i=1}^n \mathds{1}_{\bm{x}_i\in\mcal{L}_{n,k}\cap H}\bm{\varepsilon}_i+O_1'-\mbb{E}O_1',
\end{align*}
where we write the first term on the right hand side as $\tilde{\bm{\varepsilon}}$ for short and 
\begin{align*}
    O_1'=\frac{1}{L_{n,k}^{\text{trun}}}\sum_{i=1}^n \mathds{1}_{\bm{x}_i\in\mcal{L}_{n,k}\cap H}(r_0(\bm{x}_i)-r_0(x_0)).
\end{align*}
Due to the construction of the weights and H\"{o}lder smoothness, we also have
\begin{align*}
    |O_1'|\lesssim \left(\frac{k}{n}\right)^{\frac{\gamma}{d}}.
\end{align*}
Note that conditional on $L_{n,k}^{\text{trun}}$, $\tilde{\bm{\varepsilon}}$ is an average of i.i.d normal variables and independent of $O_1'$. Then, we still apply \eqref{linearnormaldistance}, obtain, arguing similarly as for $r_{n,k}(x_0)^{\text{knn}}$, that
\begin{align*}
    \mathsf{d}_K\left(\frac{r_{n,k}^{\text{trun}}( \bm{0})-\mbb{E}r_{n,k}^{\text{trun}}( \bm{0})}{\sqrt{\Var r_{n,k}^{\text{trun}}( \bm{0})}},\mcal{N}\right)
    &\lesssim (k\log^{d-1}k)^{\frac{1}{2}}\left(\frac{k}{n}\right)^{\frac{\gamma}{d}}+ k\log^{d-1}k\cdot \left(\frac{k}{n}\right)^{\frac{2\gamma}{d}}.
\end{align*}
Again, in the case when $\bm{\varepsilon}$ is non-Gaussian noise, we can use Berry-Esseen to obtain a the rate of Gaussian approximation for $\bar{\bm{\varepsilon}}$ of the order $(k\log^{d-1}k)^{-1/2}$, which will dominate the normal approximation rate for $r^{\text{trun}}_{n,k}(x_0)$ under weak conditions on $k$.

In general, $k$-NN based estimators having polynomial rates for Gaussian approximation have been investigated by \cite{shi2022flexible,shi2024gaussian}. However, as mentioned above, general $k$-PNN based estimators in Corollary \ref{thmrf} have a rectangle-shaped, long range, anisotropic local dependence. This fundamental difference from $k$-NNs (for $d\geq 2$) also makes its Gaussian approximation nontrivial, making all the above standard procedures invalid to apply. When $d=1$, it can be easily seen that for $k$-NN,
\begin{equation}\label{eq:pnn}
	\mathbb{E} \max_{\bm{x} \in \mathcal{K}(x_0)} \|\bm{x}-x_0\|  \asymp k/n.
\end{equation}
For $d=1$, the above fact is also true for $k$-PNNs, because in this case two nearest neighbors coincide. However, for $d > 1,$ $k$-PNNs behave quite differently. Consider, for instance, the case $d=2$ and $k=1$ and $\mbb{X}=[0,1]^2$. 
Then, $1$-PNNs are the so-called minimal points. It is well known that minimal points -- particularly when taking, for instance, the origin $x_0 = 0$ and considering minimal points within the unit square $[0,1]^2$ -- exhibit a notable property: the $x$-coordinate of the lowest minimal point in the square is uniformly distributed over $[0,1]$. For a formal derivation of this, see, e.g.,~\cite[Proof of Proposition 8]{penrose2004random}. 

Thus, instead of concentrating around some fractional power of $k/n$, one indeed has that $\mathbb{E} \max_{\bm{x} \in \mathcal{L}_{n,k}(x_0)} \|\bm{x}-x_0\|$ concentrates around a constant that is at least $1/2$. Note further that smoothness assumptions on the regression function $r_0$ do not help here, because $\mathbb{E} \max_{\bm{x} \in \mathcal{L}_{n,k}(x_0)} \|\bm{x}-x_0\|$ only depends on the covariates. This significant difference between the two types of nearest neighbors cannot be handled by simply following 
$k$NN designed stabilization techniques applied in \cite{shi2022flexible,shi2024gaussian} or other standard procedures. This motivates us to consider region based stabilization (see Section \ref{generalregion} for details), which is able to take into account such delicate long-range (of constant order) dependencies, i.e., introduced through the hyperectangles in the definition of $k$PNNs. Moreover, the above standard proofs need the additional assumptions of an additive regression model in \eqref{regressionmodel} and a smooth regression function (H\"{o}lder). Region based stabilization is applicable without these assumptions (see Corollary \ref{thmrf}). 

\subsection{Bias Analysis} To conclude this section, we present a quantitative analysis on the order of the bias of the random forest estimator with uniform weights \eqref{rnkunif} and show that under some regularity conditions the bias is not small enough to be negligible compared to the standard deviation. Recall that a function $\psi:\mbb{R}^{d}\rightarrow \mbb{R}$ is said to be H\"{o}lder continuous at $x_{0}\in\mbb{R}^{d}$ if there exist constants $L_{\psi}>0$ and $0<\gamma_{\psi}<1$ such that for $x\in\mbb{R}^{d}$,
\begin{align*}
    |\psi(x)-\psi(x_0)|\le L_{\psi}|x-x_{0}|^{\gamma_{\psi}}.
\end{align*}
%

\begin{Proposition}\label{varbias}
   Let the assumptions required for Corollary \ref{thmrf} prevail. In addition, assume the density $g(x)$ and the function $r_0(x)$ are H\"{o}lder continuous at $x_{0}$ with parameters $L_{g},\gamma_{g}>0$ and $L_{1},\gamma_{1}>0$, respectively, and assume $y_{x}=r(x,\varepsilon)$ to be uniformly bounded.
Then for any $0 < \zeta < 1$ and $x_{0}\in\mbb{R}^{d}$, 
the bias satisfies
\begin{align*}
    |\mbb{E}r_{n,k}(x_{0})-r_0(x_{0})|\lesssim \big(( \log^{-(\gamma_{g}\wedge \gamma_{1})\zeta}n)\vee (k^{-1/4}\log^{-(d-1)/4}n)\big)
\end{align*}
for $n$ large enough. 
\end{Proposition}

     See Appendix \ref{sec:propvarbias} for the proof. The bias of the subsampling random forest has been considered by \citet[Theorem 3.2]{wager2018estimation}. Our result above considers the non-subsampling version associated with $k$-PNN estimators and allow $k$ to increase with $n$. We want to emphasize that the bias bound here is consistent with that derived by \citet[Lemma 3 and section 5.3]{biau2010layered}, and we assume H\"{o}lder continuity to quantify the convergence of the bias to zero as mentioned above.

  As we mentioned before, the bias given above turns out to be relatively large compared to the standard deviation. Indeed, according to \eqref{rhobound} and Lemma \ref{lemma11}, the standard deviation is lower bounded by $(k\log^{d-1}n)^{-1/2}$ such that $(\Var r_{n,k}(x_0))^{-1/2}|\mbb{E}r_{n,k}(x_0)-r_0(x_0)|\rightarrow \infty$. A bias reduction technique is hence of great importance for inference on the true regression function $r_{0}$, for instance, deriving consistent confidence intervals. While some preliminary work is undertaken by \cite{mentch2016quantifying} with bootstrap in the context of bagging random forests, it remains an open problem how to reduce the bias to get non-asymptotically valid confidence intervals. We leave a detailed methodological study of this, including numerical simulation studies, for future work.

\section{Region-based stabilizing functionals and Gaussian approximation}\label{generalregion}
In this section, we introduce some preliminaries about point processes, region-based stabilizing functionals, and some related notation. We refer to~\cite{schulte2023rates} and \cite{bhattacharjee2022gaussian} for additional details and background. As mentioned earlier, the general multivariate  Gaussian approximation results (i.e., Theorems~\ref{d2d3} and~\ref{dconvex}) established in this section form the backbone for establishing our main result in Theorem~\ref{thmrfw} for random forests.

\subsection{Functionals of point processes}
Let $(\mbb{X},\mcal{F})$ be a measure space with a $\sigma$-finite measure $\mbb{Q}$. We will generally consider marked Poisson processes (e.g. see recent work of \cite{schulte2023rates} for more details) where the points live in $\mbb{X}$ while their marks live in a probability space $(\mbb{M},\mcal{F}_{\mbb{M}},\mbb{Q}_{\mbb{M}})$. 
Let $\check{\mbb{X}}:=\mbb{X}\times \mbb{M}$ with $\check{\mcal{F}}$ as the product $\sigma$-field of $\mcal{F}$ and $\mcal{F}_{\mbb{M}}$ and $\check{\mbb{Q}}:=\mbb{Q}\otimes\mbb{Q}_{\mbb{M}}$ is the product measure. When $(\mbb{M},\mcal{F}_{\mbb{M}},\mbb{Q}_{\mbb{M}})$ is a singleton endowed with a Dirac point mass, then the measure $\check{\mbb{Q}}$ reduces to $\mbb{Q}$. For $\check{x}\in\check{\mbb{X}}$, we shall use the representation $\check{x}:=(x,m_{x})$ with $x\in\mbb{X}$ and $m_{x}\in\mbb{M}$. Let \textbf{N} be the set of $\sigma$-finite counting measures on $(\check{\mbb{X}},\check{\mcal{F}})$, which can be interpreted as point configurations in $\check{\mbb{X}}$. The set \textbf{N} is equipped with the smallest $\sigma$-field $\mathscr{N}$ such that the maps $m_{A}:\textbf{N}\rightarrow \mbb{N}\cup\{0,\infty\},\mcal{M}\mapsto\mcal{M}(A)$ are measurable for all $A\in\check{\mcal{F}}$. A point process is a random element in \textbf{N}. For $\eta \in \textbf{N}$, we write $\check{x}\in\eta$ if $\eta(\{\check{x}\})\ge 1$. Furthermore, denote by $\eta_{A}$ the restriction of $\eta$ onto the set $A\in\check{\mcal{F}}$. For $\eta_1,\eta_2\in\tbf{N}$, we write $\eta_1\le \eta_2$ if the difference $\eta_1-\eta_2$ is non-negative. Denote by $\Plambda$ and $\Plambdam$ the Poisson processes with intensity measures $\lambda:=n\mbb{Q}$ (resp. $\check{\lambda}:=n\check{\mbb{Q}}$) on $(\mbb{X},\mcal{F})$ (resp. $(\check{\mbb{X}},\check{\mcal{F}})$).

To proceed, we need additional definitions and notation. Denote by $\tbf{F}(\tbf{N})$ the class of all measurable functions $f:\textbf{N}\rightarrow \mbb{R}$, and by $L^{0}(\check{\mbb{X}})=L^{0}(\check{\mbb{X}},\check{\mcal{F}})$ the class of all real-valued, measurable functions $F$ on $\check{\mbb{X}}$. Note that, as $\check{\mcal{F}}$ is the completion of $\sigma(\eta)$, each $F\in L^{0}(\check{\mbb{X}})$ can be written as $F=f(\eta)$ for some measurable function $f\in \tbf{F}(\textbf{N})$. Such a mapping $f$, called a \textit{representative} of $F$, is $\check{\mbb{Q}}\circ \eta^{-1}$-a.s. uniquely defined. In order to simplify the presentation, we make this convention: whenever a general function $F$ is introduced, we will select one of its representatives and denote such a representative mapping by the same symbol $F$. We denote by $L_{\Plambdam}^{2}(\check{\mbb{X}})=L_{\Plambdam}^{2}(\check{\mbb{X}},\check{\mcal{F}})$ the space of all square-integrable functions $F$ of the Poisson process $\Plambdam$ with $\mbb{E}(F^{2})<\infty$.

For $n, m\in \mbb{N}$, and $i\in[m]$, consider a collection of real-valued $\check{\mcal{F}} \otimes \mathscr{N}$-measurable score functions $\xi_{n}^{(i)}(\cdot,\cdot)$ defined on each pair $(\check{x},\eta)$, where $\check{x}\in\eta$ and $\eta \in\tbf{N}$. We are interested in the following functionals of the Poisson process $\Plambdam$:
\begin{align}\label{sumofscore}
    F_{n}^{(i)}=F_{n}^{(i)}(\Plambdam):=\sum_{\check{x}\in \Plambdam}\xi_{n}^{(i)}((x,m_{x}),\Plambdam).
\end{align}
We define $\bar{F}_{n}^{(i)}:=F_{n}^{(i)}-\mbb{E}[F_{n}^{(i)}]$ and seek to have a result on rates of multivariate normal approximation for the $m$-vector $\mbf{F}_{n}=(F_{n}^{(1)},\ldots,F_{n}^{(m)})$ or $\bar{\mbf{F}}_{n}=(\bar{F}_{n}^{(1)},\ldots,\bar{F}_{n}^{(m)})$ with a appropriate normalizer and $m\ge 1$.

\begin{Definition}[Cost/Difference Operators]\label{domF}
Let $F$ be a measurable function on $\Nb$. The family of add-one cost operators, $D=(D_{\check{x}})_{\check{x}\in \check{\mbb{X}}}$, are defined as 
\begin{align*}
	D_{\check{x}}F(\eta):=F(\eta+\delta_{\check{x}})-F(\eta), \quad \check x \in \check{\mbb{X}}, \eta \in \Nb.
\end{align*}
Similarly, we can define a second-order cost operator (also called iterated add-one cost operator): for any  $\check{x}_1,\check{x}_{2}\in\check{\mbb{X}}$ and $ \eta \in \Nb$,
\begin{align*}
	D^2_{\check{x}_1,\check{x}_2}F(\eta):=F(\eta+\delta_{\check{x}_1}+\delta_{\check{x}_2})-F(\eta+\delta_{\check{x}_1})-F(\eta+\delta_{\check{x}_2})+F(\eta).
\end{align*}
We say that $F$ belongs to the domain of the difference operator $F\in \domD$ if $F\in L_{\Plambdam}^{2}(\check{\mbb{X}})$ and
\begin{align*}
    \int_{\check{\mbb{X}}}\mbb{E}((D_{\check{x}}F)^{2})\check{\lambda}(d\check{x})<\infty.
\end{align*}
\end{Definition}

\begin{Definition}[Region of Stabilization]\label{regionstab}
For $n,m \in \mbb{N}$ we consider the class of $\check{\mcal{F}} \otimes \mathscr{N}$-measurable score functions $\xi_{n}^{(i)}(\check{x},\eta)$ for $i \in [m]$. Throughout the paper, we will always assume that if $\xi_{n}^{(i)}(\check{x},\eta_1)=\xi_{n}^{(i)}(\check{x},\eta_2)$ for some $\eta_1,\eta_2\in\tbf{N}$ with $0\neq \eta_1\le \eta_2$ then
\begin{align}\label{R0}
    \xi_{n}^{(i)}(\check{x},\mu_1)=\xi_{n}^{(i)}(\check{x},\eta')\quad \text{for all}\ \eta'\in\tbf{N}\ \text{with}\ \eta_1\le \eta'\le \eta_2.
\end{align}
This is a form of monotonicity property that is natural to any reasonable choice of score functions.
\end{Definition}

We now introduce some additional assumptions on the score functions that are sufficient to derive our Gaussian approximation results. Specifically, we assume that for each $i \in [m]$, the score functions $\xi_{n}^{(i)}(\check{x},\eta)$ are region-stabilizing \citep{bhattacharjee2022gaussian}, i.e., for all $n\ge 1$,
\begin{itemize}
    \item[\namedlabel{r1}{(\tbf{R1})}] there exists a map $R_{n}^{(i)}$ from $\{(\check{x},\eta)\in\check{\mbb{X}}\times \tbf{N}:\check{x}\in\eta\}$ to $\check{\mcal{F}}$ such that for all $\eta\in\tbf{N}$ and $\check{x}\in\eta$, we have that    
    \begin{align}\label{R1}
        \xi_{n}(\check{x},\eta)=\xi_{n}(\check{x},\eta_{R_{n}^{(i)}(\check{x},\eta)});
    \end{align}
    \item[\namedlabel{r2}{(\tbf{R2})}] the set 
    \begin{align*}
        \{(\check{x},\check{y}_1,\check{y}_2,\eta):\{\check{y}_1,\check{y}_2\}\subseteq R_{n}^{(i)}(\check{x},\eta+\delta_{\check{x}})\}
    \end{align*}
    is measurable with respect to the $\sigma$-field on $\check{\mbb{X}}^{3}\times\tbf{N}$;
    \item[\namedlabel{r3}{(\tbf{R3})}] the map $R_{n}^{(i)}$ is monotonically decreasing in the second argument:
    \begin{align*}
        R_{n}^{(i)}(\check{x},\eta_1)\supseteq R_{n}^{(i)}(\check{x},\eta_2),\quad \eta_1\le \eta_2,\ \check{x}\in\eta_1;
    \end{align*}
    \item[\namedlabel{r4}{(\tbf{R4})}]  for all $\eta\in\tbf{N}$ and $\check{x}\in\eta$, we have that
    $$\eta_{R_{n}^{(i)}(\check{x},\eta)}\neq 0 \implies
     (\eta+\delta_{\check{y}})_{R_{n}^{(i)}(\check{x},\eta+\delta_{\check{y}})}\neq 0,~\text{for all}~ \check{y}\notin R_{n}^{(i)}(\check{x},\eta).
     $$
\end{itemize}
Before moving on with our further assumptions, we note here that the notion of region-stabilization is a generalization of the idea of stabilization radius. In particular, while classically it is assumed that a stabilizing score function at a point is determined by the configuration inside a ball around the point, our Assumption \ref{r1} only requires a local region $R_{n}$, which is not necessarily a ball, on which the score function $\xi_{n}$ can be determined. Thus, the dependency between the score functions at different points could be measured only by the size of regions around those points alone, which leads to a Gaussian limit when the regions are small enough. An example where classical stabilization works well is the $k$-NN distance based for entropy estimation \citep{berrett2019efficient,shi2022flexible}, where the ball formed by the point and its $k$-th nearest neighbor determines the $k$-NNs. 

On the other hand, if we consider the $k$-PNNs (see Definition \ref{def:kpnn}), it turns out that considering balls is vastly suboptimal, and one needs to consider general regions to prove Gaussian convergence with presumably optimal rates. Our Assumption \ref{r3} is a geometric condition that roughly says that if we add more points to our configuration, the stabilization region $R_n^{(i)}$ can only get smaller, i.e., one would need to explore the configuration in a smaller region to determine the value of the score function. Such a property is very natural and is satisfied for most stabilizing functionals. The Assumptions \ref{r2} and \ref{r4} are rather technical ones, in particular, as noted by \citet[Section 2]{bhattacharjee2022gaussian}, Assumption \ref{r2} ensures that
\begin{align*}
    \{\eta\in\tbf{N}:\check{y}\in R_{n}(\check{x},\eta+\delta_{\check{x}})\}\in\mathscr{N}
\end{align*}
for all $(\check{x},\check{y})\in \check{\mbb{X}}^{2}$, and that
\begin{align*}
    \mbb{P}(\check{y}\in R_{n}(\check{x},\eta+\delta_{\check{x}}))\quad \text{and}\quad \mbb{P}(\{\check{y}_1,\check{y}_2\}\in R_{n}(\check{x},\eta+\delta_{\check{x}}))
\end{align*}
are measurable functions of $(\check{x},\check{y})\in\check{\mbb{X}}^{2}$ and $(\check{x},\check{y}_1,\check{y}_2)\in\check{\mbb{X}}^{3}$ respectively. 

\subsubsection{Connection to random forest}\label{sec:connection}

We now connect the terminology above with the random forest notation introduced in Section~\ref{sec:kpnn}. We have $\mbb{X}=\mbb{R}^{d}$ and the measure $\mbb{Q}$ is taken to be a probability with an a.e.\ continuous density $g$ with respect to the Lebesgue measure $\lambda_{d}$ on $\mbb{R}^{d}$. Hence, the intensity measure $\lambda=ng$.  Moreover, the mark space $\mbb{M}$ which represents in this case the domain of the noise $\bm{\varepsilon}$, is taken to be $\mbb{R}$ with $\mbb{Q}_{\mbb{M}}$ being its distribution $P_{\bm{\varepsilon}}$. Correspondingly the marked version of the intensity measure is $\check{\lambda}=n\check{g}$. We now let $\Pngm$ and $\Png$ denote the canonical Poisson process on $\check{\mbb{X}}$ (resp. $\mbb{X}$) with intensity measure $n\check{\mbb{Q}}$ (resp. $n\mbb{Q}$) for $n\ge 1$.  The random forest predictor in \eqref{rnk}, as well as the one in \eqref{rnkunif} with uniform weights, are given respectively by
\begin{align*}
    r_{n,k,w}(x_0)&\coloneqq\sum_{(\bm{x},\bm{\varepsilon}_{\bm{x}})\in \Pngm}W_{n\bm{x}}(x_0)\mathds{1}_{\bm{x}\in\mcal{L}_{n,k}(x_{0})}\bm{y}_{\bm{x}},\qquad\text{and}\\
    r_{n,k}(x_{0})&\coloneqq\sum_{(\bm{x},\bm{\varepsilon}_{\bm{x}})\in \Pngm}\frac{\mathds{1}_{\bm{x}\in\mcal{L}_{n,k}(x_{0})}}{L_{n,k}(x_0)}\bm{y}_{\bm{x}}.
\end{align*}


Now, observe that for $r_{n,k,w}(x_0)$,  is a region-based stabilizing functional with the region of stabilization given by
\begin{align*}
        R_{n}(\check{x},\Pngm):=
    \left\{
    \begin{aligned}
        &\rec(x_{0},x)\times \mbb{R},\quad \text{if}\ \Pngm((\rec(x_{0},x)\backslash \{x\})\times\mbb{R})<k,\\
        &\emptyset,\quad \text{otherwise}.
    \end{aligned}
    \right.
\end{align*}
This region is similar to the one considered by \citet[Theorem 2.2]{bhattacharjee2022gaussian} in the context of minimal points, and is indeed a generalized version exploiting the connection between $k$-PNNs and minimal points. 
In particular, for most $k$-PNNs, this region is thin in some directions and long in the other directions, which makes it  suboptimal for it to be enclose by a ball. Consequently, standard results on multivariate Gaussian approximation~\citep[e.g.,][]{schulte2023rates} are not immediately applicable in this example due to the fact that they require a ball with a small radius as the region of stabilization.

%
\subsection{Tail condition}

Going back to our general model, it is clear that with a general stabilization region as ours, we need some control on its size, so that the score functions are only locally dependent facilitating a Gaussian limit. This motivates the following assumption. Below, for $x \in \mbb{X}$, we write $\bm{m}_x$ to denote the random mark associated to $x$ independent of all else.

$\namedlabel{t}{(\tbf{T})}$ For each $i \in [m]$, assume that there exists a measurable function $r_{n}^{(i)}:\mbb{X}\times\mbb{X}\rightarrow [0,\infty]$ such that 
\begin{align}\label{tail}
    \mbb{P}((y,\bm{m}_y)\in R_{n}^{(i)}((x,\bm{m}_x),\Plambdam+\delta_{(x,\bm{m}_x)})\le e^{-r_{n}^{(i)}(x,y)},\quad x,y\in \mbb{X}\ \text{a.e.}
\end{align}
When $r_{n}^{(i)}$ does not vanish, Assumption \ref{t} is an analog of the usual exponential stabilization condition by \cite{schulte2023rates}. Note that $r_{n}^{(i)}$ is allowed to be infinity and the probability \eqref{tail} is well-defined due to Assumption \ref{r2}.

\subsection{Moment condition}
$\namedlabel{m}{(\tbf{M})}$ For some $p_0>0$, there exists $p>0$
such that for all $i \in [m]$ and $\eta\in\tbf{N}$ with $\eta(\check{\mbb{X}})\le 3+p_0$, 
\begin{align}\label{mc}
    \|\xi_{n}^{(i)}((x,\bm{m}_x),\Plambdam+\delta_{(x,\bm{m}_x)}+\eta)\|_{L_{p_0+p}}\le M_{n,p_0,p}^{(i)}(x),\quad n\ge 1,\ x\in\mbb{X}\ \text{a.e.},
\end{align}
where $M_{n,p_0,p}^{(i)}:\mbb{X}\rightarrow \mbb{R},\ n,m\ge 1$, $i \in [m]$ are measurable functions. When $M_{n,p_0,p}^{(i)}(x)$ is a constant not depending on $x$, it recovers the standard case by \cite{schulte2023rates} with uniformly bounded moments. For brevity of notation, in the sequel we will always write $M_n^{(i)}$ instead of $M_{n,p_0,p}^{(i)}$, and generally drop $p_0,p$ from all subscripts.

\subsection{Gaussian approximation}
We will require a few additional quantities to present our main results. 
For $i\in[m]$, let 
\begin{align}\label{qn}
    q_{n}^{(i)}(x_1,x_2)&:=n\int_{\check{\mbb{X}}}\mbb{P}(\{(x_1,\bm{m}_{x_1}),(x_2,\bm{m}_{x_2})\}\subseteq R_{n}^{(i)}(\check{z},\Plambdam+\delta_{\check{z}}))\check{\mbb{Q}}(d\check z).
    \end{align}
For $\zeta>0$, $y\in\mbb{X}$ and $i\in[m]$, let
\begin{align}\label{gn}
g_{n}^{(i)}(y)&:=n\int_{\mbb{X}}e^{-\zeta r_{n}^{(i)}(x,y)}\mbb{Q}(dx),\quad h_{n}^{(i)}(y):=n\int_{\mbb{X}}M_{n}^{(i)}(x)^{p_0+p/2}e^{-\zeta r_{n}^{(i)}(x,y)}\mbb{Q}(dx),
\end{align}
and
\begin{align}\label{Gn}
G_{n}^{(i)}(y)&:=M_{n}^{(i)}(y)+h_{n}^{(i)}(y)^{1/(p_0+p/2)}(1+g_{n}(y)^{p_0})^{1/(p_{0}+p/2)}.
\end{align}
For $\alpha>0$, $i,j,l,t\in[m]$, and $\alpha_i, \alpha_j, \alpha_l \ge 0$, 
define for $y\in\mbb{X}$,

\begin{align}\label{eq:firstfwithmultipleindex}
\begin{aligned}
    f_{\alpha_i,\alpha_j,\alpha_l,\alpha}^{(i,j,l,t)}(y)&\coloneqq n\int_{\mbb{X}}G_{n}^{(i)}(x)^{\alpha_i}G_{n}^{(j)}(x)^{\alpha_j}G_{n}^{(l)}(x)^{\alpha_l} e^{-\alpha r_{n}^{(t)}(x,y)}\mbb{Q}(dx)\\
    &+ n\int_{\mbb{X}}G_{n}^{(i)}(x)^{\alpha_i}G_{n}^{(j)}(x)^{\alpha_j}G_{n}^{(l)}(x)^{\alpha_l}e^{-\alpha r_{n}^{(t)}(y,x)}\mbb{Q}(dx)\\
    &+n\int_{\mbb{X}}G_{n}^{(i)}(x)^{\alpha_i}G_{n}^{(j)}(x)^{\alpha_j}G_{n}^{(l)}(x)^{\alpha_l}q_{n}^{(t)}(x,y)^{\alpha}\mbb{Q}(dx).
\end{aligned}
\end{align}
Moreover, we define for $x\in\mbb{X}$ and $i\in[m]$,
\begin{align}\label{kappan}
    \kappa_{n}^{(i)}(x)&:=\mbb{P}(\xi_{n}^{(i)}((x,\bm{m}_x),\Plambdam+\delta_{(x,\bm{m}_x)})\neq 0).
\end{align}

The above quantities are essential to our multivariate Gaussian approximation of region-based stabilizing functionals, where $q_{n}^{(i)}(x_1,x_2),g_{n}^{(i)}(y)$ and $\kappa_{n}^{(i)}(x)$ correspond to the tail probability condition \ref{t}, i.e., the ``size'' of the region of stabilization (see Lemma \ref{lemmaA3}), and $h_{n}^{(i)}(x),G_{n}^{(i)}(y)$ are associated with the moment condition \ref{m}. 

For $i\in[m]$, assume $F_{n}^{(i)} \in\domD$ defined in Section \ref{domF}. We also define $$\mathrm{P}^{-1}_n:=\textup{diag}(1/\varrho_{n}^{(1)},\ldots,1/\varrho_{n}^{(m)}),$$ as the normalizer for $\bar{\mbf{F}}_{n}$. Let $\Sigma:=(\sigma_{ij})_{i,j=1}^{m}\in\mbb{R}^{m\times m}$ be any given positive definite matrix and recall that $\mcal{N}_{\Sigma}$ be the $m$-dimensional normal random vector with mean $\mbf{0}$ and   covariance matrix $\Sigma$. Define
\begin{align}
   \Gamma_{0}&\coloneqq\sum_{i,j=1}^{m}|\sigma_{ij}-\Cov(({\bar{F}_{n}^{(i)}}/{\varrho_{n}^{(i)})},(\bar{F}_{n}^{(j)}/\varrho_{n}^{(j)})|,\label{GAMMA0}\\
   \Gamma_{1}&\coloneqq\left(\sum_{i,j=1}^{m} \frac{n\mbb{Q}\big(f_{1,1,0,\beta}^{(i,j,i,i)}\big)^{2}}{(\varrho_{n}^{(i)}\varrho_{n}^{(j)})^{2}}\right)^{\frac{1}{2}}\label{GAMMA1}\\
   \Gamma_{2}&\coloneqq\sum_{i=1}^{m}\frac{n\mbb{Q}\big((\kappa_{n}^{(i)}+g_{n}^{(i)})^{3\beta}(G_{n}^{(i)})^{3}\big)}{(\varrho_n^{(i)})^{3}}.  \label{GAMMA2} 
\end{align}

The following two theorems provide rates for the multivariate normal approximation of region-based stabilizing functionals measured by $\mathsf{d}_2$-, $\mathsf{d}_3$- and $\mathsf{d}_{\textsf{cvx}}$ distances defined in Section~\ref{sec:metric}. These results are generalizations of their univariate versions proved by \cite{bhattacharjee2022gaussian}. See Appendix \ref{Appendix A} for the proofs.

\begin{Theorem}[Multivariate Normal Approximation in $\mathsf{d}_2$- and $\mathsf{d}_3$-distances]\label{d2d3}
For $i\in[m]$, suppose the functional $F_{n}^{(i)}\in\domD$ assumes the form \eqref{sumofscore} with the score function $\xi_{n}^{(i)}$ satisfying Assumptions \ref{r1}-\ref{r4}, \ref{t} and \ref{m} for $p_0=4$ and $p >0$. 
Let $\zeta:=p/(40+10p)$ in \eqref{gn} and $\beta:=p/(32+4p)$ in \eqref{GAMMA0}-\eqref{GAMMA2}. Then for a positive definite matrix $\Sigma$ as above, we have
\begin{itemize}
    \item [(a)] for all $n\ge 1$, there exists a constant $c_3>0$ depending only on $p$, such that

    \begin{align*}
    \mathsf{d}_3\left(\mathrm{P}^{-1}_n\bar{\mbf{F}}_{n},\mcal{N}_{\Sigma}\right)\le c_3m\bigg(\Gamma_0+\Gamma_1+m\Gamma_2\bigg),
    \end{align*}

\hspace{-0.25in}and

    \item [(b)] for all $n\ge 1$, there exists a constant $c_2>0$ depending only on $p$, such that 

\begin{align*}
    \mathsf{d}_2\left(\mathrm{P}^{-1}_n\bar{\mbf{F}}_{n},\mcal{N}_{\Sigma}\right)\le c_2 \bigg(\|\Sigma^{-1}\|_{op}\|\Sigma\|_{op}^{\frac{1}{2}}\Gamma_0+\|\Sigma^{-1}\|_{op}\|\Sigma\|_{op}^{\frac{1}{2}}\Gamma_1+m^2\|\Sigma^{-1}\|_{op}^{\frac{3}{2}}\|\Sigma\|_{op}\Gamma_2\bigg).
\end{align*} 
\end{itemize}
\end{Theorem}


In order to state our next theorem for the $\mathsf{d}_{\textsf{cvx}}$ distance, we introduce the following additional terms. Define
    \begin{align}
        \Gamma_{3}&:=\sum_{i=1}^{m}\frac{n\mbb{Q}\big((\kappa_{n}^{(i)}+g_{n}^{(i)})^{6\beta}(G_{n}^{(i)})^{3}\big)}{(\varrho_n^{(i)})^{3}},\label{GAMMA3}\\
        \Gamma_{4}&:=\bigg(m\sum_{i=1}^{m}\frac{n\mbb{Q}\big((\kappa_{n}^{(i)}+g_{n}^{(i)})^{6\beta}(G_{n}^{(i)})^{4}\big)}{(\varrho_n^{(i)})^{4}}\bigg)^{\frac{1}{2}}+\bigg(\sum_{i,j=1}^{m}\frac{n\mbb{Q}\big(f_{2,2,0,3\beta}^{(i,j,i,i)}\big)}{(\varrho_{n}^{(i)}\varrho_{n}^{(j)})^{2}}\bigg)^{\frac{1}{2}},\label{GAMMA4}\\
       \Gamma_{5}&:=\sqrt{m}\bigg(\sum_{i,j,l,t=1}^{m}\sum_{s=1}^{m}\frac{n\mbb{Q}\big(f_{1,1,1/2,\beta}^{(i,j,l,t)}\big)^{2}}{\varrho_{n}^{(s)}(\varrho_{n}^{(i)}\varrho_{n}^{(j)})^{2}}\bigg)^{\frac{1}{3}},\label{GAMMA5}\\
        \Gamma_{6}&:=m^{3/4}\bigg(\sum_{i,j,l,t=1}^{m}\sum_{s=1}^{m}\frac{n\mbb{Q}\left(f_{1,1,1,\beta}^{(i,j,l,t)}\right)^{2}}{(\varrho_{n}^{(s)}\varrho_{n}^{(i)}\varrho_{n}^{(j)})^{2}}\bigg)^{\frac{1}{4}}.\label{GAMMA6}
    \end{align}

\begin{Theorem}[Multivariate Normal Approximation in $\mathsf{d}_{\textsf{cvx}}$-distance]\label{dconvex}
    For $i\in[m]$, suppose the functional $F_{n}^{(i)} \in \domD$ assumes the form \eqref{sumofscore} with the score function $\xi_{n}^{(i)}$ satisfying Assumptions \ref{r1}-\ref{r4}, \ref{t} and \ref{m} with $p_0=6$. Let $\zeta:=p/(84+14p)$ in \eqref{gn} and $\beta:=p/(72+6p)$ in \eqref{GAMMA0}, \eqref{GAMMA1} and \eqref{GAMMA3}-\eqref{GAMMA6}. 
    Then, for any positive definite matrix $\Sigma:=(\sigma_{ij})_{i,j=1}^{m}\in\mbb{R}^{m\times m}$, we have
    \begin{align*}
    \mathsf{d}_{\textsf{cvx}}\left(\mathrm{P}^{-1}_n\bar{\mbf{F}}_{n},\mcal{N}_{\Sigma}\right)&\le c_{\textsf{cvx}}\, m^5\,\left(\|\Sigma^{-1/2}\|_{op}\vee \|\Sigma^{-1/2}\|_{op}^{3}\right) \cdot \left(\Gamma_{0}\vee\Gamma_{1}\vee\Gamma_3 \vee\ldots\vee\Gamma_{6}\right),
    \end{align*}
    for all $n\ge 1$, where the constant $c_{\textsf{cvx}}$ depends only on $p$ and $m$. 
\end{Theorem}

\begin{Remark}
    Theorems \ref{d2d3} and \ref{dconvex} admittedly involve several complicated quantities which may seem hard to interpret. The backbone of these results is a multivariate second-order Poincar\'{e} inequality (see Theorem \ref{2ndpoincare}) proved by \citet[Theorems 1.1 and 1.2]{schulte2019multivariate}. Just as the classical Poincar\'{e} inequality provides concentration bounds for Poisson functionals in terms of their first order difference~\citep{wu2000new}, a second order Poincar\'{e} inequality provides a central limit theorem with non-asymptotic bounds on various distances between a Poisson functional and a Gaussian random variable/vector in terms of certain moments of the first and second order differences. In many examples, one needs to make some additional assumptions on the functionals to be able to optimally bound these moments. One such simplification is the assumption of stabilization, or as in our case, region-stabilization. All the quantities in our bounds, except $\Gamma_0$, are essentially upper bounds on these quantities involving various moments of the differences, under the additional assumption of region-stabilization. On the other hand, the term $\Gamma_0$ simply measures the error in approximation incurred due to replacing the sample covariance matrix by $\Sigma$.
\end{Remark}

We would like to highlight when the region of stabilization $R_{n}^{(i)}$ is taken to be a ball with its radius having an exponentially decaying tail, and the bound $M_{n}^{(i)}(x)$ in our moment condition \ref{m} is constant independent of $x$, Theorem \ref{dconvex} simplifies to the following bound proved by \cite{schulte2023rates}:
\begin{align}\label{simpledvx}
    \mathsf{d}_{\textsf{cvx}}\left(n^{-1/2}\bar{\mbf{F}}_{n},\mcal{N}_{\Sigma}\right)&\le c_{\textsf{cvx}}n^{-1/d},
\end{align}
where $\Sigma$ is taken to be the limiting covariance matrix. The relative complexity of our bounds compared to a result such as in \eqref{simpledvx} is mainly because of the fact that our setting and assumptions are more general. Specifically, compared to the work by \cite{schulte2023rates} including: (a) we assume general regions of stabilization instead of a ball, which is essential in our main applications Theorem \ref{thmrfw} and Corollary \ref{thmrf}; (b) in our moment condition, a non-uniform bound is assumed, which is often necessary to obtain optimal rates. It should be noted that in many statistical problems, a ball as region of stabilization and the uniformly bounded moments suffice to obtain presumably optimal rates of convergence, see \cite{lachieze2019normal,shi2022flexible,schulte2023rates} for several such statistical applications. Despite their seemingly complicated forms, our Theorems \ref{d2d3} and \ref{dconvex} are often required to obtain optimal rates, or even to prove Gaussian limits in some case, with the random forest results being a concrete example.  



\section{Discussion}

In this work, we provided multivariate Gaussian approximation bounds for non-adaptive and non-bagging random forest predictions in the multivariate setting. The proof technique is based on our observation that such random forests could be expressed as sum of score functions which subsequently satisfy a certain region-based stabilization property. Based on this observation, the Gaussian approximation bounds are derived by using Malliavin-Stein's method.

We conclude our paper with the following potential future direction. First note that combining our multivariate result with standard tightness arguments will entail that trained non-adaptive and non-bagging random forests are Gaussian processes (with a particular covariance function) in the limit. As discussed by~\cite{athey2019generalized}, honesty-type conditions (which essentially make random forest to be non-adaptive for all statistical analysis purpose) appear to be necessary to have Gaussian limits. It is interesting to explore limit theorems and distributional approximation bounds (both at multivariate and process level) for adaptive random forests, so as to reveal the advantages of adaptivity, when considering true regression functions to be coming from a more general function class, and also make further advances towards developing rigorous non-asymptotically valid statistical inference for random forest predictions.


\section{Notation}\label{sec:notation}
\begin{itemize}  \setlength\itemsep{0.7em}
    \item Boldface is used to denote random objects. 
    \item $a.e.$, $a.s.$: almost everywhere (resp. almost surely).
    \item $\mbb{R}_{+}$: $[0,\infty)$, non-negative real half-line.
    \item $\lambdad$: the Lebesgue measure on $\mbb{R}^{d}$ for $d \in \N$.
    \item Poi($\lambda_0$): Poisson random variable with parameter $\lambda_0$.
    \item $\delta_{x}$: the Dirac measure at $x$
    \item $[n]$: the set $\{1,2,\ldots,n\}$ for $n\in\mbb{N}$.
    \item $\|\cdot\|$: Euclidean norm
    \item $\|\cdot\|_{L_{p}}$: $L^{p}$-norm; for a random variable $\bm{x}$, $\|\bm{x}\|_{L_{p}}:=\left(\mbb{E}(|\bm{x}|^{p})\right)^{1/p}$ for $p>0$
    \item $\|\cdot\|_{op}$: operator norm.
    \item $\|\cdot\|_{\infty}$: supremum norm.
    \item For a function $f:[1,\infty)\rightarrow \mbb{R}_{+}$, we write $f(n)=\mcal{O}(n)$ to mean $f(n)/n$ is bounded for all $n\ge 1$ and $f(n)=o(n)$ to mean $f(n)/n\rightarrow 0$ as $n\rightarrow \infty$.
    \item $\mathds{1}_{A}$: indicator function over a set $A$
    \item $a\vee b$,$a\wedge b$: coordinatewise maximum (resp. minimum) of $a, b \in \R^d$
    \item $\rec(a,b)$: closed hyperrectangle defined by $a$ and $b$ in $\mbb{R}^{d}$
    \item $|A|$: cardinality of a set $A$
    \item $\mcal{N}_A$ denotes a zero-mean Gaussian random vector (or variable) with covariance matrix $A$. When $A=I$, we simply use $\mcal{N}$
    \item Throughout the paper, $A\lesssim_{\Delta} B$ means there exists a positive, non-zero constant $C_{\Delta}$ only depending on $\Delta$ such that $A\le C_{\Delta} B$. $A\lesssim B$ stands for the existence of an absolute constant $C$ such that $A\le C B$ and we write $A\asymp B$ if $A\lesssim B$ and $B\lesssim A$.
    \item $\Plambda$: Poisson process (see Section \ref{generalregion}) with intensity measure $\lambda:=n\mbb{Q}$.
    \item $\Plambdam$: marked Poisson process (see Section \ref{generalregion}) with intensity measure $\check{\lambda}:=n\mbb{Q}\otimes \mbb{Q}_{\mbb{M}}$. 
    \item $\Pngm$: marked Poisson process with intensity measure $n\check{\mbb{Q}}:=n\mbb{Q}\otimes \mbb{Q}_{\mbb{M}}$ where $\mbb{Q}$ has a density $g$ with respect to $\lambda_{d}$ and $\mbb{Q}_{M}$ is a probability measure.
    \item $\Png$: Poisson process with intensity measure $n\mbb{Q}$ where $\mbb{Q}$ has a density $g$ with respect to $\lambda_{d}$.
    \item $D_{x},D^{2}_{x,y}$: add-one (resp. second-order) difference operators; see Section \ref{generalregion}.
    \item $\domD$: the domain of the difference operator.
    \item $\mbb{Q}f$: For an integrable function $f$ on $\mbb{X}$, $\mbb{Q}f:=\int_{\mbb{X}}f(x)\mbb{Q}(dx)$.
\end{itemize}

\subsection*{Acknowledgement}
The first, third and fourth listed authors were supported by NSF Grant DMS-2053918. The second author was supported in part by the German Research Foundation (DFG) Project 531540467.

\appendix

\input{appendix}

\bibliographystyle{abbrvnat}
\bibliography{example}

\end{document}

%% file: appendix.tex
\section{Proofs of results in Section \ref{generalregion}}\label{Appendix A}

For multivariate normal approximation of Poisson functionals, the second order Poincar\'{e} inequalities \citep[see, e.g.,][Theorem 1.1 and 1.2]{schulte2019multivariate} serve as a key tool. Even though the result therein is stated for an unmarked Poisson process, we can obviously apply it to a marked point process by simply considering the marked space as the underlying space. Since we consider the marked Poisson process $\Plambdam$ with independent marks distributed as $\mbb{Q}_{\mbb{M}}$ and intensity measure $\check \lambda=\lambda\otimes \mbb{Q}_{\mbb{M}}=n\mbb{Q}\otimes \mbb{Q}_{\mbb{M}}$, where $\mbb{Q}$ is a $\sigma$-finite measure,
we state in Theorem \ref{2ndpoincare} below the marked version of the results. According to \citet[Theorem 1.1 and 1.2]{schulte2019multivariate}, and upon using the Cauchy-Schwarz inequality, we have the following result; see the proof of Theorem 4.5 therein for more details.

Let $\mbf{H}:=(H^{(1)},\ldots,H^{(m)})$ be a vector of functionals of the Poisson process $\Plambdam$ with $\mbb{E}[H^{(i)}]=0$ and $H^{(i)}\in\domD$ for all $i\in[m]$. Denote $D_{\check{x}}\mbf{H}:=(D_{\check{x}}H^{(1)},\ldots,D_{\check{x}}H^{(m)})$ and $D^2_{\check{x}_1,\check{x}_2}\mbf{H}:=(D^2_{\check{x}_1,\check{x}_2}H^{(1)},\ldots,D^2_{\check{x}_1,\check{x}_2}H^{(m)})$ for $\check{x},\check{x}_1,\check{x}_2\in\check{\mbb{X}}$ and define 
\begin{align*}
\gamma_{1}&:=\bigg(\sum_{i,j=1}^{m}\int_{\mbb{X}^{3}} (\mathbb{E}(D_{\check{x}_{1}}H^{(j)})^{2}(D_{\check{x}_{2}}H^{(j)})^{2})^{\frac{1}{2}} (\mathbb{E}(D^2_{\check{x}_{1},\check{x}_{3}}H^{(i)})^{2}\\
&\qquad\qquad\qquad\qquad\qquad\qquad \times (D^2_{\check{x}_{2},\check{x}_{3}}H^{(i)})^{2})^{\frac{1}{2}}\lambda^{3}(d(x_{1},x_{2},x_{3}))\bigg)^{\frac{1}{2}},\\
\gamma_{2}&:=\bigg(\sum_{i,j=1}^{m}\int_{\mbb{X}^{3}} (\mathbb{E}(D^2_{\check{x}_{1},\check{x}_{3}}H^{(i)})^{2}(D^2_{\check{x}_{2},\check{x}_{3}}H^{(i)})^{2})^{\frac{1}{2}}\\&\qquad\qquad\qquad\times (\mathbb{E}(D^2_{\check{x}_{1},\check{x}_{3}}H^{(j)})^{2}(D^2_{\check{x}_{2},\check{x}_{3}}H^{(j)})^{2})^{\frac{1}{2}}\lambda^{3}(d(x_{1},x_{2},x_{3}))\bigg)^{\frac{1}{2}},\\
\gamma_{3}&:=\sum_{i=1}^{m}\int_{\mbb{X}}\mbb{E}|D_{\check{x}}H^{(i)}|^{3}\lambda(dx),\\
\gamma_{4}&:=\bigg(\sum_{i,j=1}^{m}\int_{\mbb{X}}\mbb{E}(D_{\check{x}}H^{(i)})^{4}\lambda(dx)+6\int_{\mbb{X}^{2}}(\mbb{E}(D^2_{\check{x}_{1},\check{x}_{2}}H^{(i)})^{4})^{\frac{1}{2}}(\mbb{E}(D_{\check{x}_{1}}H^{(j)})^{4})^{\frac{1}{2}}\\&\quad\quad\quad\quad\lambda^{2}(d(x_{1},x_{2}))+3\int_{\mbb{X}^{2}}(\mbb{E}(D^2_{\check{x}_{1},\check{x}_{2}}H^{(i)})^{4})^{\frac{1}{2}}(\mbb{E}(D^2_{\check{x}_{1},\check{x}_{2}}H^{(j)})^{4})^{\frac{1}{2}}\lambda^{2}(d(x_{1},x_{2}))\bigg)^{\frac{1}{2}},\\
\gamma_{5}&:=\bigg(3\sum_{i,j=1}^{m}\int_{\mbb{X}^{3}}\left(\mbb{E}\mathds{1}_{D^2_{\check{x}_{1},\check{x}_{3}}\mbf{H}\neq \mbf{0},D^2_{\check{x}_{2},\check{x}_{3}}\mbf{H}\neq \mbf{0}}(\|D_{\check{x}_{1}}\mbf{H}\|+\|D^2_{\check{x}_{1},\check{x}_{3}}\mbf{H}\|)^{\frac{3}{4}}\right.\\&\left.\qquad\qquad\qquad\qquad\times(\|D_{\check{x}_{2}}\mbf{H}\|+\|D^2_{\check{x}_{2},\check{x}_{3}}\mbf{H}\|)^{\frac{3}{4}}|D_{\check{x}_{1}}H^{(i)}|^{\frac{3}{2}}|D_{\check{x}_{2}}H^{(i)}|^{\frac{3}{2}}\right)^{\frac{2}{3}}\\&\qquad\qquad\qquad\times (\mbb{E}|D_{\check{x}_{1}}H^{(j)}|^{3}|D_{\check{x}_{2}}H^{(j)}|^{3})^{\frac{1}{3}}\lambda^{3}(d(x_{1},x_{2},x_{3}))\\&\quad\quad+\sum_{i,j=1}^{m}\int_{\mbb{X}^{3}}\left(\mbb{E}(\|D_{\check{x}_{1}}\mbf{H}\|+\|D^2_{\check{x}_{1},\check{x}_{3}}\mbf{H}\|)^{\frac{3}{2}}(\|D_{\check{x}_{2}}\mbf{H}\|+\|D^2_{\check{x}_{2},\check{x}_{3}}\mbf{H}\|)^{\frac{3}{2}}\right)^{\frac{1}{3}}\\&\quad\quad\quad\quad\quad\quad\times \bigg(\frac{45}{2}(\mbb{E}|D^2_{\check{x}_{1},\check{x}_{3}}H^{(i)}|^{3}|D^2_{\check{x}_{2},\check{x}_{3}}H^{(i)}|^{3})^{\frac{1}{3}}(\mbb{E}|D_{\check{x}_{1}}H^{(j)}|^{3}|D_{\check{x}_{2}}H^{(j)}|^{3})^{\frac{1}{3}}\\&\quad\quad\quad\quad\quad\quad+\frac{9}{2}(\mbb{E}|D^2_{\check{x}_{1},\check{x}_{3}}H^{(i)}|^{3}|D^2_{\check{x}_{2},\check{x}_{3}}H^{(i)}|^{3})^{\frac{1}{3}}(\mbb{E}|D^2_{\check{x}_{1},\check{x}_{3}}H^{(j)}|^{3}|D^2_{\check{x}_{2},\check{x}_{3}}H^{(j)}|^{3})^{\frac{1}{3}}\bigg)\\&\qquad\qquad\qquad\qquad\lambda^{3}(d(x_{1},x_{2},x_{3}))\bigg)^{\frac{1}{3}},\\
\gamma_{6}&:=\bigg(3\sum_{i,j=1}^{m}\int_{\mbb{X}^{3}}\left(\mbb{E}\mathds{1}_{D^2_{\check{x}_{1},\check{x}_{3}}\mbf{H}\neq \mbf{0},D^2_{\check{x}_{2},\check{x}_{3}}\mbf{H}\neq \mbf{0}}(\|D_{\check{x}_{1}}\mbf{H}\|^2+\|D^2_{\check{x}_{1},\check{x}_{3}}\mbf{H}\|^2)^{\frac{3}{4}}\right.\\&\left.\qquad\qquad\qquad\qquad\times(\|D_{\check{x}_{2}}\mbf{H}\|^2+\|D^2_{\check{x}_{2},\check{x}_{3}}\mbf{H}\|^2)^{\frac{3}{4}}|D_{\check{x}_{1}}H^{(i)}|^{\frac{3}{2}}|D_{\check{x}_{2}}H^{(i)}|^{\frac{3}{2}}\right)^{\frac{2}{3}}\\&\qquad\qquad\qquad(\mbb{E}|D_{\check{x}_{1}}H^{(j)}|^{3}|D_{\check{x}_{2}}H^{(j)}|^{3})^{\frac{1}{3}}\lambda^{3}(d(x_{1},x_{2},x_{3}))\\&\quad\quad+\sum_{i,j=1}^{m}\int_{\mbb{X}^{3}}\left(\mbb{E}(\|D_{\check{x}_{1}}\mbf{H}\|^2+\|D^2_{\check{x}_{1},\check{x}_{3}}\mbf{H}\|^2)^{\frac{3}{2}}(\|D_{\check{x}_{2}}\mbf{H}\|^2+\|D^2_{\check{x}_{2},\check{x}_{3}}\mbf{H}\|^2)^{\frac{3}{2}}\right)^{\frac{1}{3}}\\&\quad\quad\quad\quad\quad\quad\times \bigg(\frac{135}{8}(\mbb{E}|D^2_{\check{x}_{1},\check{x}_{3}}H^{(i)}|^{3}|D^2_{\check{x}_{2},\check{x}_{3}}H^{(i)}|^{3})^{\frac{1}{3}}(\mbb{E}|D_{\check{x}_{1}}H^{(j)}|^{3}|D_{\check{x}_{2}}H^{(j)}|^{3})^{\frac{1}{3}}\\&\quad\quad\quad\quad\quad\quad+\frac{27}{8}(\mbb{E}|D^2_{\check{x}_{1},\check{x}_{3}}H^{(i)}|^{3}|D^2_{\check{x}_{2},\check{x}_{3}}H^{(i)}|^{3})^{\frac{1}{3}}(\mbb{E}|D^2_{\check{x}_{1},\check{x}_{3}}H^{(j)}|^{3}|D^2_{\check{x}_{2},\check{x}_{3}}H^{(j)}|^{3})^{\frac{1}{3}}\bigg)\\&\qquad\qquad\qquad\qquad\lambda^{3}(d(x_{1},x_{2},x_{3}))\bigg)^{\frac{1}{4}}.
\end{align*}
We note here that by a slight abuse of notation for the sake of brevity, in the expressions of $\gamma_i$, $1 \le i \le 6$ above as well as in integrals in the rest of this section, we write $\check{x}_1,\check{x}_2$ and $\check{x}_2$ to mean $(x_1,\bm{m}_{x_1})$, $(x_2,\bm{m}_{x_2})$ and $(x_3,\bm{m}_{x_3})$ respectively, i.e., we integrate only over $x_1,x_2$ and $x_3$, while their marks are random, and are integrated as part of the various expectations in the integrands. 

\begin{Theorem}[\cite{schulte2019multivariate} Theorems 1.1 and 1.2, and 4.5]\label{2ndpoincare} For $\mbf{H}:=(H^{(1)},\ldots,H^{(m)})$ a vector of functionals $\Plambdam$ with $\mbb{E}[H^{(i)}]=0$ and $H^{(i)}\in\domD$ for all $i\in[m]$, let $\gamma_i$, $1 \le i \le 6$ be as above.
Then, for a positive semi-definite matrix $\Sigma:=(\sigma_{ij})_{i,j=1}^{m}\in\mbb{R}^{m}\times\mbb{R}^{m}$, we have 
\begin{align*}
\mathsf{d}_3\left(\mbf{H},N_{\Sigma}\right)\le \frac{m}{2}\sum_{i,j=1}^{m}|\sigma_{ij}-\Cov(H^{(i)},H^{(j)})|+m\gamma_{1}+\frac{m}{2}\gamma_{2}+\frac{m^{2}}{4}\gamma_{3}.
\end{align*}
Additionally, if $\Sigma$ is positive definite, we have
\begin{align*}
        \mathsf{d}_2(\mbf{H},N_{\Sigma})&\le
        \|\Sigma^{-1}\|_{op}\|\Sigma\|_{op}^{\frac{1}{2}}\sum_{i,j=1}^{m}\left|\sigma_{ij}-\Cov(H^{(i)},H^{(j)})\right|+2\|\Sigma^{-1}\|_{op}\|\Sigma\|_{op}^{\frac{1}{2}}\gamma_{1}\\&\quad+\|\Sigma^{-1}\|_{op}\|\Sigma\|_{op}^{\frac{1}{2}}\gamma_{2}+\frac{\sqrt{2\pi}m^2}{8}\|\Sigma^{-1}\|_{op}^{\frac{3}{2}}\|\Sigma\|_{op}\gamma_{3},
\end{align*}
and
\begin{align*}
     \mathsf{d}_{\textsf{cvx}}(\mbf{H},N_{\Sigma})&\le 941m^5\left(\|\Sigma^{-1/2}\|_{op}\vee \|\Sigma^{-1/2}\|_{op}^{3}\right)\\&\times \Big(\sum_{i,j=1}^{m}\left|\sigma_{ij}-\Cov(H^{(i)},H^{(j)})\right|\vee\gamma_{1}\vee\ldots\vee\gamma_{6}\Big),
\end{align*}
where $\Sigma^{1/2}$ is the positive definite matrix such that $\Sigma^{1/2}\Sigma^{1/2}=\Sigma$ and $\Sigma^{-1/2}:=(\Sigma^{1/2})^{-1}$.
\end{Theorem}

\subsection{Proof of Theorem \ref{d2d3}}
The key idea of proving Theorem \ref{d2d3} is to bound the terms $\gamma_{i}$, $1\le i\le 6$, appearing in Theorem \ref{2ndpoincare} by combining properties of region of stabilization, Assumptions \ref{r1}-\ref{r4}, \ref{t} and \ref{m} with $p_0=4$. We start by noicing that by H\"{o}lder's inequality, for $q\in(0,4+p/2)$ and $\check{y},\check{y}_1,\check{y}_2\in\check{\mbb{X}}$, we have
\begin{align}\label{holder1}
    \mbb{E}|D_{\check{y}}F_{n}^{(i)}|^{q}\le (\mbb{E}|D_{\check{y}}F_{n}^{(i)}|^{4+p/2})^{\frac{q}{4+p/2}}\mbb{P}(D_{\check{y}}F_{n}^{(i)}\neq 0)^{\frac{4+p/2-q}{4+p/2}},
\end{align}
and
\begin{align}\label{holder2}
    \mbb{E}|D^2_{\check{y}_1,\check{y}_2}F_{n}^{(i)}|^{q}\le (\mbb{E}|D^2_{\check{y}_1,\check{y}_2}F_{n}^{(i)}|^{4+p/2})^{\frac{q}{4+p/2}}\mbb{P}(D^2_{\check{y}_1,\check{y}_2}F_{n}^{(i)}\neq 0)^{\frac{4+p/2-q}{4+p/2}}.
\end{align}
The following result, whose proof is immediate from definitions, reveals the connection between the cost functions $D_{\check{y}}F_{n}^{(i)}$, $D^2_{\check{y}_1,\check{y}_2}F_{n}^{(i)}$ and the score function $\xi_{n}^{(i)}$ for $i\in[m]$ and $n\ge 1$.
\begin{Lemma}\label{lemma1}
For $\check{y},\check{y}_1,\check{y}_2\in\check{\mbb{X}}$, $i\in[m]$ and $n\ge 1$,
\begin{align*}
    D_{\check{y}}F_{n}^{(i)}(\Plambdam)=\xi_{n}^{(i)}(\check{y},\Plambdam+\delta_{\check{y}})+\sum_{\check{z}\in\Plambdam}D_{\check{y}}\xi_{n}^{(i)}(\check{z},\Plambdam),
\end{align*}
and
\begin{align*}
    D^2_{\check{y}_1,\check{y}_2}F_{n}^{(i)}(\Plambdam)&=D_{\check{y}_1}\xi_{n}^{(i)}(\check{y}_2,\Plambdam+\delta_{\check{y}_2})+D_{\check{y}_2}\xi_{n}^{(i)}(\check{y}_1,\Plambdam+\delta_{\check{y}_1})\\
    &\qquad\qquad+\sum_{\check{z}\in\Plambdam}D^2_{\check{y}_1,\check{y}_2}\xi_{n}^{(i)}(\check{z},\Plambdam).
\end{align*}
\end{Lemma}

The next lemma implies the cost operator $D_{\check{y}}$ vanishes if $\check{y}$ is outside the region of stabilization.
\begin{Lemma}\label{lemma2}
    Let Assumptions \ref{r1}-\ref{r4} hold. For $\check{y},\check{y}_1,\check{y}_2\in\check{\mbb{X}}$, $i\in[m]$ and $n\ge 1$, we have that if $\check{y}\notin R_{n}^{(i)}(\check{x},\Plambdam+\delta_{\check{x}})$,
    \begin{align*}
        D_{\check{y}}\xi_{n}^{(i)}(\check{x},\Plambdam+\delta_{\check{x}})=0,
    \end{align*}
    and if $\{\check{y}_1,\check{y}_2\}\subsetneq R_{n}^{(i)}(\check{x},\Plambdam+\delta_{\check{x}})$,
    \begin{align*}
        D^2_{\check{y}_1,\check{y}_2}\xi_{n}^{(i)}(\check{x},\Plambdam+\delta_{\check{x}})=0.
    \end{align*}
\end{Lemma}
\begin{proof}
    According to \citet[Lemma 5.3]{bhattacharjee2022gaussian}, for the marked Poisson process $\Pngm$ and all $i\in[m]$, the result follows.
\end{proof}

\begin{Lemma}\label{lemma3}\label{lemmaA3}
    Let Assumptions \ref{r1}-\ref{r4} and \ref{t} hold. For $y,y_1,y_2 \in \mbb{X}$, we have
    \begin{align*}
        \mbb{P}(D_{\check{y}}F_{n}^{(i)}\neq 0)\le \kappa_{n}^{(i)}( y)+g_{n}^{(i)}(y),
    \end{align*}
    and
    \begin{align*}
        \mbb{P}(D^2_{\check{y}_1,\check{y}_2}F_{n}^{(i)}\neq 0)\le e^{-r_{n}^{(i)}({y}_2,{y}_1)}+e^{-r_{n}^{(i)}({y}_1,{y}_2)}+q_{n}^{(i)}({y}_1,{y}_2).
    \end{align*}
\end{Lemma}
\begin{proof}
    According to the Mecke formula, \eqref{gn}, \eqref{kappan}, Lemma \ref{lemma1} and Lemma \ref{lemma2}, we have
    \begin{align*}
        \mbb{P}(D_{\check{y}}F_{n}^{(i)}(\Plambdam)\neq 0)&\le \mbb{P}(\xi_{n}^{(i)}(\check{y},\Plambdam+\delta_{\check{y}})\neq 0)+\mbb{E}\sum_{\check{z}\in\Plambdam}\mathds{1}_{D_{\check{y}}\xi_{n}^{(i)}(\check{z},\Plambdam) \neq 0}\\&\le \kappa_{n}^{(i)}(y)+n\int_{\check{\mbb{X}}}\mbb{P}(D_{\check{y}}\xi_{n}^{(i)}(\check{z},\Plambdam+\delta_{\check{z}})\neq 0)\check{\mbb{Q}}(d\check{z})\\&\le\kappa_{n}^{(i)}({y})+g_{n}^{(i)}({y}),
    \end{align*}
    where the last inequality follows from the fact that $0<\zeta=p/(40+10p)<1$. Similarly, the Mecke formula, \eqref{tail}, \eqref{qn}, Lemma \ref{lemma1} and Lemma \ref{lemma2} yield
    \begin{align*}
    &\mbb{P}(D_{\check{y}_1,\check{y}_2}F_{n}^{(i)}(\Plambdam)\neq 0)\\
    &\le \mbb{P}(D_{\check{y}_1}\xi_{n}^{(i)}(\check{y}_2,\Plambdam+\delta_{\check{y}_2})\neq 0)+\mbb{P}(D_{\check{y}_2}\xi_{n}^{(i)}(\check{y}_1,\Plambdam+\delta_{\check{y}_1})\neq 0)\\
    &+\mbb{E}\sum_{\check{z}\in\Plambdam}\mathds{1}_{D^2_{\check{y}_1,\check{y}_2}\xi_{n}^{(i)}(\check{z},\Plambdam) \neq 0}\\
    &\le e^{-r_{n}^{(i)}({y}_2,{y}_1)}+e^{-r_{n}^{(i)}({y}_1,{y}_2)}+n\int_{\check{\mbb{X}}}\mbb{P}(D_{\check{y}_1,\check{y}_2}\xi_{n}^{(i)}(\check{z},\Plambdam+\delta_{\check{z}})\neq 0)\check{\mbb{Q}}(d\check{z})\\
    &\le e^{-r_{n}^{(i)}({y}_2,{y}_1)}+e^{-r_{n}^{(i)}({y}_1,{y}_2)}+q_{n}^{(i)}({y}_1,{y}_2).
    \end{align*}
\end{proof}
\begin{Lemma}\label{lemma4}
    Under Assumptions \ref{r1}-\ref{r4}, \ref{t} and \ref{m} with $p_0=4$, there exits a constant $C_p \in (0,\infty)$ depending only on $p$ such that for all $i\in[m]$, $n\ge 1$, ${y}\in{\mbb{X}}$ and $\eta\in\Nb$ with $\eta(\check{\mbb{X}})\le 1$, we have 
    \begin{align*}
        \mbb{E}|D_{\check{y}}F_{n}^{(i)}(\Plambdam+\eta)|^{4+p/2}&\le C_p\left(M_{n}^{(i)}(y)^{4+p/2}+h_{n}^{(i)}(y)(1+g_{n}^{(i)}(y)^4)\right)\\
        &\le C_p G_{n}^{(i)}(y)^{4+p/2}.
    \end{align*}
\end{Lemma}

\begin{proof}[Proof of Lemma \ref{lemma4}]
    The proof of the first inequality is an extension of a result by \citet[Lemma 5.5]{bhattacharjee2022gaussian} to the marked Poisson process $\Plambdam$, noting additionally that the intensity measure $n\check{\mbb{Q}}=n\mbb{Q}\otimes\mbb{Q}_{\mbb{M}}$ assumes a product form due to independent marks, and hence the marks can be integrated using the Cauchy-Schwarz inequality. The second inequality is stratghforward from Lemma \ref{eq:basicanalysis}.
\end{proof}

Now, we are in a position to prove Theorem \ref{d2d3}.

\begin{proof}[Proof of Theorem \ref{d2d3}]
According to Theorem \ref{2ndpoincare}, $\mathsf{d}_2$- and $\mathsf{d}_3$-distances only involve $\gamma_{1}$, $\gamma_{2}$ and $\gamma_{3}$. We begin by bounding $\gamma_{1}$. 

For $i\in[m]$, let $H^{(i)}:=(\varrho_{n}^{(i)})^{-1}\bar{F}_{n}^{(i)}$. By H\"{o}lder inequality, for $i,j\in[m]$,
\begin{align}\label{holder3}
        (\mathbb{E}(D_{\check{x}_{1}}F_n^{(j)})^{2}(D_{\check{x}_{2}}F_n^{(j)})^{2})^{\frac{1}{2}}&\le  (\mathbb{E}(D_{\check{x}_{1}}F_n^{(j)})^{4})^{\frac{1}{4}}(\mathbb{E}(D_{\check{x}_{2}}F_n^{(j)})^{4})^{\frac{1}{4}},
\end{align}
and
\begin{align}\label{holder4}
        (\mathbb{E}(D^2_{\check{x}_{1},\check{x}_{3}}F_n^{(i)})^{2}(D^2_{\check{x}_{2},\check{x}_{3}}F_n^{(i)})^{2})^{\frac{1}{2}}&\le (\mathbb{E}(D^2_{\check{x}_{1},\check{x}_{3}}F_n^{(i)})^{4})^{\frac{1}{4}}(\mathbb{E}(D^2_{\check{x}_{2},\check{x}_{3}}F_n^{(i)})^{4})^{\frac{1}{4}}.
\end{align}
According to \eqref{holder1} and Lemma \ref{lemma4}, we have
\begin{align}\label{640}
    (\mathbb{E}(D_{\check{x}_{1}}F_n^{(j)})^{4})^{\frac{1}{4}}&\le (\mbb{E}|D_{\check{x}_1}F_{n}^{(j)}|^{4+p/2})^{\frac{1}{4+p/2}}\mbb{P}(D_{\check{x}_1}F_{n}^{(j)}\neq 0)^{\frac{p}{32+4p}}\nonumber\\
    &\le C_{p}^{\frac{1}{4+p/2}}G_{n}^{(j)}(x_1).
\end{align}
Similarly,
\begin{align}\label{641}
    (\mathbb{E}(D^2_{\check{x}_{1},\check{x}_{3}}F_n^{(i)})^{4})^{\frac{1}{4}}&\le (\mbb{E}|D^2_{\check{x}_1,\check{x}_3}F_{n}^{(i)}|^{4+p/2})^{\frac{1}{4+p/2}}\mbb{P}(D^2_{\check{x}_1,\check{x}_3}F_{n}^{(i)}\neq 0)^{\frac{p}{32+4p}}.
\end{align}

By definition, 
\begin{align*}
    D^2_{\check{x}_{1},\check{x}_{3}}F_n^{(i)}(\Plambdam)&=D_{\check{x}_{1}}F_n^{(i)}(\Plambdam+\delta_{\check{x}_3})-D_{\check{x}_{1}}F_n^{(i)}(\Plambdam)\\&=D_{\check{x}_{3}}F_n^{(i)}(\Plambdam+\delta_{\check{x}_1})-D_{\check{x}_{3}}F_n^{(i)}(\Plambdam).
\end{align*}
Consequently, by Lemmas \ref{eq:basicanalysis} and \ref{lemma4}, we have
\begin{align*}
&\mbb{E}|D^2_{\check{x}_{1},\check{x}_{3}}F_n^{(i)}(\Plambdam)|^{4+p/2}\\
    &\le 2^{3+p/2}(\mbb{E}|D_{\check{x}_{1}}F_n^{(i)}(\Plambdam+\delta_{\check{x}_3})|^{4+p/2}+\mbb{E}|D_{\check{x}_{1}}F_n^{(i)}(\Plambdam)|^{4+p/2})\\
    &\le 2^{4+p/2}C_p G_{n}^{(i)}(x_1)^{4+p/2},
\end{align*}
as well as
\begin{align*}
    \mbb{E}|D^2_{\check{x}_{1},\check{x}_{3}}F_n^{(i)}|^{4+p/2}\le 2^{4+p/2}C_p G_{n}^{(i)}(x_3)^{4+p/2}. 
\end{align*}
Therefore, according to \eqref{641} and Lemma~\ref{lemma3}, we obtain
\begin{align}\label{642}
    (\mathbb{E}(D^2_{\check{x}_{1},\check{x}_{3}}F_n^{(i)})^{4})^{\frac{1}{4}}&\le 2C_p^{\frac{1}{4+p/2}} (G_{n}^{(i)}(x_1)\wedge G_{n}^{(i)}(x_3))
    \mbb{P}(D^2_{\check{x}_1,\check{x}_3}F_{n}^{(i)}\neq 0)^{\frac{p}{32+4p}}\nonumber\\&\le 2C_p^{\frac{1}{4+p/2}}(G_{n}^{(i)}(x_1)\wedge G_{n}^{(i)}(x_3))\left(e^{-\frac{p}{32+4p}r_{n}^{(i)}(x_3,x_1)}+e^{-\frac{p}{32+4p}r_{n}^{(i)}(x_1,x_3)}\right.\nonumber\\&\left.\qquad+q_{n}^{(i)}(x_1,x_3)^{\frac{p}{32+4p}}\right).
\end{align}

Combining \eqref{640} and \eqref{642}, and recalling~\eqref{eq:firstfwithmultipleindex}, we obtain
\begin{align}\label{gamma1b}
    \gamma_{1}^2&\le 4C_p^{\frac{4}{4+p/2}}\sum_{i,j=1}^{m}(\varrho_{n}^{(i)})^{-2}(\varrho_{n}^{(j)})^{-2}n\int_{\mbb{X}}\bigg(n\int_{\mbb{X}}G_{n}^{(j)}(x_1)G_{n}^{(i)}(x_1)\left(e^{-\frac{p}{32+4p}r_{n}^{(i)}(x_3,x_1)}\right.\nonumber\\&\left.\qquad\qquad\qquad+e^{-\frac{p}{32+4p}r_{n}^{(i)}(x_1,x_3)}+q_{n}^{(i)}(x_1,x_3)^{\frac{p}{32+4p}}\right){\mbb{Q}}(d x_1)\bigg)^{2}{\mbb{Q}}(dx_3)\nonumber\\&\le 4C_p^{\frac{4}{4+p/2}}\sum_{i,j=1}^{m}\frac{n \mbb{Q}\bigg(f_{1,1,0,\beta}^{(i,j,i,i)}\bigg)^{2}}{(\varrho_{n}^{(i)}\varrho_{n}^{(j)})^{2}}.
\end{align}
Using \eqref{642} again, we obtain
\begin{align}\label{gamma2b}
    \gamma_{2}^{2}\le 16C_p^{\frac{4}{4+p/2}}\sum_{i,j=1}^{m}\frac{n \mbb{Q}\bigg(f_{1,1,0,\beta}^{(i,j,i,i)}\bigg)^{2}}{(\varrho_{n}^{(i)}\varrho_{n}^{(j)})^{2}}.
\end{align}
As for $\gamma_{3}$, note that by letting $q=3$ in \eqref{holder1}, we have for $i\in[m]$,
\begin{align*}
    \mbb{E}|D_{\check{x}}F_{n}^{(i)}|^{3}\le (\mbb{E}|D_{\check{x}}F_{n}^{(i)}|^{4+p/2})^{\frac{3}{4+p/2}}\mbb{P}(D_{\check{x}}F_{n}^{(i)}\neq 0)^{\frac{1+p/2}{4+p/2}}.
\end{align*}

Arguing similarly as for \eqref{640}, by Lemma \ref{lemma4} we have
\begin{align*}
    \mbb{E}|D_{\check{x}}F_{n}^{(i)}|^{3}&\le (\mbb{E}|D_{\check{x}}F_{n}^{(j)}|^{4+p/2})^{\frac{3}{4+p/2}}\mbb{P}(D_{\check{x}}F_{n}^{(i)}\neq 0)^{\frac{2+p}{8+p}}\nonumber\\&\le C_{p}^{\frac{3}{4+p/2}}G_{n}^{(i)}(x)^3\mbb{P}(D_{\check{x}}F_{n}^{(i)}\neq 0)^{\frac{p}{8+p}}\\&\le C_{p}^{\frac{3}{4+p/2}}G_{n}^{(i)}(x)^3(\kappa_{n}^{(i)}(x)+g_{n}^{(i)}(x))^{4\beta}.
\end{align*}
Therefore, we obtain
\begin{align}\label{gamma3b}
    \gamma_{3}&\le C_{p}^{\frac{3}{4+p/2}}\sum_{i=1}^{m}(\varrho_{n}^{(i)})^{-3}n\int_{\mbb{X}}G_{n}^{(i)}(x)^3(\kappa_{n}^{(i)}(x)+g_{n}^{(i)}(x))^{4\beta}\mbb{Q}(dx)\nonumber\\&\le C_{p}^{\frac{3}{4+p/2}}\sum_{i=1}^{m}\frac{n\mbb{Q}\left((\kappa_{n}^{(i)}+g_{n}^{(i)})^{4\beta}\left(G_{n}^{(i)}\right)^3\right)}{(\varrho_n^{(i)})^{3}}.
\end{align}
Combining \eqref{gamma1b}, \eqref{gamma2b} and \eqref{gamma3b} and invoking Theorem \ref{2ndpoincare}, the result follows.
\end{proof}
    
\subsection{Proof of Theorem \ref{dconvex}} 
Since $\mathsf{d}_{\textsf{cvx}}$ is a distance with non-smooth test functions, it requires relatively stronger assumptions. We now take $p_0=6$ in \eqref{mc}. Again, we make use of Theorem \ref{2ndpoincare} with $H^{(i)}:=(\varrho_{n}^{(i)})^{-1}\bar{F}_{n}^{(i)}$ for $i\in[m]$. Note that the bound on $\mathsf{d}_{\textsf{cvx}}$ involves three additional $\gamma_{i}$, $i=4,5,6$ compared to the bound on $\mathsf{d}_2$ and $\mathsf{d}_3$. We will need a slightly modified version of Lemma \ref{lemma4} as stated below in Lemma \ref{lemma5}, whose proof follows from that of Lemma \ref{lemma4} by observing that $6+p/2=4+(2+p/2)$.

\begin{Lemma}\label{lemma5}
    Under Assumptions \ref{r1}-\ref{r4}, \ref{t} and \ref{m} with $p_0=6$, there exits a constant $C_p$ depending only on $p$ such that for all $i\in[m]$, $n\ge 1$, ${y}\in{\mbb{X}}$ and $\eta\in\tbf{N}$ with $\eta(\check{\mbb{X}})\le 1$, we have 
    \begin{align*}
        \mbb{E}|D_{\check{y}}F_{n}^{(i)}(\Pngm+\eta)|^{6+p/2}\le C_p\left(M_{n}^{(i)}(y)^{6+p/2}+h_{n}^{(i)}(y)(1+g_{n}^{(i)}(y)^6)\right) \le C_p G_n^{(i)}(y)^{6+p/2}.
    \end{align*}
\end{Lemma}

\begin{proof}[Proof of Theorem \ref{dconvex}] We start by bounding $\gamma_{i}$, $i=1,2,3$ in a similar way as in the proof of Theorem \ref{d2d3}. 
    By considering similar (changing $p_0=4$ to $p_0=6$) H\"{o}lder inequalities as for \eqref{holder1}, \eqref{holder2}, \eqref{holder3} and \eqref{holder4} in the proof of Theorem \ref{d2d3} and using Lemma \ref{lemma5}, we have for $i,j\in[m]$,
    \begin{align*}
        (\mathbb{E}(D_{\check{x}_{1}}F_n^{(j)})^{4})^{\frac{1}{4}}&\le (\mbb{E}|D_{\check{x}_1}F_{n}^{(j)}|^{6+p/2})^{\frac{1}{6+p/2}}\mbb{P}(D_{\check{x}_1}F_{n}^{(j)}\neq 0)^{\frac{4+p}{48+4p}}\nonumber\\
        &\le C_{p}^{\frac{1}{6+p/2}}G_{n}^{(j)}(x_1),
    \end{align*}
    \begin{align}\label{644}
        (\mathbb{E}(D^2_{\check{x}_{1},\check{x}_{3}}F_n^{(i)})^{4})^{\frac{1}{4}}&\le 2C_p^{\frac{1}{6+p/2}}(G_{n}^{(i)}(x_1)\wedge G_{n}^{(i)}(x_3)) \mbb{P}(D_{\check{x}_1,\check{x}_3}F_{n}^{(i)}\neq 0)^{\frac{4+p}{48+4p}}\nonumber\\&\le 2C_p^{\frac{1}{6+p/2}}(G_{n}^{(i)}(x_1)\wedge G_{n}^{(i)}(x_3))\left(e^{-\frac{4+p}{48+4p} r_{n}^{(i)}(x_3,x_1)}+e^{-\frac{4+p}{48+4p} r_{n}^{(i)}(x_1,x_3)}\right.\nonumber\\&\left.\qquad+q_{n}^{(i)}(x_1,x_3)^{\frac{4+p}{48+4p}}\right),
    \end{align}
    and
    \begin{align*}
        \mbb{E}|D_{\check{x}}F_{n}^{(i)}|^{3}&\le \mbb{E}|D_{\check{x}}F_{n}^{(j)}|^{6+p/2})^{\frac{3}{6+p/2}}\mbb{P}(D_{\check{x}}F_{n}^{(i)}\neq 0)^{\frac{6+p}{12+p}}\nonumber\\&\le C_{p}^{\frac{3}{6+p/2}}G_{n}^{(i)}(x)^3\mbb{P}(D_{\check{x}}F_{n}^{(i)}\neq 0)^{6\beta}\\&\le C_{p}^{\frac{3}{6+p/2}}G_{n}^{(i)}(x)^3(\kappa_{n}^{(i)}(x)+g_{n}^{(i)}(x))^{6\beta}.
    \end{align*}
    Combining all the bounds above and recalling~\eqref{eq:firstfwithmultipleindex}, we obtain
    \begin{align*}
        \gamma_{1}^2,\gamma_{2}^2\le 16C_p^{\frac{4}{6+p/2}}\sum_{i,j=1}^{m}\frac{n\mbb{Q}\bigg(f_{1,1,0,\beta}^{(i,j,i,i)}\bigg)^{2}}{(\varrho_{n}^{(i)}\varrho_{n}^{(j)})^{2}},
    \end{align*}
    and
    \begin{align*}
        \gamma_{3}\le C_{p}^{\frac{3}{6+p/2}}\sum_{i=1}^{m}\frac{n\mbb{Q}\left((\kappa_{n}^{(i)}+g_{n}^{(i)})^{6\beta}\left(G_{n}^{(i)}\right)^{3}\right)}{(\varrho_n^{(i)})^{3}}.
    \end{align*}
    
    Next, we proceed to bounding the remaining terms $\gamma_{i}$, $i=4,5,6$. First, we focus on $\gamma_{4}$. Appplying H\"{o}lder's inequality as before and using Lemma \ref{lemma5}, we have
    \begin{align}\label{643}
        \mathbb{E}(D_{\check{x}}F_n^{(i)})^{4}&\le (\mbb{E}|D_{\check{x}}F_{n}^{(i)}|^{6+p/2})^{\frac{4}{6+p/2}}\mbb{P}(D_{\check{x}}F_{n}^{(i)}\neq 0)^{\frac{4+p}{12+p}}\nonumber\\&\le C_{p}^{\frac{4}{6+p/2}}G_{n}^{(i)}(x)^4 \mbb{P}(D_{\check{x}}F_{n}^{(i)}\neq 0)^{6\beta}\nonumber\\&\le C_{p}^{\frac{4}{6+p/2}}G_{n}^{(i)}(x)^4(\kappa_{n}^{(i)}(y)+g_{n}^{(i)}(x))^{6\beta}.
    \end{align}
   Thus, by \eqref{644} and \eqref{643}, we obtain
    \begin{align*}
        \gamma_{4}^2 \lesssim C_{p}^{\frac{4}{6+p/2}}m\sum_{i=1}^{m}\frac{n\mbb{Q}\left((\kappa_{n}^{(i)}+g_{n}^{(i)})^{6\beta}\left(G_{n}^{(i)}\right)^{4}\right)}{(\varrho_n^{(i)})^{4}}+ C_{p}^{\frac{4}{6+p/2}}\sum_{i,j=1}^{m}\frac{n\mbb{Q}f_{2,2,0,3\beta}^{(i,j,l,i)}}{(\varrho_{n}^{(i)}\varrho_{n}^{(j)})^{2}}.
    \end{align*}
    Next, we turn to $\gamma_{5}$ and $\gamma_{6}$. By Lemmas \ref{eq:basicanalysis} and \ref{lemma5}, we have
    \begin{align}\label{vector1}
    \mbb{E}\|D_{\check{x}}\tbf{F}_n\|^{6+p/2}&\le m^{\frac{4+p/2}{2}}\sum_{i=1}^{m}\mbb{E}|D_{\check{x}}F_{n}^{(i)}|^{6+p/2}\nonumber\\&\le m^{\frac{4+p/2}{2}}C_p\sum_{i=1}^{m}G_{n}^{(i)}(x)^{6+p/2},
    \end{align}
    and
    \begin{align}\label{vector2}
        \mbb{E}\|D^2_{\check{x}_1,\check{x}_3}\tbf{F}_n\|^{6+p/2}&\le m^{\frac{4+p/2}{2}}\sum_{i=1}^{m}\mbb{E}|D^2_{\check{x}_1,\check{x}_3}F_{n}^{(i)}|^{6+p/2}\nonumber\\&\le m^{\frac{4+p/2}{2}}2^{6+p/2}C_p\sum_{i=1}^{m} \bigg( G_{n}^{(i)}(x_1) \wedge G_{n}^{(i)}(x_3)\bigg)^{6+p/2}.
    \end{align}
    Then, by the H\"{o}lder inequality (noting that $\frac{p}{12+p}+2\times \frac{3}{24+2p}+2\times \frac{3}{12+p}+\frac{3}{12+p}=1$, the last power being for the factor one), we have
    \begin{align*}
       &\mbb{E}\mathds{1}_{D^2_{\check{x}_{1},\check{x}_{3}}\mbf{F}_n\neq \mbf{0},D^2_{\check{x}_{2},\check{x}_{3}}\mbf{F}_n\neq \mbf{0}}(\|D_{\check{x}_{1}}\mbf{F}_n\|+\|D^2_{\check{x}_{1},\check{x}_{3}}\mbf{F}_n\|)^{\frac{3}{4}}\\&\quad\times(\|D_{\check{x}_{2}}\mbf{F}_n\|+\|D^2_{\check{x}_{2},\check{x}_{3}}\mbf{F}_n\|)^{\frac{3}{4}}|D_{\check{x}_{1}}F_n^{(i)}|^{\frac{3}{2}}|D_{\check{x}_{2}}F_n^{(i)}|^{\frac{3}{2}}\\&\le \mbb{E}\mathds{1}_{D^2_{\check{x}_{1},\check{x}_{3}}\mbf{F}_n\neq \mbf{0},D^2_{\check{x}_{2},\check{x}_{3}}\mbf{F}_n\neq \mbf{0}}(\|D_{\check{x}_{1}}\mbf{F}_n\|^{\frac{3}{4}}+\|D^2_{\check{x}_{1},\check{x}_{3}}\mbf{F}_n\|^{\frac{3}{4}})\\&\quad\times(\|D_{\check{x}_{2}}\mbf{F}_n\|^{\frac{3}{4}}+\|D^2_{\check{x}_{2},\check{x}_{3}}\mbf{F}_n\|^{\frac{3}{4}})|D_{\check{x}_{1}}F_n^{(i)}|^{\frac{3}{2}}|D_{\check{x}_{2}}F_n^{(i)}|^{\frac{3}{2}}\\&\lesssim \mbb{P}(D^2_{\check{x}_{1},\check{x}_{3}}\mbf{F}_n\neq \mbf{0},D^2_{\check{x}_{2},\check{x}_{3}}\mbf{F}_n\neq \mbf{0})^{\frac{p}{12+p}}\\&\qquad\times \left((\mbb{E}\|D_{\check{x}_1}\mbf{F}_n\|^{6+p/2})^{\frac{3}{4(6+p/2)}}+(\mbb{E}\|D^2_{\check{x}_1,\check{x}_3}\mbf{F}_n\|^{6+p/2})^{\frac{3}{4(6+p/2)}}\right)\\&\qquad\times \left((\mbb{E}\|D_{\check{x}_2}\mbf{F}_n\|^{6+p/2})^{\frac{3}{4(6+p/2)}}+(\mbb{E}\|D^2_{\check{x}_2,\check{x}_3}\mbf{F}_n\|^{6+p/2})^{\frac{3}{4(6+p/2)}}\right)\\&\qquad\times(\mbb{E}|D_{\check{x}_1}F_{n}^{(i)}|^{6+p/2})^{\frac{3}{2(6+p/2)}}(\mbb{E}|D_{\check{x}_2}F_{n}^{(i)}|^{6+p/2})^{\frac{3}{2(6+p/2)}}.
    \end{align*}
    
    Plugging in \eqref{vector1} and \eqref{vector2} and by Lemma \ref{lemma5}, we have
    \begin{align}\label{gamma511}
        &\mbb{E}\mathds{1}_{D^2_{\check{x}_{1},\check{x}_{3}}\mbf{F}_n\neq \mbf{0},D^2_{\check{x}_{2},\check{x}_{3}}\mbf{F}_n\neq \mbf{0}}(\|D_{\check{x}_{1}}\mbf{F}_n\|+\|D^2_{\check{x}_{1},\check{x}_{3}}\mbf{F}_n\|)^{\frac{3}{4}}\nonumber\\&\quad\times(\|D_{\check{x}_{2}}\mbf{F}_n\|+\|D^2_{\check{x}_{2},\check{x}_{3}}\mbf{F}_n\|)^{\frac{3}{4}}|D_{\check{x}_{1}}F_n^{(i)}|^{\frac{3}{2}}|D_{\check{x}_{2}}F_n^{(i)}|^{\frac{3}{2}}\nonumber\\&\lesssim \mbb{P}(D^2_{\check{x}_{1},\check{x}_{3}}\mbf{F}_n\neq \mbf{0},D^2_{\check{x}_{2},\check{x}_{3}}\mbf{F}_n\neq \mbf{0})^{\frac{p}{12+p}}\nonumber\\&\qquad\times m^{\frac{3}{4}} C_p^{\frac{3}{2(6+p/2)}}\sum_{l=1}^{m}G_{n}^{(l)}(x_1)^{\frac{3}{4}} \sum_{l=1}^{m}G_{n}^{(l)}(x_2)^{\frac{3}{4}}\nonumber\\&\qquad\times C_{p}^{\frac{3}{(6+p/2)}} G_n^{(i)}(x_1)^{\frac{3}{2}}G_n^{(i)}(x_2)^{\frac{3}{2}}
        \nonumber\\&\le m^{\frac{3}{4}} C_p^{\frac{9}{2(6+p/2)}}\mbb{P}(D^2_{\check{x}_{1},\check{x}_{3}}\mbf{F}_n\neq \mbf{0},D^2_{\check{x}_{2},\check{x}_{3}}\mbf{F}_n\neq \mbf{0})^{\frac{p}{12+p}} G_{n}^{(i)}(x_1)^{\frac{3}{2}}G_{n}^{(i)}(x_2)^{\frac{3}{2}}\nonumber\\&\qquad \qquad\qquad\times \sum_{l=1}^{m}G_{n}^{(l)}(x_1)^{\frac{3}{4}} \sum_{l=1}^{m}G_{n}^{(l)}(x_2)^{\frac{3}{4}}.
    \end{align}
    Also, we have by Lemma \ref{lemma5},
    \begin{align}\label{gamma512}
        (\mbb{E}|D_{\check{x}_1}F_{n}^{(j)}|^{3}|D_{\check{x}_2}F_{n}^{(j)}|^{3})^{\frac{1}{3}}&\le (\mbb{E}|D_{\check{x}_1}F_{n}^{(j)}|^{6+p/2})^{\frac{1}{6+p/2}}(\mbb{E}|D_{\check{x}_2}F_{n}^{(j)}|^{6+p/2})^{\frac{1}{6+p/2}}\nonumber\\
        &\le C_{p}^{\frac{2}{6+p/2}}G_{n}^{(j)}(x_1)G_{n}^{(j)}(x_2).
    \end{align}
    Therefore, by writing $\gamma_{5}^{3}:=\gamma_{5.1}+\gamma_{5.2}$ for the first term and the second term in $\gamma_{5}$, we have by \eqref{gamma511} and \eqref{gamma512},
    \begin{align}\label{gamma51}
        \gamma_{5.1}& \lesssim \sqrt{m}C_{p}^{\frac{5}{6+p/2}}\sum_{i,j,l=1}^{m}\Big(\sum_{s=1}^{m}(\varrho_{n}^{(s)})^{-1}\Big)(\varrho_{n}^{(i)})^{-2}(\varrho_{n}^{(j)})^{-2}\nonumber\\&\qquad \times n\int_{\mbb{X}}\bigg(n\int_{\mbb{X}}G_{n}^{(i)}(x_1)G_{n}^{(j)}(x_1)G_{n}^{(l)}(x_1)^{\frac{1}{2}} \mbb{P}(D^2_{\check{x}_{1},\check{x}_{3}}\mbf{F}_n\neq \mbf{0})^{\frac{p}{36+3p}}\mbb{Q}(dx_1)\bigg)^{2}\mbb{Q}(dx_3),
    \end{align}
    where we have pulled the sum out of the integral, and used the fact that
    \begin{align*}
        \|D_{\check{x}}((\mathrm{P}^{-1}_n\tbf{F}_{n})\|&=\Big(\sum_{i=1}^{m}(\varrho_{n}^{(i)})^{-2}(D_{\check{x}}F_{n}^{(i)})^2\Big)^{\frac{1}{2}}\le \bigg(\Big(\sum_{i=1}^{m}(\varrho_{n}^{(i)})^{-4}\Big)^{\frac{1}{2}}\Big(\sum_{i=1}^{m}(D_{\check{x}}F_{n}^{(i)})^{4}\Big)^{\frac{1}{2}}\bigg)^{\frac{1}{2}}\\& \qquad \le \Big(\sum_{i=1}^{m}(\varrho_{n}^{(i)})^{-2}\sum_{i=1}^{m}(D_{\check{x}}F_{n}^{(i)})^{2}\Big)^{\frac{1}{2}} \le \Big(\sum_{i=1}^{m}(\varrho_{n}^{(i)})^{-1}\Big)\|D_{\check{x}}\mbf{F}_{n}\|,
    \end{align*} 
    which also applies to $D^2_{\check{x}_{1},\check{x}_{3}}\mbf{F}_n$. Moreover, note that
    \begin{align}\label{Pvector}
        \mbb{P}(D^2_{\check{x}_{1},\check{x}_{3}}\mbf{F}_n\neq \mbf{0})^{\frac{p}{36+3p}}\le \sum_{i=1}^{m}\mbb{P}(D_{\check{x}_1,\check{x}_3}F_{n}^{(i)}\neq 0)^{\frac{p}{36+3p}}.
    \end{align}
    Consequently, according to \eqref{gamma51} and \eqref{Pvector}, we have
    \begin{align*}
        \gamma_{5.1}&\lesssim m^{\frac{3}{2}} C_{p}^{\frac{5}{6+p/2}}\sum_{i,j,l,t=1}^{m}\Big(\sum_{s=1}^{m}(\varrho_{n}^{(s)})^{-1}\Big)(\varrho_{n}^{(i)})^{-2}(\varrho_{n}^{(j)})^{-2} \nonumber\\&\qquad \times n\int_{\mbb{X}}\Big(n\int_{\mbb{X}}G_{n}^{(i)}(x_1)G_{n}^{(j)}(x_1)G_{n}^{(l)}(x_1)^{\frac{1}{2}}\mbb{P}(D^2_{\check{x}_{1},\check{x}_{3}}F_n^{(t)}\neq 0)^{\frac{p}{36+3p}}\mbb{Q}(dx_1)\Big)^{2}\mbb{Q}(dx_3),
    \end{align*}
    where we have pulled the sum (over the probabilities) out of the integral resulting in the extra factor $m$ by Lemma \ref{eq:basicanalysis}. Similar arguments also yield that
    \begin{align*}
        \gamma_{5.2}&\lesssim m^{\frac{3}{2}} C_{p}^{\frac{5}{6+p/2}}\sum_{i,j,l,t=1}^{m}\Big(\sum_{s=1}^{m}(\varrho_{n}^{(s)})^{-1}\Big)(\varrho_{n}^{(i)})^{-2}(\varrho_{n}^{(j)})^{-2}\nonumber\\&\qquad \times n\int_{\mbb{X}}\Big(n\int_{\mbb{X}}G_{n}^{(i)}(x_1)G_{n}^{(j)}(x_1)G_{n}^{(l)}(x_1)^{\frac{1}{2}}\mbb{P}(D^2_{\check{x}_{1},\check{x}_{3}}F_n^{(t)}\neq 0)^{\frac{p}{72+6p}}\mbb{Q}(dx_1)\Big)^{2}\mbb{Q}(dx_3).
    \end{align*}
    Combining the bounds on $\gamma_{51}$ and $\gamma_{52}$, we obtain
    \begin{align*}
        \gamma_{5}^{3}&\lesssim m^{\frac{3}{2}}C_{p}^{\frac{5}{6+p/2}}\sum_{i,j,l,t=1}^{m}\Big(\sum_{s=1}^{m}(\varrho_{n}^{(s)})^{-1}\Big)(\varrho_{n}^{(i)})^{-2}(\varrho_{n}^{(j)})^{-2}\nonumber\\&\qquad \times n\int_{\mbb{X}}\bigg(n\int_{\mbb{X}}G_{n}^{(i)}(x_1)G_{n}^{(j)}(x_1)G_{n}^{(l)}(x_1)^{\frac{1}{2}}\mbb{P}(D^2_{\check{x}_{1},\check{x}_{3}}F_n^{(t)}\neq 0)^{\frac{p}{72+6p}}\mbb{Q}(dx_1)\bigg)^{2}\mbb{Q}(dx_3)\\&\lesssim m^{\frac{3}{2}}C_{p}^{\frac{5}{6+p/2}}\sum_{i,j,l,t=1}^{m}\sum_{s=1}^{m}\frac{n\mbb{Q}\left(f_{1,1,1/2,\beta}^{(i,j,l,t)}\right)^{2}}{\varrho_{n}^{(s)}(\varrho_{n}^{(i)}\varrho_{n}^{(j)})^{2}}.
    \end{align*}
    Similar to $\gamma_{5}$, one can derive 
    \begin{align*}
        \gamma_{6}^4& \lesssim m^2 C_{p}^{\frac{6}{6+p/2}}\sum_{i,j,l,t=1}^{m}\Big(\sum_{s=1}^{m}(\varrho_{n}^{(s)})^{-1}\Big)^2(\varrho_{n}^{(i)})^{-2}(\varrho_{n}^{(j)})^{-2}\nonumber\\&\qquad \times n\int_{\mbb{X}}\bigg(n\int_{\mbb{X}}G_{n}^{(i)}(x_1)G_{n}^{(j)}(x_1)G_{n}^{(l)}(x_1)\mbb{P}(D^2_{\check{x}_{1},\check{x}_{3}}F_n^{(t)}\neq 0)^{\frac{p}{72+6p}}\mbb{Q}(dx_1)\bigg)^{2}\mbb{Q}(dx_3)\\&\lesssim m^2 C_{p}^{\frac{6}{6+p/2}}\sum_{i,j,l,t=1}^{m}\Big(\sum_{s=1}^{m}(\varrho_{n}^{(s)})^{-1}\Big)^2(\varrho_{n}^{(i)})^{-2}(\varrho_{n}^{(j)})^{-2}n\mbb{Q}\left(f_{1,1,1,\beta}^{(i,j,l,t)}\right)^{2}\\&\lesssim m^3 C_{p}^{\frac{6}{6+p/2}}\sum_{i,j,l,t=1}^{m}\sum_{s=1}^{m}\frac{n\mbb{Q}\left(f_{1,1,1,\beta}^{(i,j,l,t)}\right)^{2}}{(\varrho_{n}^{(s)}\varrho_{n}^{(i)}\varrho_{n}^{(j)})^{2}},
    \end{align*}
where in the last step, we use the Lemma \ref{eq:basicanalysis}. Putting together all the bounds on $\gamma_1$ to $\gamma_6$ above yields the desired conclusion.
\end{proof}

\section{Proofs of results in Section \ref{sec:Main}}\label{secofmainpf}
\textbf{Additional Notation:} For $a\in\mbb{R}$, $\mbf{b}:=(b^{(1)},\ldots,b^{(d)})\in\mbb{R}^{d}$ and $\psi:\mbb{R}\rightarrow \mbb{R}$, we write $a\mbf{b}:=(ab^{(1)},\ldots,ab^{(d)})$, $a+\mbf{b}:=(a+b^{(1)},\ldots,a+b^{(d)})$ and $\psi(\mbf{b}):=(\psi(b^{(1)}),\ldots,\psi(b^{(d)}))$. Let $I$ be a subset of $[d]$ and denote by $x^{I}:=(x^{(i)})_{i\in I}$ the subvector indexed by the set $I$. We also write $I^{c}=[d]\backslash I$.

Given a target point $x_0\in\mbb{R}^{d}$, recall from \eqref{rnk} the random forest estimator at $x_0$ associated to $k$-PNNs for some $k\ge 1$, which is given by
\begin{align*}
    r_{n,k,w}(x_0)=\sum_{(\bm{x},\bm{\varepsilon}_{\bm{x}})\in \Pngm}W_{n\bm{x}}(x_0)\mathds{1}_{\bm{x}\in\mcal{L}_{n,k}(x_{0})}\bm{y}_{\bm{x}},
\end{align*}
Particularly, for the case \eqref{rnkunif} with uniform weights, we have,
\begin{align*}
    r_{n,k}(x_0)=\sum_{(\bm{x},\bm{\varepsilon}_{\bm{x}})\in \Pngm}\frac{\mathds{1}_{\bm{x}\in\mcal{L}_{n,k}(x_{0})}}{L_{n,k}(x_0)}\bm{y}_{\bm{x}},
\end{align*}
where $\bm{y}_{\bm{x}}:=r(\bm{x},\bm{\varepsilon}_{\bm{x}})$. As mentioned briefly in Section \ref{generalregion}, $r_{n,k,w}(x_0),r_{n,k}(x_0)$ can be viewed as sums of score functions of the marked Poisson process $\Pngm$ with intensity measure $n\mbb{Q}\otimes P_{\bm{\varepsilon}}$, where $\mbb{Q}$ has an a.e.\ continuous density $g$ on $\mbb{R}^{d}$. Given a target point $x_{0}\in\mbb{R}^{d}$, we can consider the score function $\xi_{n}$ associated to $r_{n,k,w}(x_0)$ given by 
\begin{align*}
    \xi_{n}(\check{x},\eta)=W_{nx}(x_0,\eta)\mathds{1}_{x\in\mcal{L}_{n,k}(x_{0},\eta)}\, y_x, \quad 0 \neq \eta \in \Nb, \check{x} \in \eta,
\end{align*}
while in the special case of uniform weights in $r_{n,k}(x_0)$, the score function becomes
\begin{align*}
    \xi_{n}(\check{x},\eta)=\frac{\mathds{1}_{x\in\mcal{L}_{n,k}(x_{0},\eta)}}{L_{n,k}(x_{0},\eta)}\, y_x, \quad 0 \neq \eta \in \Nb, \check{x} \in \eta.
\end{align*}
By the definition of $k$-PNNs, it is straightforward to see that the score functions above are region-stabilizing in the sense of \eqref{R1} with the region of stabilization given by
\begin{align}\label{regionrf}
    R_{n}(\check{x},\eta):=
    \left\{
    \begin{aligned}
        &\rec(x_{0},x)\times \mbb{R},\quad \text{if}\ \eta((\rec(x_{0},x)\backslash \{x\})\times\mbb{R})<k,\\
        &\emptyset,\quad \text{otherwise}
    \end{aligned}
    \right.
\end{align}
for $\eta \in \Nb$ and $\check{x} \in \eta$. Therefore, we aim to apply the theorems in Section \ref{generalregion}. Throughout this section, we shall omit $\lambda_d$ in integrals and simply write $dx$ instead of $\lambda_d(dx)$. It is straightforward to check that Assumptions \eqref{R0} and \ref{r1}-\ref{r4} are satisfied for this score function and the region \eqref{regionrf}. For $\eta \in \Nb$ with $\eta(\R^d \times \R) \le 9$ (since $p_0=6$ here), and $x \in \R^d$, recall that $\mathcal{P}_{x,\eta}:=\Pngm+\delta_{(x,\bm{\varepsilon}_x)}+\eta$.
By independence, for any $p >0$ we have
\begin{multline*}
    \left\|W_{nx}(x_0,\mathcal{P}_{x,\eta})\mathds{1}_{x\in\mcal{L}_{n,k}(x_{0},\mathcal{P}_{x,\eta})}y_{x}\right\|_{L_{6+p}}\\
    \le \left(\mbb{E}_{P_{\bm{\varepsilon}}}|r(\bm{\varepsilon}_{x},x)|^{6+p}\right)^{1/(6+p)}\sup_{(x,\eta): |\eta| \le 9} \left\|W_{nx}(x_0,\mathcal{P}_{x,\eta})\right\|_{L_{6+p}},
\end{multline*}
and
\begin{align}\label{eq:handlingadaptiveweights}
 \bigg\|\frac{\mathds{1}_{x\in\mcal{L}_{n,k}(x_{0},\mathcal{P}_{x,\eta})}}{L_{n,k}(x_{0},\mathcal{P}_{x,\eta})}y_{x}\bigg\|_{L_{6+p}}\le \left(\mbb{E}_{P_{\bm{\varepsilon}}}|r(\bm{\varepsilon}_{x},x)|^{6+p}\right)^{\frac{1}{(6+p)}}\sup_{(x,\eta): |\eta| \le 9} \left\|\frac{1}{L_{n,k}(x_{0},\mathcal{P}_{x,\eta})}\right\|_{L_{6+p}}.
\end{align}
Since $L_{n,k}(x_{0},\mathcal{P}_{x,\eta}) \ge 1$ for all $x \neq x_0$, Assumption \ref{m} is satisfied for $x \neq x_0$ with $p_0 =6$ and 
\begin{align}\label{omegar}
    M_{n}(x):=\Omega_{n}r_{6+p}^{*}(x),
\end{align}
where 
\begin{equation*}
\Omega_{n}:=\left\{
\begin{aligned}
&\sup_{(x,\eta): |\eta| \le 9}\left\|W_{nx}(x_0,\mathcal{P}_{x,\eta})\right\|_{L_{6+p}}\qquad \text{for $r_{n,k,w}(x_0)$},\\
&\sup_{(x,\eta): |\eta| \le 9} \|L_{n,k}(x_{0},\mathcal{P}_{x,\eta})^{-1}\|_{L_{6+p}}\quad \text{ for $r_{n,k}(x_0)$},
\end{aligned}
\right.
\end{equation*}
which does not depend on $x$, and $r_{6+p}^{*}(x):=\left(\mbb{E}_{P_{\bm{\varepsilon}}}|r(\epsilon_{x},x)|^{6+p}\right)^{1/(6+p)}$. Also, from \eqref{regionrf}, we have for $\check{x},\check{y}\in\mbb{R}^{d}\times\mbb{R}$,
\begin{align}\label{regionofstab}
    \mbb{P}(\check{y}\in R_{n}(\check{x},\Pngm+\delta_{\check{x}}))&=\sum_{j=0}^{k-1}\mathds{1}_{y\in \rec(x_{0},x)}e^{-n\int_{\rec(x,x_{0})}g(z)dz}\frac{\left(n\int_{\rec(x,x_{0})}g(z)dz\right)^{j}}{j!}\nonumber\\
&=\mathds{1}_{y\in\rec(x_0,x)}\psi(n,k,x_0,x),
\end{align}
where for notational convenience, for $n,k\ge 1$ and $x_{0},x\in\mbb{R}^{d}$, we define
\begin{align}\label{Poicdf}
    \psi(n,k,x_{0},x)&:=\mbb{P}\left(\textup{Poi}\left(n\int_{\rec(x_{0},x)}g(z)dz\right)<k\right)\nonumber\\&=\sum_{j=0}^{k-1}e^{-n\int_{\rec(x_{0},x)}g(z)dz}\frac{\left(n\int_{\rec(x,x_{0})}g(z)dz\right)^{j}}{j!}.
\end{align}

Noting that for any $t> 0$ and $0\le j\le k-1$, $e^{-t}t^{j}\le j!$, 
so that $e^{-(1-(j+2)^{-1})t}((1-(j+2)^{-1})t)^{j}\le j!$, we obtain
\begin{align*}
    e^{-n\int_{\rec(x,x_{0})}g(z)dz}\frac{\left(n\int_{\rec(x,x_{0})}g(z)dz\right)^{j}}{j!}\le (1-(j+2)^{-1})^{-j}e^{-\frac{1}{j+2}n\int_{\rec(x,x_{0})}g(z)dz}.
\end{align*}
Note that $(1-(j+2)^{-1})^{-j} \le e^{j/(j+2)} \le e$. Therefore, we can upper bound the probability in Assumption \ref{t} as
\begin{align}\label{decayn}
    \mbb{P}(\check{y}\in R_{n}(\check{x},\Pngm+\delta_{\check{x}})) \lesssim e\sum_{j=0}^{k-1}e^{-\frac{1}{j+2}n\int_{\rec(x,x_0)}g(z)dz}=:e^{-r_n(x,y)},
\end{align}
when $y\in \rec(x_{0},x)$, while we take $r_{n}(x,y)=\infty$ otherwise. 
Therefore, Assumption \ref{t} is also satisfied. It remains to estimate the quantities appearing in Theorems \ref{d2d3} and \ref{dconvex}.
    
To this end, we first introduce a function that will play a key role in the estimation. 
For an a.e. continuous function $\phi(x):\mbb{R}^{d}\rightarrow\mbb{R}_{+}$ with $\int_{\mbb{R}^{d}}\phi(x)g(x)dx<\infty$, $\alpha, s>0$, $d\ge 1$ and $x_{0}\in\mbb{R}^{d}$, define the function $c_{\alpha,s,x_0}:\mbb{R}^{d}\rightarrow \mbb{R}$ as 
    \begin{align}\label{calphan}
        c_{\alpha,s,x_{0}}(y):=s\int_{\mbb{R}^{d}}\mathds{1}_{y\in \rec(x_0,x)}e^{-\alpha s\int_{\rec(x,x_{0})}g(z)dz}\phi(x)g(x)dx,
    \end{align}
    where we suppress the dependence on $\phi$ for ease of notation. Observe that $c_{\alpha,s,x_{0}}$ has the following scaling property: for all $\alpha, s>0$,
    \begin{align}\label{cscale}
        c_{\alpha,s,x_0}(y)=\alpha^{-1}c_{1,\alpha s,x_0}(y).
    \end{align}
    Therefore, we will often take $\alpha=1$ without loss of generality. While a similar function was also studied by \citet[Section 3]{bhattacharjee2022gaussian} in the context of minimal points, it is important to emphasize here that we relax several assumptions made therein. For instance, we do not require a uniform density $g$ on a compact set $[0,1]^{d}$, and consider instead an a.e. continuous density $g$ on $\mbb{R}^{d}$. The way we deal with such a general density is to divide the integral over $\mbb{R}^{d}$ into one over a suitable compact hyperrectangle $A$ which we choose to be a neighborhood of the point $x_{0}$, and another integral over its complement $A^{c}$. For the integral over $A$, we can apply similar arguments as done by \citet[Section 3]{bhattacharjee2022gaussian}, since up to finitely many rotations, translations and scalings, any hyperrectangle in $\mbb{R}^{d}$ is ``equivalent'' to $[0,1]^{d}$. On $A^c$, we bound the integral over the coordinates that are within the neighborhood and those outside the neighborhood separately.
    
    For $\epsilon>0$, we write the $m$-dimensional vector $\bm{\epsilon}:=(\epsilon,\ldots,\epsilon)$. Note that we can choose a set $A\subseteq \R^d$ with $\mbb{Q}(A)=1$ such that for $x_0\in A$ we have $g(x_0)>0$ and $\phi(x),g(x)$ are continuous at $x_0$. For such an $x_0$ and $\delta\in(0,1/2\ g(x_0))$, by continuity, there exists $\epsilon\coloneqq\epsilon(x_0,\delta)>0$ such that for $x\in \rec(x_0-\bm{\epsilon},x_0-\bm{\epsilon})$,
\begin{equation}\label{eq:gpcont}
        |\phi(x)-\phi(x_0)|<\delta \quad \text{and} \quad |g(x)-g(x_0)|<\delta.
    \end{equation}
     Also recall that for $x \in \R^d$ and $j \in [d]$, we denote $x^{[j]}:=(x^{(1)},\ldots,x^{(j)})$. For $j \in \{0\}\cup[d]$ let $\mcal{C}_{\bm{\epsilon}}(x_{0}^{\underline{j}}):=x_0^{\underline{j}} + [-\epsilon,\epsilon]^j$, where $\overline{j}$ denotes a $j$-tuple with elements in $\{0\}\cup[d]$.
\begin{Lemma}\label{lemma6}
Let $x_0 \in A$ be as above and $\delta\in(0,1/2\ g(x_0))$. Then there exists $\epsilon>0$ such that for all $\alpha>0$ and $n,d,k\ge 1$, we have
\begin{align*}
    &c_{\alpha (k+1)^{-1},n,x_0}(y)\le ne^{-\alpha (k+1)^{-1} n (g(x_0)-\delta)\epsilon^{d}}\int_{\mbb{R}^{d}}\phi(x)g(x)dx\ 
    \nonumber\\&\qquad\qquad\qquad+\sum_{j=1}^{d}\frac{\Lambda_j D_{j}}{\alpha(g(x_0)-\delta)}\left(\frac{k+1}{\epsilon^{d-j}}\right) \sum_{\underline{j}} e^{-\alpha \left(\frac{n(g(x_0)-\delta)\epsilon^{d-j}}{k+1}\right)|y^{\underline{j}}-x_{0}^{\underline{j}}|/2}\nonumber\\&\qquad \times \left(\left|\log \left(\alpha \left(\frac{n(g(x_0)-\delta)\epsilon^{d-j}}{k+1}\right)|y^{\underline{j}}-x_{0}^{\underline{j}}|\right)\right|^{j-1}+1\right)\mathds{1}_{y^{\underline{j}} \in \mcal{C}_{\bm{\epsilon}}(x_{0}^{\underline{j}})},
\end{align*}
where $\sum_{\underline{j}}$ denotes the sum over all of $j$-tuples in $[d]$, and for $j\in[d]$, $D_{j}>0$ is a constant depending only on $j$, while $\Lambda_j\equiv \Lambda(j,\epsilon)>0$ is a constant depending on $x_0$, $j$, $\epsilon,\phi$ and $g$. In particular, we can take $\Lambda_d \coloneqq (\phi(x_0)+\delta)(g(x_0)+\delta)$.
\end{Lemma}

\begin{proof}
Fix $x_0 \in A$, $\delta\in(0,1/2\ g(x_0))$. By arguments similar to \citet[Proof of Theorem 2.2]{biau2010layered}, 
for $n>0$, 
\begin{align*}
    c_{1,n,x_0}(y)&=n\int_{\mbb{R}^{d}}\mathds{1}_{y\in \rec(x_0,x)}e^{-n\int_{\rec(x,x_0)}g(z)dz}\phi(x)g(x)dx\\&=n\int_{\mcal{C}_{\bm{\epsilon}}(x_{0})}\mathds{1}_{y\in \rec(x_0,x)}e^{-n\int_{\rec(x,x_0)}g(z)dz}\phi(x)g(x)dx\\&\qquad+n\int_{\mbb{R}^{d}\backslash \mcal{C}_{\bm{\epsilon}}(x_{0})}\mathds{1}_{y\in \rec(x_0,x)}e^{-n\int_{\rec(x,x_0)}g(z)dz}\phi(x)g(x)dx\\&=:c_1(y)+c_2(y),
\end{align*}
where $\mcal{C}_{\bm{\epsilon}}(x_{0}) = \mcal{C}_{\bm{\epsilon}}(x_0^{\underline{d}}) =  \rec(x_0 -\bm{\epsilon}, x_0 +\bm{\epsilon})$. We first bound $c_1(y)$. Fix $\epsilon \in (0,\epsilon(x_0,\delta))$, and denote $\Delta_d=(g(x_0)-\delta)^{\frac{1}{d}}$. Let $\Lambda_d \in (0,\infty)$, depending only on $x_0$, $j$, $\epsilon,\phi$ and $g$, be such that $\phi g$ is uniformly bounded by $\Lambda_d$ on $\mcal{C}_{\bm{\epsilon}}(x_{0})$. In particular, By \eqref{eq:gpcont} we see that we can take $\Lambda_d \le (\phi(x_0)+\delta)(g(x_0)+\delta)$. Thus we obtain
\begin{align*}
    c_1(y)&
    \le \Lambda_d n\int_{\mcal{C}_{\bm{\epsilon}}(x_{0})}\mathds{1}_{y\in \rec(x_0,x)}e^{-n(g(x_0)-\delta)|x-x_0|}dx\nonumber\\&=\Lambda_d n\int_{\rec(-\bm{\epsilon},\bm{\epsilon})}\mathds{1}_{y-x_0\in \rec(0,x)}e^{-n(g(x_0)-\delta)|x|}dx\nonumber\\&=\frac{\Lambda_d}{g(x_0)-\delta}n \int_{\rec(-\Delta_d\bm{\epsilon},\Delta_d\bm{\epsilon})}\mathds{1}_{\Delta_d(y-x_0)\in \rec(0,x)}e^{-n|x|}dx.
    \end{align*}
    First note that $c_1(y)=0$ for $y \notin \mcal{C}_{\bm{\epsilon}}(x_{0})$, since the indicator $\mathds{1}_{y\in \rec(x_0,x)}$ in the first step is then always zero. Also note that the indicator $\mathds{1}_{\Delta_d(y-x_0)\in \rec(0,x)}$ enforces that $x$ in the integral can only be in one of the $2^d$ orthants. Let $\operatorname{abs}(y):= (|y^i|)_{i \in [d]}$
    denote the vector of absolute values of the coordinates of $y \in \R^d$. Then, by symmetry we obtain
    \begin{align}\label{lowd}
    c_1(y) &\le \frac{\Lambda_d}{g(x_0)-\delta}n \int_{\mbf{0}\prec x\prec \Delta_d \bm{\epsilon}}\mathds{1}_{\operatorname{abs}(\Delta_d(y-x_0))\in \rec(0,x)}e^{-n|x|}dx\nonumber\\
    &\le \frac{(\phi(x_0)+\delta)(g(x_0)+\delta)}{g(x_0)-\delta}n \int_{\mbf{0}\prec x\prec \Delta_d \bm{\epsilon}}\mathds{1}_{\operatorname{abs}(\Delta_d(y-x_0))\in \rec(0,x)}e^{-n|x|}dx.
\end{align}
 Next, we note the following inequality which is due to \citet[Lemma 3.1]{bhattacharjee2022gaussian}: for $\alpha>0$ and $n\ge 1$, there exists a constant $D>0$ depending only on $d$ such that
\begin{align}\label{BM22}
    n\int_{[0,1]^{d}}\mathds{1}_{x\succ y}e^{-\alpha n|x|}dx\le \frac{D}{\alpha}e^{-\alpha n |y|/2}(1+|\log (\alpha n |y|)|^{d-1}).
\end{align}

Using the transformation
 $
 \tilde{x}=(\Delta_d \epsilon)^{-1}x
 $
in the first step and \eqref{BM22} in the second (replacing $n$ by $n|\Delta_{d}\bm{\epsilon}|$ and taking $y=\epsilon^{-1}(y-x_0)$), from \eqref{cscale} we obtain
\begin{align}\label{sublowd}
    &n\int_{\mbf{0}\prec x\prec \Delta_d \bm{\epsilon}}\mathds{1}_{\operatorname{abs}(\Delta_d(y-x_0))\in \rec(0,x)}e^{-n|x|}dx \nonumber \\&=n|\Delta_d\bm{\epsilon}|\int_{[0,1]^d}\mathds{1}_{0\prec \epsilon^{-1}\operatorname{abs}(y-x_0)\prec \tilde{x}}e^{-n|\Delta_d\bm{\epsilon}||\tilde{x}|}d\tilde{x}\nonumber\\&\le D e^{-n|\Delta_d\bm{\epsilon}| |\epsilon^{-1}(y-x_0)|/2}(1+|\log ( n|\Delta_d\bm{\epsilon}| |\epsilon^{-1}(y-x_0)|)|^{d-1})\nonumber\\&=D e^{-n|\Delta_d(y-x_0)|/2}\left(1+|\log (n|\Delta_d(y-x_0)|)|^{d-1}\right).
\end{align}

To bound $c_2$, we argue similar to \citet[Proof of Theorem 2.2]{biau2010layered}. Note that
\begin{align}\label{partC}
    \mbb{R}^{d}\backslash\mcal{C}_{\bm{\epsilon}}(x_{0})=\bigcup_{j=0}^{d-1}\mcal{C}_{j},
\end{align}
where, $\mcal{C}_{j}$, $j \in \{0\}\cup[d-1]$ denotes the collection of all $y\in\mbb{R}^{d}\backslash\mcal{C}_{\bm{\epsilon}}(x_{0})$ which have exactly $j$ of the $d$ coordinates within an $\epsilon$-neighborhood of the corresponding coordinates of $x_0$. By symmetry, for each $j\in\{0\}\cup[d-1]$,
\begin{align}\label{partCC}
    \mcal{C}_{j}=\bigcup_{\underline{j}}\mcal{C}_{\underline{j}},
\end{align}
where the index $\underline{j}$ runs over all $\binom{d}{j}$ possible $j$-tuples in $[d]$, and $\mcal{C}_{\underline{j}} \equiv \mcal{C}_{\underline{j}}^{x_0}$ denotes the collection of points for which the coordinates in $\underline{j}$ are within an $\epsilon$-neighborhood of those coordinates of $x_0$. Denote the function
$$
(\phi g)_{\underline{j}}(x^{\underline{j}}):=\int_{\R^{d-j}} \phi(x) g(x) dx^{[d] \setminus \underline{j}}.
$$ 
For $\epsilon=\epsilon(x_0,\delta)$, note that for each $j \in \{0\}\cup[d-1]$, there exists $\Lambda_j \in (0,\infty)$ depending only on $x_0$, $j$, $\epsilon,\phi$ and $g$ such that the functions $(\phi g)_{\underline{j}}$ are uniformly bounded over $C_{\bm{\epsilon}}(x_{0}^{\underline{j}})$ by $\Lambda_j$.

\noindent Considering the integral in $c_2$ over $\mcal{C}_{0}$, by \eqref{eq:gpcont} we have
\begin{align}\label{j0}
    &n\int_{\mcal{C}_{0}}\mathds{1}_{y\in \rec(x_0,x)}e^{- n\int_{\rec(x,x_0)}g(z)dz}\phi(x)g(x)dx\nonumber\\
    & \le n\int_{\mcal{C}_{0}}\mathds{1}_{y\in \rec(x_0,x)}e^{- n\int_{\mcal{C}_{\bm{\epsilon}}(x_{0})}g(z)dz}\phi(x)g(x)dx\nonumber\\
    & \le ne^{-n(g(x_0)-\delta)\epsilon^{d}}\int_{\mbb{R}^{d}}\phi(x)g(x)dx. 
\end{align}
Similarly, for $j\in[d-1]$, one may write
\begin{align*}
    &n\int_{\mcal{C}_{j}}\mathds{1}_{y\in \rec(x_0,x)}e^{- n\int_{\rec(x,x_0)}g(z)dz}\phi(x)g(x)dx\\&=s\sum_{\underline{j}}\int_{\mcal{C}_{\underline{j}}}\mathds{1}_{y\in \rec(x_0,x)}e^{- n\int_{\rec(x,x_0)}g(z)dz}\phi(x)g(x)dx\\&\le n\sum_{\underline{j}}\int_{\mcal{C}_{\underline{j}}}\mathds{1}_{y\in \rec(x_0,x)}e^{-n(g(x_0)-\delta)\epsilon^{d-j}\prod_{l \in \underline{j}}|x^{(l)}-x_0^{(l)}|}\phi(x)g(x)dx\\
    &\le n\sum_{\underline{j}}\int_{\mcal{C}_{\bm{\epsilon}}(x_{0}^{\underline{j}})}\mathds{1}_{y^{\underline{j}}\in \rec(x_{0}^{\underline{j}},x^{\underline{j}})}e^{-n(g(x_0)-\delta)\epsilon^{d-j}\prod_{l \in \underline{j}}|x^{(l)}-x_0^{(l)}|}(\phi g)_{\underline{j}}(x^{\underline{j}})dx^{\underline{j}},
\end{align*}
where in the final step, we have integrated $\phi(x)g(x)$ over the coordinates $[d]\backslash \underline{j}$. We again note that the integral inside the sum is zero when $y^{\underline{j}} \in \mcal{C}_{\bm{\epsilon}}(x_{0}^{\underline{j}})^c$. Thus, the sum is zero when the number of coordinates $i \in [d]$ where $y^{(i)} \in [x_0^{(i)} -\epsilon,x_0^{(i)} +\epsilon]$ is less than $j$.
Since each $(\phi g)_{\underline{j}}$ is uniformly bounded by $\Lambda_j$ over $\mcal{C}_{\bm{\epsilon}}(x_{0}^{\underline{j}})$, 
letting $\Delta_j:=(g(x_0)-\delta)^{\frac{1}{j}}$, arguing sumilarly as for $c_1$, we obtain
\begin{align}\label{highd}
    &n\int_{\mcal{C}_{\underline{j}}}\mathds{1}_{y\in \rec(x_0,x)}e^{-n(g(x_0)-\delta)\epsilon^{d-j}\prod_{l \in \underline{j}}|x^{(l)}-x_0^{(l)}|}\phi(x)g(x)dx\nonumber\\&\le
    \Lambda_j n\int_{\mcal{C}_{\bm{\epsilon}}(x_{0}^{\underline{j}})}\mathds{1}_{y^{\underline{j}}\in \rec(x_{0}^{\underline{j}},x^{\underline{j}})}e^{-n(g(x_0)-\delta)\epsilon^{d-j}\prod_{l \in \underline{j}}|x^{(l)}-x_0^{(l)}|}dx^{\underline{j}}\nonumber\\&=\Lambda_j n\int_{[0,\epsilon]^j}\mathds{1}_{\operatorname{abs}(y^{\underline{j}}-x_{0}^{\underline{j}})\in \rec(0,x^{\underline{j}})}e^{-n(g(x_0)-\delta)\epsilon^{d-j}\prod_{l \in \underline{j}}|x^{(l)}|}dx^{\underline{j}}\nonumber\\&= \frac{\Lambda_j}{(g(x_0)-\delta)\epsilon^{d-j}}n \int_{[0,\Delta_j \epsilon^{d/j}]^j}\mathds{1}_{\operatorname{abs}(y^{\underline{j}}-x_{0}^{\underline{j}}) \in \rec(0,\Delta_{j}^{-1}\epsilon^{-(d-j)/j}x^{\underline{j}})}e^{-n|x^{\underline{j}}|}dx^{\underline{j}}.
\end{align}
Note that the integral in the last step in \eqref{highd}, when we take $j=d$, is exactly the same as the integral in \eqref{lowd}. Therefore, a similar argument as used in bounding $c_1$ can be applied here, and we obtain for $j\in[d-1]$,
\begin{align}\label{jge0}
    &n\int_{\mcal{C}_{j}}\mathds{1}_{y\in \rec(x_0,x)}e^{- n\int_{\rec(x,x_0)}g(z)dz}\phi(x)g(x)dx\nonumber\\&\le \frac{\Lambda_j D_{j}}{(g(x_0)-\delta)\epsilon^{d-j}} \sum_{\underline{j}} e^{-n|\Delta_j\epsilon^{(d-j)/j}(y^{\underline{j}}-x_{0}^{\underline{j}})|/2}\left(|\log (n|\Delta_j\epsilon^{(d-j)/j}(y^{\underline{j}}-x_{0}^{\underline{j}})|)|^{j-1}+1\right),
\end{align}
where the constant $D_{j} \in (0,\infty)$ depends only on $j$ by \eqref{BM22}. 
Now, combining $\eqref{j0}$ and \eqref{jge0}, we have
\begin{align}\label{c2y}
    c_2(y)&\le n e^{-n(g(x_0)-\delta)\epsilon^{d}}\int_{\mbb{R}^{d}}\phi(x)g(x)dx+\sum_{j=1}^{d-1}\frac{\Lambda_j D_{j}}{(g(x_0)-\delta)\epsilon^{d-j}} \sum_{\underline{j}} e^{-n|\Delta_j\epsilon^{(d-j)/j}(y^{\underline{j}}-x_{0}^{\underline{j}})|/2}\nonumber\\&\qquad \times\left(|\log (n|\Delta_j\epsilon^{(d-j)/j}(y^{\underline{j}}-x_{0}^{\underline{j}})|)|^{j-1}+1\right) \mathds{1}_{y^{\underline{j}} \in \mcal{C}_{\bm{\epsilon}}(x_{0}^{\underline{j}})}.
\end{align}
Combining \eqref{lowd}, \eqref{sublowd} and \eqref{c2y}, the result now follows by \eqref{cscale} upon replacing $n$ as $\alpha (k+1)^{-1} n$.
\end{proof}

\begin{Remark}\label{remarkc}
We will often make use of this translation $y^{\underline{j}}-x_0^{\underline{j}}$ and the scaling $1/\epsilon$, as done in the proof of Lemma \ref{lemma6} above, in various similar integrals in later proofs: 
up to the term $ne^{-\alpha (k+1)^{-1} n(g(x_0)-\delta)\epsilon^{d}}\int_{\mbb{R}^{d}}\phi(x)g(x)dx$, for $y\notin \mcal{C}_{\bm{\epsilon}}(x_{0})$, the terms in the sum in the bound in Lemma \ref{lemma6} are obtained from lower dimensional versions of $c_{\alpha,n,x_0}$, where we upper bound the integral of $\phi g$ over coordinates that are not within an $\epsilon$ neighbourhood of the corresponding coordinates of $x_0$. This approach enables us to make use of some results by \cite{bhattacharjee2022gaussian} in the setting where $\mbb{X}=[0,1]^{d}$ and $g$ is a uniform density.
\end{Remark}

For $\alpha,s>0$, $d\in \N$, define another function $\tilde{c}_{\alpha,s}:\mbb{R}^{d}\rightarrow \mbb{R}_{+}$ as 
    \begin{align*}
        \tilde{c}_{\alpha,s}(x_0):=s\int_{\mbb{R}^{d}}e^{-\alpha s\int_{\rec(x,x_0)}g(z)dz}\phi(x)g(x)dx,
    \end{align*}
where the function $\phi(x):\mbb{R}^{d}\rightarrow\mbb{R}$ as before is a.e.\ continuous with $\int_{\mbb{R}^{d}}\phi(x)g(x)dx<\infty$. As for $c_{\alpha,s,x_{0}}$, the function $\tilde{c}_{\alpha,s}$ also satisfies a scaling property. Consequently, we will often take $\alpha=1$. Below we use the notation $\mcal{O}_{\delta}$ to mean that the constant in the $\mcal{O}$ term may depend on $\delta$ (it may also depend on other parameters and functions, such as $\alpha$, $g,\phi, x_0$, which remain fixed for us).

\begin{Lemma}\label{lemma7}
    Under the same setting of Lemma \ref{lemma6}, for all $n,d\ge 1$ and $\alpha>0$, we have
    \begin{align*}
        \tilde{c}_{\alpha (k+1)^{-1},n}(x_0)&\le ne^{-\alpha (k+1)^{-1}n(g(x_0)-\delta)\epsilon^{d}}\int_{\mbb{R}^{d}}\phi(x)g(x)dx\\&+\sum_{j=1}^{d}\frac{\binom{d}{j}\Lambda_j}{\alpha(g(x_0)-\delta)}\left(\frac{k+1}{\epsilon^{d-j}}\right)\mcal{O}\left(\log^{j-1}\left(\frac{n(g(x_0)-\delta)\epsilon^{d}}{k+1}\right)\right),
    \end{align*}
    where $\Lambda_j$ for $j\in[d]$ is as in Lemma \ref{lemma6}. Moreover, for fixed $0<\delta<1/2\ g(x_{0})$, we have
    \begin{align}\label{eq:ctildeorder}
        \tilde{c}_{\alpha (k+1)^{-1},n}(x_0)=(k+1)\mcal{O}_{\delta}\left(\log^{d-1}\left(\frac{n}{k+1}\right)\right).
    \end{align}

\end{Lemma}

\begin{proof}
    We follow a very similar argument as in the proof of Lemma \ref{lemma6} and also use notation introduced there. For $s>0$, write
\begin{align*}
    \tilde{c}_{1,n}(x_0)&=n\int_{\mbb{R}^{d}}e^{- n\int_{\rec(x,x_0)}g(z)dz}\phi(x)g(x)dx\\&=n\int_{\mcal{C}_{\bm{\epsilon}}(x_{0})}e^{- n\int_{\rec(x,x_0)}g(z)dz}\phi(x)g(x)dx\\&\qquad+n\int_{\mbb{R}^{d}\backslash \mcal{C}_{\bm{\epsilon}}(x_{0})}e^{- n\int_{\rec(x,x_0)}g(z)dz}\phi(x)g(x)dx\\&:=\tilde{c}_1(x_0)+\tilde{c}_2(x_0).
\end{align*}

Arguing as for bounding $c_1$ in the proof of Lemma \ref{lemma6}, with $\Lambda_j$ as therein, we have
\begin{align}\label{Elowd}
    \tilde{c}_1(x_0)&
    \le 2^{d}\frac{\Lambda_d}{g(x_0)-\delta}n\int_{\mbf{0}\prec x\prec \Delta_{d} \bm{\epsilon}}e^{-n|x|}dx\nonumber\\
    & \le 2^{d}\frac{(\phi(x_{0})+\delta)(g(x_{0})+\delta)}{g(x_0)-\delta}n\int_{\mbf{0}\prec x\prec \Delta_{d} \bm{\epsilon}}e^{-n|x|}dx,
\end{align}
where the additional $2^d$ is due to identical integrals over the $2^d$ different orthants.
It is well known that \citep[see, e.g.,][]{bai2005maxima} for $n\ge 1$ and $d\in \N$.
\begin{align*}
    n\int_{[0,1]^{d}}e^{-b|x|}dx=\mcal{O}(\log^{d-1}n),
\end{align*}
Similar to the proof of Lemma \ref{lemma6}, by the transformation $\tilde{x}=(\Delta_d\epsilon)^{-1}x$, we obtain
\begin{align}\label{Ecy1}
    n\int_{\mbf{0}\prec x\prec \Delta_{d} \bm{\epsilon}}e^{-n|x|}dx=\mcal{O}(\log^{d-1}(n\Delta_{d}^{d}\epsilon^{d}))=\mcal{O}_{\delta}(\log^{d-1}n).
\end{align}

Next, we bound $\tilde{c}_{2}(x_0)$. Note using the same notation as in \eqref{partC} and \eqref{partCC}, arguing similarly as in \eqref{j0} we have
\begin{align}\label{Ej0}
    n\int_{\mcal{C}_{0}}e^{- n\int_{\rec(x,x_0)}g(z)dz}\phi(x)g(x)dx 
    \le n e^{-n(g(x_0)-\delta)\epsilon^{d}}\int_{\mbb{R}^{d}}\phi(x)g(x)dx.
\end{align}
For $j\in[d-1]$, the argument for $\mcal{C}_j$ also mimics the same in Lemma \ref{lemma6}. As in there, one may write upon integrating $\phi(x)g(x)$ over the coordinates $[d]\backslash \underline{j}$ and bounding $(\phi g)_{\underline{j}}$ by $\Lambda_j$ and letting $\Delta_j:=(g(x_0)-\delta)^{\frac{1}{j}}$,
\begin{align}\label{Ehighd}
    &n\int_{\mcal{C}_{j}}e^{- n\int_{\rec(x,x_0)}g(z)dz}\phi(x)g(x)dx \nonumber\\
&\le n \sum_{\underline{j}}\int_{\mcal{C}_{\bm{\epsilon}}(x_{0}^{\underline{j}})}e^{-n(g(x_0)-\delta)\epsilon^{d-j}\prod_{l \in \underline{j}}|x^{(l)}-x_0^{(l)}|}(\phi g)_{\underline{j}}(x^{\underline{j}})dx^{\underline{j}}\nonumber\\
&\le 2^{j}\Lambda_j n \sum_{\underline{j}}\int_{[0,\epsilon]^j}e^{-n(g(x_0)-\delta)\epsilon^{d-j}|x^{\underline{j}}|}dx^{\underline{j}}\nonumber\\&\le \frac{2^{j}\binom{d}{j}\Lambda_j}{(g(x_0)-\delta)\epsilon^{d-j}}n\int_{[0,\Delta_j \epsilon^{d/j}]^j}e^{-n|x^{[j]}|}dx^{[j]}.
\end{align}
Note that the integral on the last inequality \eqref{Ehighd} is exactly the same as the integral in \eqref{Elowd} when $j=d$. Therefore, arguing as for bounding $\tilde{c}_1$ above, we have for $j=1,\ldots,d-1$,
\begin{align}\label{Ejge0}
    n\int_{\mcal{C}_{j}}e^{- n\int_{\rec(x,x_0)}g(z)dz}\phi(x)g(x)dx\le \frac{2^{j}\binom{d}{j}\Lambda_j}{(g(x_0)-\delta)\epsilon^{d-j}}\mcal{O}(\log^{j-1}(n\Delta_{j}^{j}\epsilon^{d})).
\end{align}
Combining $\eqref{Ej0}$ and  \eqref{Ejge0}, we have
\begin{align}\label{Ecy2}
    \tilde{c}_2(x_0)&\le ne^{-n(g(x_0)-\delta)\epsilon^{d}}\int_{\mbb{R}^{d}}\phi(x)g(x)dx
    + \sum_{j=1}^{d-1}\frac{2^j\binom{d}{j}\Lambda_j}{(g(x_0)-\delta)\epsilon^{d-j}}\mcal{O}(\log^{j-1}(n\Delta_{j}^{j}\epsilon^{d})).
\end{align}

The result now follows from noting that $\tilde{c}_{\alpha (k+1)^{-1},n}(x_0)=(k+1)\alpha^{-1} \tilde{c}_{1,\alpha (k+1)^{-1} n}(x_0)$, by replacing $n$ as $\alpha (k+1)^{-1} n$ and combining the bounds in \eqref{Elowd}, \eqref{Ecy1} and \eqref{Ecy2}.
\end{proof}

Denote by $I_+ \subseteq [d]$ the coordinates $i\in [d]$ with $x^{(i)}\ge 0$, so that for $j\in [d]\backslash I_{+}$, $x^{(j)}<0$. Define $\rec(x,\partial \mbb{R}^{d}):=\prod_{i\in I_{+},j\in I_{+}^{c}}(-\infty,x^{(j)}]\times [x^{(i)},\infty)$ as the "hyperrectangle" defined by $x$ and the boundary of $\mbb{R}^{d}$. Let $A_{x_0}(x)$ denote the set of points in $\R^d$ which are in the same orthant as $x \in \R^d$ w.r.t.\ $x_0$. For $x_1,x_2\in\mbb{R}^{d}$ with $x_2 \in A_{x_0}(x_1)$ (i.e., $x_1$ and $x_2$ are in the same orthant w.r.t.\ $x_0$), denote by $(x_1\vee x_2)_{x_0}$ the unique point in $\rec(x_1,\partial \mbb{R}^{d})\cap \rec(x_2,\partial \mbb{R}^{d})\neq \emptyset$ 
having the minimal distance to $x_0$. In particular, when $x_0=\mbf{0}$ and $x_{0}\prec x_1,x_2$, we have $(x_1\vee x_2)_{x_0}=x_{1}\vee x_{2}$.

In the setting when $g$ is uniform on $\mbb{X}=[0,1]^{d}$ and $\phi\equiv 1$ on $\mbb{X}$, all three bounds in the following result follow according to \citet[Lemma 3.2]{bhattacharjee2022gaussian}. Here, we extend these results to the general setting we consider. 
\begin{Lemma}\label{lemma8}
For all 
$i\in\mbb{N}$, $\alpha, t>0$, $n,d\ge 2$ and $x_0,\delta$ as in Lemma \ref{lemma6}, when $k<n-1$, we have 
\begin{align}\label{eq:cint}
    n\int_{\mbb{R}^{d}}c_{\alpha (k+1)^{-1},n,x_0}(y)^{t}\phi(y)g(y)dy=(k+1)^{t+1}\mcal{O}_{\delta}\left(\log^{d-1}\left((k+1)^{-1}n\right)\right),
\end{align}
\begin{align*}
    &n\int_{\mbb{R}^{d}}\left(n\int_{\mbb{R}^{d}} \mathds{1}_{x \in A_{x_0}(y)} c_{\alpha(k+1)^{-1},n,x_0}((x\vee y)_{x_0})\phi(x)g(x)dx\right)^{i}g(y)dy\\&\qquad\qquad=(k+1)^{2i+1}\mcal{O}_{\delta}\left(\log^{d-1}\left((k+1)^{-1}n\right)\right),
\end{align*}
\begin{align}\label{generallemma10}
    &n\int_{\mbb{R}^{d}}\left(n\int_{\mbb{R}^{d}} \mathds{1}_{x \in A_{x_0}(y)} e^{-\alpha(k+1)^{-1} n\int_{\rec\left(x_0,(x\vee y)_{x_0}\right)}g(z)dz}\phi(x)g(x)dx\right)^{i}g(y)dy\nonumber\\&\qquad\qquad=(k+1)^{i+1}\mcal{O}_{\delta}\left(\log^{d-1}\left((k+1)^{-1}n\right)\right).
\end{align}
\end{Lemma}

\begin{proof}
The arguments employed to prove the bounds are similar to that used by \citet[Lemma 3.2] {bhattacharjee2022gaussian}, upon using the approach outlined in Remark \ref{remarkc}, and demonstrated in the proof of Lemma \ref{lemma6}. Therefore, we will only give a very brief outline of the proofs here by using notation introduced in the proofs above. We start by proving the first two bounds. Note that 
for all $t>0$, $n\ge 1$, $j \in [n]$ and $\alpha>0$, we trivially have
\begin{align*}
n^{t+1} e^{-j\alpha (k+1)^{-1}n(g(x_0)-\delta)\epsilon^{d}} 
=(k+1)^{t+1}\mcal{O}_{\delta}\left(\log^{d-1}((k+1)^{-1}n)\right).
\end{align*}
On the other hand, as argued in Remark \ref{remarkc}, for the terms in the sum in the upper bound in Lemma \ref{lemma6}, for each $\underline{j}$ with $y^{\underline{j}} \in \mcal{C}_{\bm{\epsilon}}(x_{0}^{\underline{j}})$, we can first integrate over the other $d-j$ coordinates and then upper bound $(\phi g)_{\underline{j}}$ uniformly over $C_{\bm{\epsilon}}(x_{0}^{\underline{j}})$ by $\Lambda_j$, and finally arguing as in the proof of \cite[Equation (3.10) in Lemma 3.2] {bhattacharjee2022gaussian} 
with $s$ there replaced by $\frac{n (g(x_0)-\delta)\epsilon^{d-j}}{k+1}$, we can obtain an upper bound for the integrals of each of these summands. Since $\int_{\mbb{R}^{d}}\phi(x)g(x)dx<\infty$, the first conclusion follows by a simple application of Lemma \ref{eq:basicanalysis}. 
The second conclusion now also follows arguing exactly as in \cite[Equation (3.12) in Lemma 3.2] {bhattacharjee2022gaussian}. Arguing similarly using Lemma \ref{lemma7} instead of Lemma \ref{lemma6}, the last bound also follows mimicking arguments in the proof of \cite[Equation (3.11) in Lemma 3.2] {bhattacharjee2022gaussian}.
\end{proof}

Before we proceed to more results related to the function $c_{\alpha,n,x_0}$, which serves as an upper bound on the probability \eqref{regionofstab} that a point $\check y$ is in the region of stabilization of another point $\check x$, we present the following lemma providing a lower bound to this probability. The result indeed shows that the upper bounds in Lemma \ref{lemma8} which are polynomial in $k$ are tight by our method.


\begin{Lemma}\label{lowertightk}
Let $\phi$ be bounded from below by a constant $C_{\phi} >0$. Then for $t>0$, there exists a constant $C_{low}>0$ depending only on $d,\alpha$ and $t$ such that for $y\in\rec(x_0,x)$, when $k\le 2n$, we have:
\begin{align*}
    n\int_{\mbb{R}^{d}}\left(n\int_{\mbb{R}^{d}}\left(\mbb{P}(\check{y}\in R_{n}(\check{x},\Pngm+\delta_{\check{x}}))\right)^{\alpha}\phi(x)g(x)dx\right)^{t}\phi(y)g(y)dy\ge C_{low}k^{t+1}.
\end{align*}
    
\end{Lemma}

\begin{proof}
    Since by \eqref{eq:gpcont}, the density $g$ can be lower bounded by a positive constant in some rectangle around $x_0$, without loss of generality, we consider the density $g=\mathds{1}_{[0,1]^{d}}$ and $x_0=\mbf{0}$. 
    Thus from \eqref{regionofstab} we have
    \begin{align*}
         &n\int_{\mbb{R}^{d}}\left(n\int_{\mbb{R}^{d}}\left(\mbb{P}(\check{y}\in R_{n}(\check{x},\Pngm+\delta_{\check{x}}))\right)^{\alpha}\phi(x)g(x)dx\right)^{t}\phi(y)g(y)dy\\&=n\int_{[0,1]^{d}}\left(n\int_{[0,1]^{d}}\mathds{1}_{y\in\rec(\mbf{0},x)}\psi(n,k,\mbf{0},x)^{\alpha}\phi(x)dx\right)^{t}\phi(y)dy\\&\ge C_{\phi}^{t+1}n\int_{[0,1]^{d}}\left(n\int_{[0,1]^{d}}\mathds{1}_{y\in\rec(\mbf{0},x)}\psi(n,k,\mbf{0},x)^{\alpha}dx\right)^{t}dy,
    \end{align*}
    where $\psi(n,k,\mbf{0},x)$ is defined at \eqref{Poicdf}. Note that for $g$ being uniform and $x_0=\mbf{0}$, the parameter of the Possion distribution in \eqref{Poicdf} simply becomes $n|x|$. By Poisson concentration (see Lemma \ref{concforpoi}), we have for $x$ with $n|x|\le \frac{1}{2}k$, 
    \begin{align*}
        \psi(n,k,\mbf{0},x)\ge 1-e^{-\frac{k}{8}}.
    \end{align*}
    Therefore, by restricting the integral over $y$ to region $A:=\{(k/(4n))^{1/d}\mbf{1} \prec x \prec (k/(2n))^{1/d}\mbf{1}\}$, and noting that the volume of this region is larger than $(1-d(2^{1/d} -1)) k/(4n)$, we obtain the lower bound
    \begin{align*}
        &n\int_{[0,1]^{d}}\left(n\int_{[0,1]^{d}}\mathds{1}_{y\in\rec(\mbf{0},x)}\psi(n,k,\mbf{0},x)^{\alpha}dx\right)^{t}dy\\
        &\ge n\int_{[0,(k/4n)^{1/d}]^d}
        \left(n\int_{A}\psi(n,k,\mbf{0},x)^{\alpha}dx\right)^{t}dy \\
        &\ge (1-e^{-\frac{k}{8}})^{\alpha t} (1-d(2^{1/d} -1))^t \frac{1}{4^{t+1}}k^{t+1}.
    \end{align*}
\end{proof}

\begin{Remark}
    By Lemmas \ref{lemma8} and \ref{lowertightk}, it follows that the bound on the double integral therein of the tail probability $\mbb{P}(\check{y}\in R_{n}(\check{x},\Pngm+\delta_{\check{x}}))$ of the region of stabilization is tight in $k$, i.e, the rate $k^{j+1}$ cannot be improved by our method. This is due to the fact that the Poisson distribution concentrates around its mean.
    This along with \eqref{decayn} results in a tail bound that is exponential decaying in $n$, at the cost of having a polynomial growth in $k$.
\end{Remark}

\begin{Lemma}\label{lemma9}
    For $\alpha_1,\alpha_2>0$, $0<\zeta<\beta$, $i\in\mbb{N}$, $j\in\{1,2\}$ and $n\ge 2$,
    \begin{align*}
        &n\int_{\mbb{R}^{d}}\left(n\int_{\mbb{R}^{d}} \mathds{1}_{x \in A_{x_{0}}(y)} c_{\alpha_1(k+1)^{-1},n,x_{0}}(x)^{i}e^{-\alpha_2(k+1)^{-1} n\int_{\rec\left(x_{0},(x\vee y)_{x_{0}}\right)}g(z)dz}g(x)dx\right)^{j}g(y)dy\\&\qquad\qquad=(k+1)^{(i+1)j+1}\mcal{O}_\delta\left(\log^{d-1}\left((k+1)^{-1}n\right)\right),
    \end{align*}
    \begin{align*}
        &n\int_{\mbb{R}^{d}}\left(n\int_{\mbb{R}^{d}} \mathds{1}_{x \in A_{x_{0}}(y)} c_{\zeta(k+1)^{-1},n,x_{0}}(x)^{i}c_{\beta(k+1)^{-1},n,x_{0}}((x\vee y)_{x_{0}})g(x)dx\right)^{j}g(y)dy\\&\qquad\qquad=(k+1)^{(i+2)j+1}\mcal{O}_\delta\left(\log^{d-1}\left((k+1)^{-1}n\right)\right).
    \end{align*}
\end{Lemma}

\begin{proof}
     The proof of the first bound for $j=2$ is similar to the proof of Lemma \ref{lemma8} by applying Lemma 3.4 (rather than Lemma 3.2) in \citet{bhattacharjee2022gaussian} and replacing $s$ by $\frac{n(g(x_0) - \delta)\epsilon^{d-j^\prime}}{k+1}$ for $j^\prime \in [n]$. The bounds of $A_2$ from the proof of Lemma 3.5 in \citet{bhattacharjee2022gaussian}, yield the second bound in Lemma \ref{lemma9} for $j=2$. For $j=1$, the desired two bounds follow by mimicking the derivation of the bounds of $A_1$ and $A_2$, respectively, in the proof of Lemma 3.3 in \citet{bhattacharjee2022gaussian} with $s$ replaced by $\frac{n(g(x_0) - \delta)\epsilon^{d-j^\prime}}{k+1}$, $j^\prime \in [n]$.
\end{proof}

\begin{Lemma}\label{lemma10}
For $\alpha>0$, $p>0$, $i\in\{1,2\}$ and $n\ge 2$,
\begin{align*}
    &n\int_{\mbb{R}^{d}}\left(n\int_{\mbb{R}^{d}}\mathds{1}_{x\in\rec(x_0,y)}e^{-\alpha(k+1)^{-1} n\int_{\rec(y,x_0)}g(z)dz}\phi(x)g(x)dx\right)^{i}g(y)dy\\&\qquad\qquad=(k+1)^{i+1}\mcal{O}_\delta\left(\log^{d-1}\left((k+1)^{-1}n\right)\right).
\end{align*}
\end{Lemma}

\begin{proof}
Note that when $x\in\rec(x_0,y)$, we have $(x\vee y)_{x_0}=y$. Thus, replacing $\rec(y,x_0)$ by $\rec((x\vee y)_{x_0},x_0)$ in the exponent in the integral in Lemma \ref{lemma10} and dropping the indicator $\mathds{1}_{x\in\rec(x_0,y)}$, the double integral is upper bounded by the integral in \eqref{generallemma10}, and hence the result follows by invoking Lemma \ref{lemma8}.
\end{proof}

Recall the definiton \eqref{Poicdf} of the c.d.f.\ of the Poisson distribution. 
In the following lemma, we show that the function $\psi(n,k,x_{0},x)$ has a localizing effect, "forcing" the integral of the product $\psi(n,k,x_{0},x) \phi(x) g(x)$, where $\phi$ is an a.e.\ continuous and integrable function, to converge to $\phi(x_{0})$ with the rate $k\log^{d-1}n$.

\begin{Lemma}\label{lemma11}
Under the setting of Lemma \ref{lemma6}, for $n,d\ge 2$, $k\ge 1$ and $k=O(n^{\alpha})$ with $0<\alpha<1$, we have 
\begin{align*}
     &\frac{2^{d}}{(d-1)!}\frac{(\phi(x_{0})-\delta)(g(x_{0})-\delta)}{g(x_{0})+\delta}k\mcal{O}_{\delta}(\log^{d-1}n)\nonumber\\&\le n\int_{\mbb{R}^{d}}\psi(n,k,x_{0},x)\phi(x)g(x)dx\\&\le \frac{2^{d}}{(d-1)!}\frac{(\phi(x_{0})+\delta)(g(x_{0})+\delta)}{g(x_{0})-\delta}k\mcal{O}_{\delta}(\log^{d-1}n).
\end{align*}
Particularly, taking $\phi(x)\equiv 1$, there exist constants $C_1>C_2>0$ (depending on the parameters $\delta,\epsilon$ and $g, x_0$) such that
\begin{align*}
    C_2 k \log^{d-1}n \le \mbb{E}L_{n,k}(x_{0})\le C_1 k \log^{d-1}n.
\end{align*}
\end{Lemma}

\begin{proof}
Arguing as in the proof of Lemma \ref{lemma7}, we have
\begin{align*}
&n\int_{\mbb{R}^{d}}\psi(n,k,x_{0},x)\phi(x)g(x)dx\\
%
%
&= \sum_{j=0}^{k-1}n\int_{\mcal{C}_{\bm{\epsilon}}(x_{0})}e^{-n\int_{\rec(x_{0},x)}g(z)dz}\frac{\left(n\int_{\rec(x_{0},x)}g(z)dz\right)^{j}}{j!}\phi(x)g(x)dx\\
&\qquad+\sum_{j=0}^{k-1}n\int_{\mbb{R}^{d}\backslash \mcal{C}_{\bm{\epsilon}}(x_{0})}e^{-n\int_{\rec(x_{0},x)}g(z)dz}\frac{\left(n\int_{\rec(x_{0},x)}g(z)dz\right)^{j}}{j!}\phi(x)g(x)dx\\
&\le \sum_{j=0}^{k-1}n\int_{\mcal{C}_{\bm{\epsilon}}(x_{0})}e^{-n(g(x_0)-\delta)|x-x_0|}\frac{\left(n(g(x_0)-\delta)|x-x_0|\right)^{j}}{j!}\phi(x)g(x)dx\\
&\qquad+\sum_{j=0}^{k-1}n\int_{\mbb{R}^{d}\backslash \mcal{C}_{\bm{\epsilon}}(x_{0})}e^{-n(g(x_0)-\delta)|x-x_0|}\frac{\left(n(g(x_0)-\delta)|x-x_0|\right)^{j}}{j!}\phi(x)g(x)dx\\
&=:\sum_{j=0}^{k-1}(e_{j,1}+e_{j,2}),
\end{align*}
where the inequality follows due to the decresingness of the c.d.f. of Poisson distribution with respect to the Poisson parameter.

We first consider $e_{j,1}$ for $0\le j\le k-1$. Note by \eqref{eq:gpcont} that
\begin{align*}
    e_{j,1}
    &\le (\phi(x_{0})+\delta)(g(x_{0})+\delta)\\
    &\hspace{1in} \times n\int_{\mcal{C}_{\bm{\epsilon}}(x_{0})}e^{-n(g(x_{0})-\delta)|x-x_0|}\frac{(n(g(x_{0}) -\delta)|x-x_0|)^{j}}{j!}dx\\
    &=2^d\frac{(\phi(x_{0})+\delta)(g(x_{0})+\delta)}{g(x_{0})-\delta}n\int_{\mbf{0}\prec x\prec \Delta_{d} \bm{\epsilon}}e^{-n|x|}\frac{(n|x|)^{j}}{j!}dx.
\end{align*}
Now, with similar calculations as in \cite{bai2005maxima}, we have
\begin{align*}
    &n\int_{[0,1]^{d}}e^{-n|x|}\frac{(n|x|)^{j}}{j!} dx\nonumber\\&=\frac{1}{j!}\int_{[0,n^{\frac{1}{d}}]^{d}}e^{-|u|}|u|^{j}du\qquad (x=n^{-1/d}u)\nonumber\\&=\frac{1}{j!}\int_{[-d^{-1}\log n,\infty)^{d}}\text{exp}\left\{e^{-\sum_{j=1}^{d}z_{j}}-(j+1)\sum_{j=1}^{d}z_{j}\right\}dz\qquad (z_j=-\log u_j)\nonumber\\&=\frac{1}{j!(d-1)!}\int_{-\log n}^{\infty}(\log n +x)^{d-1}\text{exp}(-(j+1)x-e^{-x})dx\qquad \left(x=\sum_{j=1}^{d}z_{j}\right)\nonumber\\&=\frac{1}{j!(d-1)!}\int_{0}^{n}(\log n -\log y)^{d-1}e^{-y}y^{j}dy\qquad (y=e^{-x})
    \nonumber\\&=\frac{\log^{d-1}n}{j!(d-1)!}\sum_{0\le l \le d-1}\binom{d-1}{l}\frac{(-1)^{l}}{\log^{l}n}\int_{0}^{\infty}(\log^{l}y) e^{-y}y^{j}dy+\mcal{O}(e^{-n} n^j \log^{d-1}n)\nonumber\\&=\frac{\log^{d-1}n}{(d-1)!}+\mcal{O}(\log j \log^{d-2}n),
\end{align*}
where, noting that 
$
    \int_{0}^{\infty}(\log^{l}y) e^{-y}\frac{y^{j}}{j!}dy
$
is the $l$-th moment of $\log X$ with $X$ following a gamma distribution $\text{Gamma}(j+1,1)$ for $0<l<d$, according to the moment generating function of $\log$-Gamma distribution, we have
\begin{align*}
    \int_{0}^{\infty}(\log^{l}y) e^{-y}\frac{y^{j}}{j!}dy=\frac{\Gamma^{(l)}(j+1)}{\Gamma(j+1)}=\mcal{O}(\log^{l}j),
\end{align*}
with $\Gamma(\cdot)$ as the gamma function and $\Gamma^{(l)}(\cdot)$ denotes its $l$-th derivative.

Therefore, with the transformation $\tilde{x}=(\Delta_{d}\epsilon)^{-1}x$, we obtain
\begin{align}\label{ej1}
    \sum_{j=0}^{k-1}e_{j,1}&\le \frac{2^{d}}{(d-1)!}\frac{(\phi(x_{0})+\delta)(g(x_{0})+\delta)}{g(x_{0})-\delta}(k\log^{d-1}(n \Delta_{d}^{d}\epsilon^{d}) +k\mcal{O} \left(\log k \log^{d-2}(n\Delta_{d}^{d}\epsilon^{d})\right)\nonumber\\&= \frac{2^{d}}{(d-1)!}\frac{(\phi(x_{0})+\delta)(g(x_{0})+\delta)}{g(x_{0})-\delta}k\ \mcal{O}_{\delta}
    \left(\log^{d-1}n\right).
\end{align}

Next, for $e_{j,2}$, $0\le j\le k-1$, similar to bounding $\tilde{c}_{2}$ in Lemma \ref{lemma7}, we have
\begin{align}\label{ej2}
    \sum_{j=0}^{k-1}e_{j,2}&\le \sum_{j=0}^{k-1} ne^{-n(g(x_{0})-\delta)\epsilon^{d}}\frac{(n(g(x_{0})-\delta)\epsilon^d)^{j}}{j!}\int_{\mbb{R}^{d}}\phi(x)g(x)dx\nonumber\\&\qquad+\sum_{j=0}^{k-1} \sum_{l=1}^{d-1}\frac{2^l\binom{d}{l}\Lambda_{l}}{(g(x_0)-\delta)\epsilon^{d-l}}\mcal{O}\left(\log j \log^{l-1}(n \Delta_{l}^{l}\epsilon^{d})\right)\nonumber\\&=\mcal{O}_{\delta}\left(k \log k
    \log^{d-2}n\right).
\end{align}
Combining \eqref{ej1} and \eqref{ej2}, we obtain the upper bound.

As for the lower bound, we trivially have
\begin{align*}
&n\int_{\mbb{R}^{d}}\psi(n,k,x_{0},x)\phi(x)g(x)dx 
\ge \sum_{j=0}^{k-1}e_{j,1}.
\end{align*}
Using a similar argument as for the upper bound, letting $\Delta_d'=(g(x_0)+\delta)^{1/d}$, we have for $0\le j\le k-1$,
\begin{align}\label{ej11}
    e_{j,1}
    \ge &(\phi(x_{0})-\delta)(g(x_{0})-\delta)n\int_{\mcal{C}_{\bm{\epsilon}}(x_{0})}e^{-n(g(x_{0})+\delta)|x-x_0|}\frac{(n(g(x_{0})+ \delta)|x-x_0|)^{j}}{j!}dx\nonumber\\=&2^d\frac{(\phi(x_{0})-\delta)(g(x_{0})-\delta)}{g(x_{0})+\delta}n\int_{\mbf{0}\prec x\prec \Delta_{d}' \bm{\epsilon}}e^{-n|x|}\frac{(n|x|)^{j}}{j!}dx\nonumber\\=&\frac{2^{d}}{(d-1)!}\frac{(\phi(x_{0})-\delta)(g(x_{0})-\delta)}{g(x_{0})+\delta}(\log^{d-1}(n \Delta_{d}'^{d}\epsilon^{d})+\mcal{O}(\log j \log^{d-2}(n\Delta_{d}'^{d}\epsilon^{d}))).
\end{align}
Consequently, it yields
\begin{align*}
    \sum_{j=0}^{k-1}e_{j,1}&\ge \frac{2^{d}}{(d-1)!}\frac{(\phi(x_{0})-\delta)(g(x_{0})-\delta)}{g(x_{0})+\delta}k\mcal{O}_{\delta}\left(\log^{d-1}n\right).
\end{align*}
This proves the first assertion. By taking $\phi(x)\equiv 1$, we have the second assertion.
\end{proof}

\begin{Remark}\label{rem:B7}
    A slightly more careful computation in the proof of Lemma \ref{lemma11} above (first fixing $0<\delta<g(x_{0})/2$ and letting $n\rightarrow\infty$, and then letting $\delta\rightarrow 0$ with $k=o(\log^{(d-1)/(2\tau)}n)$ in Corollary \ref{thmrf})  gives the following limit:
\begin{align*}
    \lim_{n\rightarrow\infty}\frac{n\int_{\mbb{R}^{d}}\psi(n,k,x_{0},x)\phi(x)g(x)dx}{k\log^{d-1}n}=\frac{2^{d}}{(d-1)!}\phi(x_{0}).
\end{align*}
\end{Remark}

\begin{Remark}
It was shown by \cite{lin2006random} that if the density $g$ is bounded above and away from zero from below, the expected number of $k$-PNNs $\mbb{E}L_{n,k}(x_{0})$ to a target point $x_{0}\in\mbb{R}^{d}$ is of the order $k\log^{d-1}n$. For general $g$, one might expect that the size of $\mbb{E}L_{n,k}(x_{0})$ depends on the smoothness of $g$. Lemma \ref{lemma11} shows that the same order holds for any a.e.\ continuous density $g$.
\end{Remark}


\subsection{Proofs of Theorem \ref{thmrfw} and Corollary \ref{thmrf}}
We will employ Theorems \ref{d2d3} and \ref{dconvex} to prove the results. In view of this, we take $\bar{\mbf{F}}_{n}$ therein as $\mbf{r}_{n,k,w}$, we pick the normalizer $\rho_{n}^{(i)}:=\sqrt{\Var r_{n,k,w}(x_{0,i})}$ for all $i\in[m]$ and take
\begin{align*}
\sigma_{ij}:=\frac{\Cov(r_{n,k,w}(x_{0,i}),r_{n,k,w}(x_{0,j}))}{\sqrt{\Var r_{n,k,w}(x_{0,i})}\sqrt{\Var r_{n,k,w}(x_{0,j})}},
\end{align*}
for all $i,j\in[m]$ so that $\Sigma=\operatorname{Cov}(\mathrm{P}^{-1}_n\bar{\mbf{F}}_{n})$ in Theorems \ref{d2d3} and \ref{dconvex}. We define $\rho_{n}^{(i)}$ and $\sigma_{ij}$ similarly for the uniform weights case $r_{n,k}(x_{0,i})$. We have already checked the Assumptions \ref{r1}-\ref{r4}, \ref{t} and \ref{m} with $p_0 \in (0,6]$ for the score functions associated to $r_{n,k,w}$ and $r_{n,k}$. We fix $p_0=6$ in the sequel. Thus we can apply Theorems \ref{d2d3} and \ref{dconvex}. By our choice of $\Sigma$, we have that $\Gamma_{0}=0$ in \eqref{GAMMA0}. Also, letting $\sigma^{2}:=\inf_{x}~\Var_{P_{\bm{\varepsilon}}}r(\bm{\varepsilon}_{x},x)$, note that for all $i\in[m]$, by the law of total variance, we have
\begin{align}\label{rhoboundw}
    \left(\rho_{n}^{(i)}\right)^{2}&=\Var r_{n,k,w}(x_{0,i})\nonumber\\
    &\ge \mbb{E}\left(\Var\left(\left. 
    \sum_{(\bm{x},\bm{\varepsilon}_{\bm{x}})\in\Pngm}W_{n\bm{x}}(x_{0,i}) 
    \bm{y}_{\bm{x}}\right|\Png\right)\right)\nonumber\\&\ge\sigma^{2}\mbb{E}\left(\sum_{\bm{x}\in\Png}W_{n\bm{x}}(x_{0,i})^{2}\right).
\end{align}
Specializing to the uniform weights case $r_{n,k}(x_{0,i})$, $i\in[m]$, with $\sigma^2$ as before, by Jensen's inequality we obtain the lower bound
\begin{align}\label{rhobound}
    \left(\rho_{n}^{(i)}\right)^{2}&
\ge\sigma^{2}\mbb{E}\left(\frac{1
    }{L_{n,k}(x_{0,i})^{2}}\sum_{\bm{x}\in\Png}\mathds{1}_{\bm{x}\in\mcal{L}_{n,k}(x_{0,i})}\right)\nonumber\\&=\sigma^{2}\mbb{E}\left(\frac{1
    }{L_{n,k}(x_{0,i})}\right)\nonumber\\&\ge \sigma^2\frac{1}{\mbb{E}L_{n,k}(x_{0,i})}.
\end{align}

Recall from \eqref{omegar} the moment bound $M_n^{(i)}(x)=\Omega_{i,n} r^*_{6+p}(x)$ for $i\in[m]$, with $r_{6+p}^{*}(x):=\left(\mbb{E}_{P_{\bm{\varepsilon}}}|r(\epsilon_{x},x)|^{6+p}\right)^{1/(6+p)}$ and
\begin{equation}\label{omeganexpress}
\Omega_{i,n}:=\left\{
\begin{aligned}
&\sup_{(x,\eta): |\eta| \le 9}\left\|W_{nx}(x_{0,i},\mathcal{P}_{x,\eta})\right\|_{L_{6+p}}\qquad \text{for $r_{n,k,w}(x_{0,i})$},\\
&\sup_{(x,\eta): |\eta| \le 9} \|L_{n,k}(x_{0,i},\mathcal{P}_{x,\eta})^{-1}\|_{L_{6+p}}\quad \text{for $r_{n,k}(x_{0,i})$},
\end{aligned}
\right.
\end{equation}
where $\mathcal{P}_{x,\eta}:=\Pngm+\delta_{(x,\bm{m}_x)}+\eta$.

Recall from \eqref{decayn} that
\begin{equation}\label{eq:rnbd}
    e^{-r_{n}^{(1)}(x,y)}:=e \sum_{j=1}^{k-1} e^{-\frac{1}{j+2}n\int_{\rec(x,x_{0,1})}g(z)dz} \lesssim  ke^{-\frac{1}{k+1}n\int_{\rec(x,x_{0,1})}g(z)dz},\, y\in \rec(x_{0,1},x),
\end{equation}
and $r_{n}^{(1)}(x,y)=\infty$ otherwise. Then, from \eqref{gn} we have 
\begin{align*}
    h_{n}^{(1)}(y)&=n\int_{\mbb{R}^{d}}M_{n}^{(1)}(x)^{6+p/2}e^{-\zeta r_{n}^{(1)}(x,y)}g(x)dx\\&\lesssim \Omega_{1,n}^{6+p/2}k^{\zeta}n\int_{\mbb{R}^{d}}\mathds{1}_{y\in \rec(x_{0,1},x)}e^{-\zeta (k+1)^{-1}n\int_{\rec(x,x_{0,1})}g(z)dz}r_{6+p}^{*}(x)^{6+p/2}g(x)dx\\&\le \Omega_{1,n}^{6+p/2}k^{\zeta}c_{\zeta (k+1)^{-1},n,x_{0,1}}(y),
\end{align*}
and 
\begin{align}\label{eq:gn}
    g_{n}^{(1)}(y)&=n\int_{\mbb{R}^{d}}e^{-\zeta r_{n}^{(1)}(x,y)}g(x)dx\lesssim k^{\zeta
    }c_{\zeta (k+1)^{-1},n,x_{0,1}}(y),
\end{align}
with $c_{\zeta (k+1)^{-1},n,x_{0,1}}$ is defined in \eqref{calphan} with $\phi$ as above. Plugging these in \eqref{Gn}, we obtain
\begin{align}\label{Grf}
    &G_{n}^{(1)}(y)\nonumber \\
    =&\Omega_{1,n}r_{6+p}^{*}(y)+\Omega_{1,n}(k^{\zeta}c_{\zeta (k+1)^{-1},n,x_{0,1}}(y))^{1/(6+p/2)}(1+(k^{\zeta}c_{\zeta (k+1)^{-1},n,x_{0,1}}(y))^6)^{1/(6+p/2)}\nonumber\\\le& \Omega_{1,n}r_{6+p}^{*}(y)+\Omega_{1,n}(1+(k^{\zeta}c_{\zeta (k+1)^{-1},n,x_{0,1}}(y))^{7/(6+p/2)})\nonumber\\\lesssim& \Omega_{1,n}(r_{6+p}^{*}(y)\vee 1)+\Omega_{1,n}(k^{\zeta}c_{\zeta (k+1)^{-1},n,x_{0,1}}(y))^{7/(6+p/2)}.
\end{align}
Also, from \eqref{qn} we obtain
\begin{align}\label{eq:qbd}
    q_{n}^{(1)}(x_1,x_2)&=n\int_{\mbb{R}^{d} \times \R}\mbb{P}(\{\check{x}_1,\check{x}_2\}\subseteq R_{n}^{(1)}(\check{z},\Pngm+\delta_{\check{z}}))\mbb{Q}(d\check z) \nonumber\\
    &\lesssim kc_{(k+1)^{-1},n,x_{0,1}}\left((x_1\vee x_2)_{x_{0,1}}\right) \mathds{1}_{x_1 \in A_{x_{0,1}}(x_2)}.
\end{align}
Finally, according to \eqref{kappan}, we have
\begin{equation}\label{eq:kapbd}
    \kappa_{n}^{(1)}(x)=\mbb{P}(\xi_{n}^{(1)}(\check{x},\Pngm+\delta_{\check{x}})\neq 0)\lesssim ke^{-(k+1)^{-1}n\int_{\rec(x_{0,1},x)}g(z)dz}.
\end{equation}

While it is complicated to deal with general weights $\Omega_{i,n}$ from \eqref{omeganexpress}, we can start with the following estimate of $\Omega_{i,n}$ in the uniform case.

\begin{Lemma}\label{concetrationoftailKPNN}
In the case of uniform weights, for $k \ge 11$, $k=O(n^{\alpha})$ with $0<\alpha<1$ and $i\in[m]$, we have
\begin{align}\label{Omegabound}
\Omega_{i,n} = \sup_{(x,\eta): |\eta| \le 9} \|L_{n,k}(x_{0,i},\mathcal{P}_{x,\eta})^{-1}\|_{L_{6+p}} \lesssim \frac{1}{ k\log^{d-1} n} \asymp \frac{1}{ \mbb{E} L_{n,k}(x_{0,i})}.
\end{align}
\end{Lemma}

\begin{proof}
Fix $i \in [m]$, $x \in \mbb{R}^d$, $\eta$ with $|\eta| \le 9$. According to \citet[Proposition 4.20]{ABS22} and Lemma \ref{eq:basicanalysis}, it follows that for any integer $r \ge 1$, 
$$
\mbb{E} (L- \mbb{E} L)^{2r} \lesssim_r \left[\mbb{E}\left(\int \left(D_{\check{y}} L\right)^2 \check{\lambda}(d\check y) \right)^{r} + \mbb{E}\left(\sum_{\check y \in \Pngm} \left(D_{\check{y}}^-L\right)^2\right)^{r}\right],
$$
where $L \equiv L_{n,k}(x_{0,i},\mathcal{P}_{x,\eta})$ and $D_{\check{x}}^{-}L(\eta):=L(\eta)-L(\eta-\delta_{\check{x}})$ for $\check x \in \eta$ is the remove-one cost operator, similar to the add-one cost function in Definition \ref{domF}. Then for the first summand above, using H\"{o}lder's inequality we obtain
\begin{align}\label{eq:momL}
    \mbb{E}\left(\int \left(D_{\check{y}} L\right)^2 \check{\lambda}(d\check y) \right)^{r}
    &=  \int \mbb{E} \left[\prod_{i=1}^r \left(D_{{\check{y}_i}} L\right)^2\right]\check{\lambda}^r(d\check y_1, \hdots, d\check y_r)\nonumber\\
    &\le \left(\int \left[\mbb{E} (D_{\check{y}} L)^{2r}\right]^{1/r} \check{\lambda}(d \check{y})\right)^r.
\end{align}
Now, we also have by H\"{o}lder's inequality that
$$
\left[\mbb{E} (D_{\check{y}} L)^{2r}\right]^{1/r} \le (\mbb{E}|D_{\check{y}} L|^{2r+1})^{\frac{2}{2r+1}}\mbb{P}(D_{\check{y}}L\neq 0)^{\frac{1}{r(2r+1)}}.
$$
Notice, that $D_{\check{y}}L \neq 0$ implies that $\rec(y,x_{0,i})$ has at most $k-1$ points from the configuration $\mathcal{P}_{x,\eta}$ (since if not, then adding $\check y$ won't change the value of $L$), and hence also from the configuration $\Pngm$. Thus, by \eqref{eq:rnbd},
$$
\mbb{P}(D_{\check{y}}L\neq 0)
\lesssim k e^{-\frac{1}{k+1}n\int_{\rec(y,x_{0,i})}g(z)dz}.
$$
On the other hand, $\mbb{E}|D_{\check{y}} L|^{2r+1}$ can be bounded similarly as in Lemma \ref{lemma4}. Indeed, letting 
$$
L'=\sum_{x \in \Pngm} \mathds{1}_{x \text{ is a $k$-PNN to $x_{0,i}$ in } \mathcal{P}_{x,\eta}} =: \sum_{x \in \Pngm} \xi'(x,\Pngm),
$$
we have $|L-L'| \le 10$, so that $\mbb{E}|D_{\check{y}} L|^{2r+1} \lesssim_r 1+\mbb{E}|D_{\check{y}} L'|^{2r+1}$. Also, the scores $\xi'$ has region of stabilization as defined at \eqref{regionrf}. Now arguing as in Lemma \ref{lemma4} with $4+p/2$ replaced by $2r+1$ (i.e., $p$ replaced by $4r-6$), since $L'$ is a sum of indicators, taking the bound on the $L^{4r-2}$ (in place of $L^{4+p}$) norm in Lemma \ref{lemma4} trivially as 1, we obtain
$$
\mbb{E}|D_{\check{y}} L'|^{2r+1} \lesssim_r 1+g_n(\check y)^5
$$
with $g_n$ defined as in \eqref{gn} with $\zeta =\zeta_0:=(2r-3)/(2r+17)$ (in place of $p/(40+p)$) and $r_n$ as in \eqref{decayn}. Thus, by \eqref{eq:gn}, we have
\begin{align*}
\left[\mbb{E} (D_{\check{y}} L)^{2r}\right]^{1/r} &\lesssim_r (1 + g_n(\check y)^{\frac{10}{2r+1}} )\mbb{P}(D_{\check{y}}L\neq 0)^{\frac{1}{r(2r+1)}}\\
& \lesssim_r \left(1+ k^{\frac{10 \zeta_0}{2r+1}}c_{\zeta_0 (k+1)^{-1},n,x_{0,i}}(y)^{\frac{10}{2r+1}} \right) \left(k e^{-\frac{1}{k+1}n\int_{\rec(y,x_{0,i})}g(z)dz}\right)^{\frac{1}{r(2r+1)}}.
\end{align*}
Hence,
\begin{align*}
    \int& \left[\mbb{E} (D_{\check{y}} L)^{2r}\right]^{1/r}\check{\lambda}(d\check{y}) \lesssim_r 
    k^{\frac{1}{r(2r+1)}}n\int_{\mbb{R}^{d}}e^{-\frac{1}{r(2r+1)(k+1)}n\int_{\rec(y,x_{0,i})}g(z)dz}g(y)dy\\&\quad+ k^{\frac{10\zeta_0}{2r+1}+\frac{1}{r(2r+1)}}n\int_{\mbb{R}^{d}}c_{\zeta_0 (k+1)^{-1},n,x_{0,i}}(y)^{\frac{10}{2r+1}}e^{-\frac{1}{r(2r+1)(k+1)}n\int_{\rec(y,x_{0,i})}g(z)dz}g(y)dy\\ & \qquad\qquad\qquad\quad \qquad \le k^{\frac{1}{r(2r+1)}}n\int_{\mbb{R}^{d}}e^{-\frac{1}{r(2r+1)(k+1)}n\int_{\rec(y,x_{0,i})}g(z)dz}g(y)dy\\&\qquad\qquad\qquad\qquad \qquad +
    k^{\frac{10\zeta_0}{2r+1}+\frac{1}{r(2r+1)}}\left(n\int_{\mbb{R}^{d}}c_{\zeta_0 (k+1)^{-1},n,x_{0,i}}(y)^{\frac{20}{2r+1}}g(y)dy\right)^{\frac{1}{2}}\\&\qquad\qquad\qquad\qquad\qquad\qquad  \times \left(n\int_{\mbb{R}^{d}}e^{-\frac{2}{r(2r+1)(k+1)}n\int_{\rec(y,x_{0,i})}g(z)dz}g(y)dy\right)^{\frac{1}{2}}
    \\& \qquad\qquad\qquad\quad \qquad \lesssim_r k^{\frac{10\zeta_0}{2r+1}+\frac{1}{r(2r+1)}+1+\frac{10}{2r+1}}\log^{d-1}n,
\end{align*}
where the final step is due to \eqref{eq:ctildeorder} and \eqref{eq:cint}.
Thus, \eqref{eq:momL} yields
$$
\mbb{E}\left(\int \left(D_{\check{y}} L\right)^2 \check{\lambda}(d\check y) \right)^{r}  \lesssim_r k^{r(\tau_{\text{conc}}+1)} \log^{r(d-1)} n,
$$
where we have $\tau_{\text{conc}}:=\frac{10}{2r+1}(1+\zeta_0)+\frac{1}{r(2r+1)} \le \frac{20r+1}{r(2r+1)} \le \frac{10}{r}$.

For the second summand, using Lemma \ref{eq:basicanalysis}, followed by an application of the Mecke formula, a similar argument as above yields
\begin{align*}
    \mbb{E}\left(\sum_{\check y \in \Pngm} \left(D_{\check{y}}^-L\right)^2\right)^{r}
    & \lesssim_r \mbb{E} \left[\sum_{\check y \in \Pngm} \left(D_{\check{y}}^-L\right)^{2r} + \cdots + \sum_{\check{y}_1 \neq \cdots \neq \check{y}_r \in \Pngm} \left(D_{\check{y}_1}^-L\right)^{2} \cdots\left(D_{\check{y}_r}^-L\right)^{2} \right] \\&\lesssim k^{r(\tau_{\text{conc}}+1)} \log^{r(d-1)} n.
\end{align*}
Now by Chebyshev's inequality,
$$
\mbb{P}(L\le \mbb{E}L/2) \lesssim \frac{\mbb{E} (L- \mbb{E} L)^{2r}}{(\mbb{E}L)^{2r}}. 
$$
Next we need a lower bound on $\mbb{E}L$.
Since $k>10$, even in the presence of additional points $x$ and $\eta$ (which together are at most 10), we have
$L=L_{n,k}(x_{0,i},\mathcal{P}_{x,\eta}) \ge L_{n,k-10}(x_{0,i},\Pngm)$, so that by Lemma \ref{lemma11},
$$
\mbb{E}L \ge \mbb{E} L_{n,k-10}(x_{0,i},\Pngm) \asymp k \log^{d-1} n.
$$ 
Combining this with the above tail bound, we obtain
\begin{align*}
    \mbb{P}(L\le\mbb{E}L/2)\lesssim_r (k\log^{d-1}n)^{-2r}k^{r(\tau_{\text{conc}}+1)} \log^{r(d-1)} n = k^{-r(1-\tau_{\text{conc}})}\log^{-r(d-1)}n.
\end{align*}
Thus, since $L \ge 1$, we have:
$$
\mbb{E} L^{-(6+p)} \lesssim \frac{1}{(\mbb{E} L)^{6+p}} + \mbb{P}(L \le \mbb{E}L/2) \lesssim_r (k \log^{d-1}n)^{-(6+p)} + k^{-r(1-\tau_{\text{conc}})}\log^{-r(d-1)}n,
$$
so that
\begin{align*}
    \|L^{-1}\|_{L_{6+p}} =(\mbb{E}L^{-(6+p)})^{\frac{1}{6+p}}\lesssim_r (k\log^{d-1}n)^{-1} + k^{-\frac{r(1-\tau_{\text{conc}})}{6+p}}\log^{-\frac{r}{6+p}(d-1)}n.
\end{align*}
Now since $\tau_{\text{conc}} \le 10/r$, we have $\frac{r(1-\tau_{\text{conc}})}{6+p} \ge \frac{1}{7} (r-10)$. Thus, choosing $r=17$ yields
$$
\|L^{-1}\|_{L_{6+p}} \lesssim \frac{1}{ k \log^{d-1} n}.
$$
Since the choice of $s$ and $\eta$ were arbitrary and the upper bound above doesn't depend on this choice, taking a supremum yields the desired bound. The final part of the result is due to Lemma \ref{lemma11}.
\end{proof}

\begin{Remark}\label{rmkonconcofkpnn}
    Note that although Lemma \ref{concetrationoftailKPNN} focuses on the first moment of $L_{n,k}(x_{0,i})$ for $i\in [m]$, the proof arguments provided above are actually valid for generalizing it to any moment, i.e, $(k\log^{d-1}n)^{q}\asymp \mbb{E}L^{q}_{n,k}(x_{0,i})$ for $i\in [m]$ and $q\in\mbb{Z}$
    .
\end{Remark}

In the following, we will bound the $\Gamma_{i}$'s in Theorems \ref{d2d3} and \ref{dconvex}. Throughout, we take $\phi(x)=r_{6+p}^{*}(x)^{6+p/2}\vee 1$ in \eqref{calphan}. Write 
\begin{align*}
    f_{\alpha_i,\alpha_j,\alpha_l,\alpha}^{(i,j,l,t)}(y)&=: f_{\alpha_i,\alpha_j,\alpha_l,\alpha,1}^{(i,j,l,t)}(y)+f_{\alpha_i,\alpha_j,\alpha_l,\alpha,2}^{(i,j,l,t)}(y)+f_{\alpha_i,\alpha_j,\alpha_l,\alpha,3}^{(i,j,l,t)}(y),
\end{align*}
with
\begin{align}
    f_{\alpha_i,\alpha_j,\alpha_l,\alpha,1}^{(i,j,l,t)}(y)&:=n\int_{\mbb{X}}G_{n}^{(i)}(x)^{\alpha_i}G_{n}^{(j)}(x)^{\alpha_j}G_{n}^{(l)}(x)^{\alpha_l} e^{-\alpha r_{n}^{(t)}(x,y)}\mbb{Q}(dx)\label{fij1},\\
    f_{\alpha_i,\alpha_j,\alpha_l,\alpha,2}^{(i,j,l,t)}(y)&:=n\int_{\mbb{X}}G_{n}^{(i)}(x)^{\alpha_i}G_{n}^{(j)}(x)^{\alpha_j}G_{n}^{(l)}(x)^{\alpha_l}e^{-\alpha r_{n}^{(t)}(y,x)}\mbb{Q}(dx)\label{fij2},\\
    f_{\alpha_i,\alpha_j,\alpha_l,\alpha,3}^{(i,j,l,t)}(y)&:=n\int_{\mbb{X}}G_{n}^{(i)}(x)^{\alpha_i}G_{n}^{(j)}(x)^{\alpha_j}G_{n}^{(l)}(x)^{\alpha_l}q_{n}^{(t)}(x,y)^{\alpha}\mbb{Q}(dx)\label{fij3}.
\end{align}
We first consider then case when $i=j=l=t$, and without loss of generality, we fix $i,j,l,t=1$. We start with $\Gamma_{1}$ defined at \eqref{GAMMA1}. By \eqref{Grf}, \eqref{fij1}, \eqref{fij2}, \eqref{fij3} and Lemma \ref{eq:basicanalysis}, we have
\begin{align}\label{f1bound}
    f_{1,1,0,\beta,1}^{(1,1,1,1)}(y)&=n\int_{\mbb{R}^{d}}(G_{n}^{(1)}(x))^2e^{-\beta r_{n}^{(1)}(x,y)}g(x)dx\nonumber\\&\lesssim \Omega_{1,n}^{2}\left(n\int_{\mbb{R}^{d}}(r_{6+p}^{*}(x)^{2}\vee 1)e^{-\beta r_{n}^{(1)}(x,y)}g(x)dx\right.\nonumber\\&\left.\qquad+k^{14\zeta/(6+p/2)}n\int_{\mbb{R}^{d}}c_{\zeta(k+1)^{-1},n,x_{0,1}}(x)^{14/(6+p/2)}e^{-\beta r_{n}^{(1)}(x,y)}g(x)dx\right),
\end{align}
\begin{align}\label{f2bound}
    f_{1,1,0,\beta,2}^{(1,1,1,1)}(y)&=n\int_{\mbb{R}^{d}}(G_{n}^{(1)}(x))^2e^{-\beta r_{n}^{(1)}(y,x)}g(x)dx\nonumber\\&\lesssim \Omega_{1,n}^{2}\left(n\int_{\mbb{R}^{d}}(r_{6+p}^{*}(x)^{2}\vee 1)e^{-\beta r_{n}^{(1)}(y,x)}g(x)dx\right.\nonumber\\&\left.\qquad+k^{14\zeta/(6+p/2)}n\int_{\mbb{R}^{d}}c_{\zeta(k+1)^{-1},n,x_{0,1}}(x)^{14/(6+p/2)}e^{-\beta r_{n}^{(1)}(y,x)}g(x)dx\right),
\end{align}
and
\begin{align}\label{f3bound}
    f_{1,1,0,\beta,3}^{(1,1,1,1)}(y)&=n\int_{\mbb{R}^{d}}(G_{n}^{(1)}(x))^2q_{n}^{(1)}(x,y)^{\beta}g(x)dx\nonumber\\&\lesssim \Omega_{1,n}^{2}\left(n\int_{\mbb{R}^{d}}(r_{6+p}^{*}(x)^{2}\vee 1)q_{n}^{(1)}(x,y)^{\beta}g(x)dx\right.\nonumber\\&\left.\qquad+k^{14\zeta/(6+p/2)}n\int_{\mbb{R}^{d}}c_{\zeta(k+1)^{-1},n,x_{0,1}}(x)^{14/(6+p/2)}q_{n}^{(1)}(x,y)^{\beta}g(x)dx\right).
\end{align}

Also, from \eqref{f1bound}, we have that
\begin{align}\label{equalgamma11}
\begin{aligned}
    \big(f_{1,1,0,\beta,1}^{(1,1,1,1)}(y)\big)^2&\lesssim \Omega_{1,n}^{4}\Big(n\int_{\mbb{R}^{d}}(r_{6+p}^{*}(x)^{2}\vee 1)e^{-\beta r_{n}^{(1)}(x,y)}g(x)dx\Big)^{2}\\&+\Omega_{1,n}^{4}k^{28\zeta/(6+p/2)}\Big(n\int_{\mbb{R}^{d}}c_{\zeta(k+1)^{-1},n,x_{0,1}}(x)^{14/(6+p/2)}e^{-\beta r_{n}^{(1)}(x,y)}g(x)dx\Big)^2.
\end{aligned}
\end{align}

For the rest of this proof, we fix $\delta \in (0,1/2 \min_{i} g(x_{0,i}))$ and $\epsilon$ as in Lemma \ref{lemma6} such that the conclusion therein holds for all $x_{0,i}$, $i \in [m]$. To simplify notation, we also drop the dependence on $\delta$ and simply write $\mcal{O} \equiv \mcal{O}_{\delta}$. 
Recalling that $\phi(x)=r_{6+p}^{*}(x)^{6+p/2}\vee 1$ and using \eqref{eq:rnbd}, by \eqref{eq:cint} we have
\begin{align}\label{g11}
    &n\int_{\mbb{R}^{d}}\left(n\int_{\mbb{R}^{d}}(r_{6+p}^{*}(x)^{2}\vee 1)e^{-\beta r_{n}^{(1)}(x,y)}g(x)dx\right)^{2}g(y)dy\nonumber\\&\lesssim k^{2\beta}n\int_{\mbb{R}^d}c_{\beta(k+1)^{-1},n,x_{0,1}}(y)^2 g(y)dy=(k+1)^{2\beta+3}\mcal{O}(\log^{d-1}((k+1)^{-1}n)).
\end{align}
Also, using 
that $c_{\alpha,n,x_{0,1}}(x) \le c_{\alpha,n,x_{0,1}}(y)$ for $y \in \rec(x_{0,1},x)$, and that $\zeta<\beta$, again by \eqref{eq:rnbd} and \eqref{eq:cint},
\begin{align}\label{g12}
    &n\int_{\mbb{R}^d}\left(n\int_{\mbb{R}^{d}}c_{\zeta(k+1)^{-1},n,x_{0,1}}(x)^{14/(6+p/2)}e^{-\beta r_{n}^{(1)}(x,y)}g(x)dx\right)^2g(y)dy\nonumber\\&\le n\int_{\mbb{R}^d}c_{\zeta(k+1)^{-1},n,x_{0,1}}(y)^{28/(6+p/2)}\left(n\int_{\mbb{R}^{d}}e^{-\beta r_{n}^{(1)}(x,y)}g(x)dx\right)^2g(y)dy\nonumber\\&\lesssim k^{2\beta}n\int_{\mbb{R}^d}c_{\zeta(k+1)^{-1},n,x_{0,1}}(y)^{28/(6+p/2)+2}g(y)dy\nonumber\\&=(k+1)^{2\beta+3+ 28/(6+p/2)}\mcal{O}(\log^{d-1}((k+1)^{-1}n)).
\end{align}
Combining \eqref{g11} and \eqref{g12}, from \eqref{f1bound} we obtain
\begin{align}\label{G11}
    n\mbb{Q}\left(f_{1,1,0,\beta,1}^{(1,1,1,1)}\right)^2=\Omega_{1,n}^{4}(k+1)^{2\beta+3+  28(1+\zeta)/(6+p/2)}\mcal{O}(\log^{d-1}((k+1)^{-1}n)).
\end{align}

Next, we note by \eqref{f3bound} that
\begin{multline*}
    \left(f_{1,1,0,\beta,3}^{(1,1,1,1)}(y)\right)^2 \lesssim \Omega_{1,n}^{4}\left(n\int_{\mbb{R}^{d}}(r_{6+p}^{*}(x)^{2}\vee 1)q_{n}^{(1)}(x,y)^{\beta}g(x)dx\right)^{2}\\+\Omega_{1,n}^{4}k^{28\zeta/(6+p/2)}\left(n\int_{\mbb{R}^{d}}c_{\zeta(k+1)^{-1},n,x_{0,1}}(x)^{14/(6+p/2)}q_{n}^{(1)}(x,y)^{\beta}g(x)dx\right)^{2}.
\end{multline*}
For $0<\alpha<1$ and $y\in\mbb{R}^{d}$, we observe that
\begin{align}\label{g13}
    c_{1,n,x_{0,1}}(y)^{\alpha}&=\left(n\int_{\mbb{R}^{d}}\mathds{1}_{y\in\rec(x_{0,1},x)}e^{-n\left(\int_{\rec(x_{0,1},x)}g(z)dz-\int_{\rec(x_{0,1},y)}g(z)dz\right)}\right.\nonumber\\&\left.\qquad \times e^{-n\int_{\rec(x_{0,1},y)}g(z)dz}(r_{6+p}^{*}(x)^{6+p/2}\vee 1)g(x)dx\right)^{\alpha}\nonumber\\&= e^{-\alpha n\int_{\rec(x_{0,1},y)}g(z)dz}\bigg(n\int_{\mbb{R}^{d}}\mathds{1}_{\rec(x_{0,1},x)}e^{-n\left(\int_{\rec(x_{0,1},x)}g(z)dz-\int_{\rec(x_{0,1},y)}g(z)dz\right)}\nonumber\\&\qquad \times (r_{6+p}^{*}(x)^{6+p/2}\vee 1)g(x)dx\bigg)^{\alpha}\nonumber\\&\le e^{-\alpha n\int_{\rec(x_{0,1},y)}g(z)dz}\bigg(1+n\int_{\mbb{R}^{d}}\mathds{1}_{\rec(x_{0,1},x)} \nonumber\\&\qquad\qquad\times e^{-n\left(\int_{\rec(x_{0,1},x)}g(z)dz-\int_{\rec(x_{0,1},y)}g(z)dz\right)}(r_{6+p}^{*}(x)^{6+p/2}\vee 1)g(x)dx\bigg)\nonumber\\&=e^{-\alpha n\int_{\rec(x_{0,1},y)}g(z)dz}+c_{\alpha,n,x_{0,1}}(y).
\end{align}
Thus, using \eqref{eq:qbd} in the first step and \eqref{g13} in the second, for $0<\beta<1$), Lemma \ref{lemma8} yields 
\begin{align}\label{g17}
    &n\int_{\mbb{R}^{d}}\left(n\int_{\mbb{R}^{d}}(r_{6+p}^{*}(x)^{2}\vee 1)q_{n}^{(1)}(x,y)^{\beta}g(x)dx\right)^{2}g(y)dy\nonumber\\&\lesssim k^{2\beta}n\int_{\mbb{R}^{d}}\left(n\int_{\mbb{R}^{d}} \mathds{1}_{x \in A_{x_{0,1}}(y)} c_{(k+1)^{-1},n,x_{0,1}}((x\vee y)_{x_{0,1}})^{\beta}(r_{6+p}^{*}(x)^{2}\vee 1)g(x)dx\right)^{2}g(y)dy\nonumber\\&\lesssim k^{2\beta}n\int_{\mbb{R}^{d}}\Big(n\int_{\mbb{R}^{d}} \mathds{1}_{x \in A_{x_{0,1}}(y)} (k+1)^{\beta}e^{-\beta(k+1)^{-1} n\int_{\rec\left(x_{0,1},(x\vee y)_{x_{0,1}})\right)}g(z)dz}\nonumber\\ &\hspace{3.0in}\times (r_{6+p}^{*}(x)^{2}\vee 1)g(x)dx\Big)^{2} g(y)dy \nonumber\\&\quad+k^{2\beta}n\int_{\mbb{R}^{d}}\Big(n\int_{\mbb{R}^{d}} \mathds{1}_{x \in A_{x_{0,1}}(y)} (k+1)^{\beta-1}c_{\beta(k+1)^{-1},n,x_{0,1}}((x\vee y)_{x_{0,1}})\nonumber\\ &\hspace{3.0in}\times (r_{6+p}^{*}(x)^{2}\vee 1)g(x)dx\Big)^{2}g(y)dy\nonumber\\&=(k+1)^{4\beta+3}\mcal{O}(\log^{d-1}((k+1)^{-1}n)).
\end{align}
On the other hand, using \eqref{g13} again, arguing same as above yields
\begin{align}\label{g14}
    &n\int_{\mbb{R}^{d}}\left(n\int_{\mbb{R}^{d}} c_{\zeta(k+1)^{-1},n,x_{0,1}}(x)^{14/(6+p/2)}q_{n}^{(1)}(x,y)^{\beta}g(x)dx\right)^{2}g(y)dy\nonumber\\
    &\lesssim  k^{2\beta}n\int_{\mbb{R}^{d}}\Big(n\int_{\mbb{R}^{d}} \mathds{1}_{x \in A_{x_{0,1}}(y)} (k+1)^\beta c_{\zeta(k+1)^{-1},n,x_{0,1}}(x)^{14/(6+p/2)}\nonumber\\ &\hspace{1.5in}\times e^{-\beta(k+1)^{-1} n\int_{\rec(x_{0,1},(x\vee y)_{x_{0,1}})}g(z)dz}g(x)dx\Big)^{2}g(y)dy\nonumber\\&\quad +k^{2\beta}n\int_{\mbb{R}^{d}}\Big(n\int_{\mbb{R}^{d}} \mathds{1}_{x \in A_{x_{0,1}}(y)} (k+1)^{\beta-1}c_{\zeta(k+1)^{-1},n,x_{0,1}}(x)^{14/(6+p/2)}\nonumber\\ &\hspace{2.5in} \times c_{\beta(k+1)^{-1},n,x_{0,1}}((x\vee y)_{x_{0,1}})g(x)dx\Big)^{2}g(y)dy\nonumber\\&=:A_1+A_2.
\end{align}
Since for any $\alpha,t \ge 1$, we have $c_{\alpha,n,x_{0,1}}(x)^{t} \le \max \{c_{\alpha,n,x_{0,1}}(x)^{\lfloor t \rfloor}, c_{\alpha,n,x_{0,1}}(x)^{ \lceil t\rceil}\}$, by Lemma \ref{lemma9}, we have that $A_{1},A_2=(k+1)^{4 \beta+3+\lceil 28/(6+p/2)\rceil}\mcal{O}(\log^{d-1}((k+1)^{-1}n))$. 
Therefore, by \eqref{g17} and \eqref{g14}, we obtain
\begin{align}\label{G12}
    n\mbb{Q}\left(f_{1,1,0,\beta,3}^{(1,1,1,1)}\right)^{2}=\Omega_{1,n}^{4}(k+1)^{4\beta+3+\lceil 28(1+\zeta)/(6+p/2)\rceil}\mcal{O}(\log^{d-1}((k+1)^{-1}n)).
\end{align}
From \eqref{f2bound}, we have
\begin{multline*}
    \left(f_{1,1,0,\beta,2}^{(1,1,1,1)}(y)\right)^2 \lesssim\Omega_{1,n}^{4}\left(n\int_{\mbb{R}^{d}}(r_{6+p}^{*}(x)^{2}\vee 1)e^{-\beta r_{n}^{(1)}(y,x)}g(x)dx\right)^{2}\\+\Omega_{1,n}^{4}k^{28\zeta/(6+p/2)}\left(n\int_{\mbb{R}^{d}}c_{\zeta(k+1)^{-1},n,x_{0,1}}(x)^{14/(6+p/2)}e^{-\beta r_{n}^{(1)}(y,x)}g(x)dx\right)^2.
\end{multline*}
For the first term, using \eqref{eq:rnbd} and Lemma \ref{lemma10} yields
\begin{align}\label{g15}
    &n\int_{\mbb{R}^{d}}\left(n\int_{\mbb{R}^{d}}e^{-\beta r_{n}^{(1)}(y,x)}(r_{6+p}^{*}(x)^{2}\vee 1)g(x)dx\right)^{2}g(y)dy\nonumber\\&\lesssim k^{2\beta}n\int_{\mbb{R}^{d}}\left(n\int_{\mbb{R}^{d}}\mathds{1}_{x\in\rec(x_{0,1},y)}e^{-\beta(k+1)^{-1} n\int_{\rec(y,x_{0,1})}g(z)dz}(r_{6+p}^{*}(x)^{2}\vee 1)g(x)dx\right)^{2} g(y)dy\nonumber\\&=(k+1)^{2\beta+3}\mcal{O}(\log^{d-1}((k+1)^{-1}n)).
\end{align}
As for the second term, changing the order of integration in the second step and using the Cauchy-Schwarz inequality in the third, \eqref{eq:cint} and Lemma \ref{lemma9} yield
\begin{align}\label{g16}
    &n\int_{\mbb{R}^d}\left(n\int_{\mbb{R}^{d}}c_{\zeta(k+1)^{-1},n,x_{0,1}}(x)^{14/(6+p/2)}e^{-\beta r_{n}^{(1)}(y,x)}g(x)dx\right)^2g(y)dy\nonumber\\&\lesssim k^{2\beta}n\int_{\mbb{R}^d}\Big(n\int_{\mbb{R}^{d}}\mathds{1}_{x\in\rec(x_{0,1},y)}c_{\zeta(k
+1)^{-1},n,x_{0,1}}(x)^{14/(6+p/2)}e^{-\beta(k+1)^{-1} n \int_{\rec(x_{0,1},y)}g(z)dz}\nonumber\\ &\hspace{3.5in} \times g(x)dx\Big)^2 g(y)dy\nonumber\\&\le k^{2\beta}n^2\int_{\mbb{R}^d\times \mbb{R}^d} \Big(\mathds{1}_{x \in A_{x_{0,1}}(y)} c_{\zeta(k+1)^{-1},n,x_{0,1}}(x)^{14/(6+p/2)}c_{\zeta(k+1)^{-1},n,x_{0,1}}(y)^{14/(6+p/2)}\nonumber\\ &\hspace{2.5in}\times c_{\beta(k+1)^{-1},n,x_{0,1}}((x\vee y)_{x_{0,1}})g(x)g(y)\Big)dxdy\nonumber\\ &\le k \left(n\int_{\mbb{R}^{d}}c_{\zeta(k+1)^{-1},n,x_{0,1}}(x)^{28/(6+p/2)}g(x)dx\right)^{\frac{1}{2}} A_{2}^{\frac{1}{2}}\nonumber\\&=(k+1)^{2\beta+3+28/(6+p/2)}\mcal{O}(\log^{d-1}((k+1)^{-1}n)),
\end{align}
where $A_{2}$ is defined at \eqref{g14}.

Combining \eqref{g15} and \eqref{g16}, we have
\begin{align}\label{G13}
    n\mbb{Q}\left(f_{1,1,0,\beta,2}^{(1,1,1,1)}\right)^2=\Omega_{1,n}^{4}(k+1)^{2\beta+3+\lceil 28(1+\zeta)\rceil/(6+p/2)}\mcal{O}(\log^{d-1}((k+1)^{-1}n)).
\end{align}
Now, putting together \eqref{G11}, \eqref{G12} and \eqref{G13}, we obtain
\begin{align*}
    n\mbb{Q}\left(f_{1,1,0,\beta}^{(1,1,1,1)}\right)^{2}=\Omega_{1,n}^{4}(k+1)^{4\beta+3+\lceil 28(1+\zeta)/(6+p/2)\rceil}\mcal{O}(\log^{d-1}((k+1)^{-1}n)).
\end{align*}
Therefore, by \eqref{rhoboundw}, \eqref{omeganexpress} (\eqref{rhobound} and  \eqref{Omegabound}, respectively, in the uniform case) and Lemma \ref{lemma11},
with
\begin{align}\label{varsignmawnk}
    \varsigma:=\frac{1+\zeta}{6+p/2},\quad \text{and}\quad W_{1}(n,k):=\frac{\left(\sup_{(x,\eta): |\eta| \le 9} \left\|W_{nx}(x_{0,1},\mathcal{P}_{x,\eta})\right\|_{L_{6+p}}\right)^{2}}{\mbb{E}\left(\sum_{\bm{x}\in \Png}W_{n\bm{x}}(x_{0,1})^{2}\right)},
\end{align}
we have
\begin{align}\label{gamma1ii}
   &\left(\frac{n\mbb{Q}\bigg(f_{1,1,0,\beta}^{(1,1,1,1)}\bigg)^{2}}{(\varrho_{n}^{(1)}\varrho_{n}^{(1)})^{2}}\right)^{\frac{1}{2}}=\left\{
   \begin{aligned}
   &k^{2\beta+3/2+\lceil 14\varsigma\rceil}W_{1}(n,k)\mcal{O}(\log^{(d-1)/2}n),\quad \text{for general weights},\\
   &k^{2\beta+1/2+\lceil 14\varsigma\rceil}\mcal{O}(\log^{-(d-1)/2}n),\qquad\qquad\text{for uniform weights}.
   \end{aligned}
   \right.
\end{align}

Next, we focus on $\Gamma_{3}$ and $\Gamma_{4}$ defined at \eqref{GAMMA3} and \eqref{GAMMA4}, respectively. For $i \in \{3,4\}$, by Lemma \ref{eq:basicanalysis} we have
\begin{align*}
    G_{n}^{(1)}(x)^{i}(\kappa_{n}^{(1)}(x)+g_{n}^{(1)}(x))^{6\beta}\lesssim G_{n}^{(1)}(x)^{i}\kappa_{n}^{(1)}(x)^{6\beta}+G_{n}^{(1)}(x)^{i}g_{n}^{(1)}(x)^{6\beta}.
\end{align*}
Using \eqref{Grf} in the first step, \eqref{eq:kapbd} in the second, and \eqref{eq:ctildeorder} and \eqref{eq:cint} in the final one, for $i \in \{3,4\}$ we obtain
\begin{align}\label{g21}
    &n\int_{\mbb{R}^{d}}G_{n}^{(1)}(x)^{i}\kappa_{n}^{(1)}(x)^{6\beta}g(x)dx\nonumber\\&\lesssim \Omega_{1,n}^{i}n\int_{\mbb{R}^{d}}\kappa_{n}^{(1)}(x)^{6\beta}(r^{*}(x)^{i}\vee 1)g(x)dx\nonumber\\&\qquad+\Omega_{1,n}^{i}k^{7i\zeta/(6+p/2)}n\int_{\mbb{R}^{d}}c_{\zeta(k+1)^{-1},n,x_{0,1}}(x)^{7i/(6+p/2)}\kappa_{n}^{(1)}(x)^{6\beta}g(x)dx\nonumber\\&\lesssim \Omega_{1,n}^{i}k^{6\beta}n\int_{\mbb{R}^{d}}e^{-6\beta(k+1)^{-1} n\int_{\rec(x_{0,1},x)}g(z)dz}(r^{*}(x)^{i}\vee 1)g(x)dx\nonumber\\&+\Omega_{1,n}^{i}k^{6\beta+7i\zeta/(6+p/2)}n\int_{\mbb{R}^{d}}c_{\zeta(k+1)^{-1},n,x_{0,1}}(x)^{7i/(6+p/2)}e^{-6\beta(k+1)^{-1} n\int_{\rec(x_{0,1},x)}g(z)dz}g(x)dx\nonumber\\&\le \Omega_{1,n}^{i}k^{6\beta}n\int_{\mbb{R}^{d}}e^{-6\beta(k+1)^{-1} n\int_{\rec(x_{0.1},x)}g(z)dz}(r^{*}(x)^{i}\vee 1)g(x)dx\nonumber\\&\qquad+\Omega_{1,n}^{i}k^{6\beta+7i\zeta/(6+p/2)}\left(n\int_{\mbb{R}^{d}}c_{\zeta(k+1)^{-1},n,x_{0,1}}(x)^{14i/(6+p/2)}g(x)dx\right)^{\frac{1}{2}}\nonumber\\&\qquad\qquad\qquad \times \left(n\int_{\mbb{R}^{d}}e^{-12\beta(k+1)^{-1} n\int_{\rec(x_{0,1},x)}g(z)dz}g(x)dx\right)^{\frac{1}{2}}\nonumber\\&=\Omega_{1,n}^{i}(k+1)^{6\beta+1+ \lceil 7i \varsigma\rceil}\mcal{O}(\log^{d-1}((k+1)^{-1}n)).
\end{align}
Applying \eqref{eq:cint} in Lemma \ref{lemma8}, from \eqref{eq:gn} and \eqref{Grf} we also have
\begin{align}\label{g23}
    &n\int_{\mbb{R}^{d}}G_{n}^{(1)}(x)^{i}g_{n}^{(1)}(x)^{6\beta}g(x)dx\nonumber\\&\lesssim \Omega_{1,n}^{i}k^{6\zeta\beta}n\int_{\mbb{R}^{d}}c_{\zeta(k+1)^{-1},n,x_{0,1}}(x)^{6\beta}(r^{*}(x)^{i}\vee 1)g(x)dx\nonumber\\&\qquad+\Omega_{1,n}^{i}k^{6\zeta\beta+7i\zeta/(6+p/2)}n\int_{\mbb{R}^{d}}c_{\zeta(k+1)^{-1},n,x_{0,1}}(x)^{7i/(6+p/2)}c_{\zeta(k+1)^{-1},n,x_{0,1}}(x)^{6\beta}g(x)dx 
    \nonumber\\&=\Omega_{1,n}^{i}(k+1)^{6\zeta\beta+6\beta+1+ \lceil 7i\varsigma\rceil}\mcal{O}(\log^{d-1}((k+1)^{-1}n)),
\end{align}
where the last step is by \eqref{eq:ctildeorder} and \eqref{eq:cint}.
Combining \eqref{g21} and \eqref{g23}, we obtain 
\begin{align*}
    &n\mbb{Q}\left((\kappa_{n}^{(1)}+g_{n}^{(1)})^{6\beta}\left(G_{n}^{(1)}\right)^{i}\right)= \Omega_{1,n}^{i}(k+1)^{6\zeta\beta+6\beta+1+ \lceil 7i\varsigma\rceil}\mcal{O}(\log^{d-1}((k+1)^{-1}n)).
\end{align*}
Therefore, \eqref{rhoboundw}, \eqref{omeganexpress} (\eqref{rhobound} and  \eqref{Omegabound}, respectively, in the uniform case) and Lemma \ref{lemma11} yield that for $i \in \{3,4\}$,
\begin{align}\label{G24}
&\frac{n\mbb{Q}\left((\kappa_{n}^{(1)}+g_{n}^{(1)})^{6\beta}\left(G_{n}^{(1)}\right)^{i}\right)}{(\varrho_n^{(1)})^{i}}\nonumber\\&=
\left\{
   \begin{aligned}
   &k^{6\zeta\beta+6\beta+1+\lceil 7i\varsigma\rceil}W_{1}(n,k)^{i/2}\mcal{O}(\log^{d-1}n),\quad \text{for general weights},\\
   &k^{6\zeta\beta+6\beta+1-i/2+ \lceil 7i\varsigma\rceil}\mcal{O}(\log^{(d-1)(1-i/2)}n),\quad \text{for uniform weights},
   \end{aligned}
   \right.
\end{align}
yielding a bound for $\Gamma_3$ and the first summand of $\Gamma_4$.

For the second summand in $\Gamma_{4}$, note that by \eqref{fij1} and \eqref{Grf}, we have
\begin{align*}
    f_{2,2,0,3\beta,1}^{(1,1,1,1)}(y)&\lesssim \Omega_{1,n}^{4}\left(n\int_{\mbb{R}^{d}}(r_{6+p}^{*}(x)^{4}\vee 1)e^{-3\beta r_{n}^{(1)}(x,y)}g(x)dx\right.\\&\left.\qquad+k^{28\zeta/(6+p/2)}n\int_{\mbb{R}^{d}}c_{\zeta(k+1)^{-1},n,x_{0,1}}(x)^{28/(6+p/2)}e^{-3\beta r_{n}^{(1)}(x,y)}g(x)dx\right).
\end{align*}
By \eqref{eq:rnbd} and \eqref{eq:cint}, 
\begin{align}\label{g31}
    &n\int_{\mbb{R}^{d}}n\int_{\mbb{R}^{d}}e^{-3\beta r_{n}^{(1)}(x,y)}(r_{6+p}^{*}(x)^{4}\vee 1)g(x)dxg(y)dy\nonumber\\&\lesssim k^{3\beta}n\int_{\mbb{R}^{d}}c_{3\beta(k+1)^{-1},n,x_{0,1}}(y)g(y)dy=(k+1)^{3\beta+2}\mcal{O}(\log^{d-1}((k+1)^{-1}n)).
\end{align}
Using that $\zeta<3\beta$, we have by \eqref{g12}, \eqref{eq:rnbd} and \eqref{eq:cint} that
\begin{align}\label{g32}
    &n\int_{\mbb{R}^d}n\int_{\mbb{R}^{d}}c_{\zeta(k+1)^{-1},n,x_{0,1}}(x)^{28/(6+p/2)}e^{-3\beta r_{n}^{(1)}(x,y)}g(x)dxg(y)dy\nonumber\\&\le n\int_{\mbb{R}^d}c_{\zeta(k+1)^{-1},n,x_{0,1}}(y)^{28/(6+p/2)}\left(n\int_{\mbb{R}^{d}}e^{-\zeta r_{n}^{(1)}(x,y)}g(x)dx\right)g(y)dy\nonumber\\&\lesssim k^{3\beta}n\int_{\mbb{R}^d}c_{\zeta(k+1)^{-1},n,x_{0,1}}(y)^{28/(6+p/2)+1}g(y)dy\nonumber\\&=(k+1)^{3\beta+2+ 28/(6+p/2)}\mcal{O}(\log^{d-1}((k+1)^{-1}n)).
\end{align}
Together, \eqref{g31} and \eqref{g32} yield
\begin{align}\label{G21}
    n\mbb{Q}f_{2,2,0,3\beta,1}^{(1,1,1,1)}=\Omega_{1,n}^{4}(k+1)^{3\beta+2+ 28\varsigma}\mcal{O}(\log^{d-1}((k+1)^{-1}n)).
\end{align}
Next, note that by \eqref{fij3} and \eqref{Grf},
\begin{align*}
    f_{2,2,0,3\beta,3}^{(1,1,1,1)}(y)&\lesssim \Omega_{1,n}^{4}\left(n\int_{\mbb{R}^{d}}(r_{6+p}^{*}(x)^{4}\vee 1)q_{n}^{(1)}(x,y)^{3\beta}g(x)dx\nonumber\right.\\&\left.\qquad+k^{28\zeta/(6+p/2)}n\int_{\mbb{R}^{d}}c_{\zeta(k+1)^{-1},n,x_{0,1}}(x)^{28/(6+p/2)}q_{n}^{(1)}(x,y)^{3\beta}g(x)dx\right).
\end{align*}
Since $0<3\beta<1$, using \eqref{eq:qbd} and arguing as for \eqref{g13} in the second step, Lemma \ref{lemma8} yields
\begin{align}\label{g33}
    &n\int_{\mbb{R}^{d}}n\int_{\mbb{R}^{d}}q_{n}^{(1)}(x,y)^{3\beta}(r_{6+p}^{*}(x)^{4}\vee 1)g(x)dxg(y)dy\nonumber\\&\lesssim k^{3\beta}n\int_{\mbb{R}^{d}}n\int_{\mbb{R}^{d}}\mathds{1}_{x \in A_{x_{0,1}}(y)} c_{(k+1)^{-1},n,x_{0,1}}((x\vee y)_{x_{0,1}})^{3\beta}(r_{6+p}^{*}(x)^{4}\vee 1)g(x)dxg(y)dy\nonumber\\&\lesssim k^{3\beta}n\int_{\mbb{R}^{d}}n\int_{\mbb{R}^{d}} \mathds{1}_{x \in A_{x_{0,1}}(y)} (k+1)^{3\beta-1}c_{3\beta(k+1)^{-1},n,x_{0,1}}((x\vee y)_{x_{0,1}})(r_{6+p}^{*}(x)^{4}\vee 1)\nonumber\\ &\hspace{3.5in} \times g(x)dxg(y)dy\nonumber\\&\quad+ k^{3\beta}n\int_{\mbb{R}^{d}}n\int_{\mbb{R}^{d}} \mathds{1}_{x \in A_{x_{0,1}}(y)}(k+1)^{3\beta}e^{-3\beta(k+1)^{-1} n\int_{\rec\left(x_{0,1},(x\vee y)_{x_{0,1}}\right)}g(z)dz}(r_{6+p}^{*}(x)^{4}\vee 1)\nonumber \\ &\hspace{3.5in}\times g(x)dxg(y)dy\nonumber\\&=(k+1)^{6\beta+2}\mcal{O}(\log^{d-1} ((k+1)^{-1} n)).
\end{align}
Using \eqref{eq:qbd} and arguing similarly, we also have
\begin{align}\label{g34}
    &n\int_{\mbb{R}^{d}}n\int_{\mbb{R}^{d}}c_{\zeta(k+1)^{-1},n,x_{0,1}}(x)^{28/(6+p/2)}q_{n}^{(1)}(x,y)^{3\beta}g(x)dxg(y)dy\nonumber\\&\le k^{3\beta}n\int_{\mbb{R}^{d}}n\int_{\mbb{R}^{d}} \mathds{1}_{x \in A_{x_{0,1}}(y)} c_{\zeta(k+1)^{-1},n,x_{0,1}}(x)^{28/(6+p/2)}c_{(k+1)^{-1},n,x_{0,1}}((x\vee y)_{x_{0,1}})^{3\beta}\nonumber\\ &\hspace{4.0in}\times g(x)dxg(y)dy\nonumber\\&\lesssim k^{3\beta}n\int_{\mbb{R}^{d}}n\int_{\mbb{R}^{d}} \mathds{1}_{x \in A_{x_{0,1}}(y)}c_{\zeta(k+1)^{-1},n,x_{0,1}}(x)^{28/(6+p/2)}(k+1)^{3\beta}\nonumber\\ &\hspace{2.5in}\times e^{-3\beta (k+1)^{-1}n\int_{z\in\rec(x_{0,1},(x\vee y)_{x_{0,1}})}g(z)dz}g(x)dxg(y)dy\nonumber\\&\qquad +k^{3\beta}n\int_{\mbb{R}^{d}}n\int_{\mbb{R}^{d}} \mathds{1}_{x \in A_{x_{0,1}}(y)}c_{\zeta(k+1)^{-1},n,x_{0,1}}(x)^{28/(6+p/2)}(k+1)^{3\beta-1}\nonumber\\ &\hspace{2.5in}\times c_{3\beta(k+1)^{-1},n,x_{0,1}}((x\vee y)_{x_{0,1}})g(x)dxg(y)dy\nonumber\\&:=B_1+B_2.
\end{align}
By Lemma \ref{lemma9}, we have $B_1,B_2=(k+1)^{6\beta+2+\lceil 28/(6+p/2)\rceil}\mcal{O}(\log^{d-1}((k+1)^{-1}n))$. Therefore, combining \eqref{g33} and \eqref{g34}, we obtain 
\begin{align}\label{G22}
    n\mbb{Q}f_{2,2,0,3\beta,3}^{(1,1,1,1)}=\Omega_{1,n}^{4}(k+1)^{6\beta+2+\lceil 28\varsigma\rceil} \mcal{O}(\log^{d-1}((k+1)^{-1}n)).
\end{align}
Note also that by \eqref{fij2} and \eqref{Grf},
\begin{align*}
    f_{2,2,0,3\beta,2}^{(1,1,1,1)}(y)&\lesssim \Omega_{1,n}^{4}\left(n\int_{\mbb{R}^{d}}(r_{6+p}^{*}(x)^{4}\vee 1)e^{-3\beta r_{n}^{(1)}(y,x)}g(x)dx\right.\\&\left.\qquad+k^{28\zeta/(6+p/2)}n\int_{\mbb{R}^{d}}c_{\zeta(k+1)^{-1},n,x_{0,1}}(x)^{28/(6+p/2)}e^{-3\beta r_{n}^{(1)}(y,x)}g(x)dx\right).
\end{align*}
By \eqref{eq:rnbd} and Lemma \ref{lemma10}, we have
\begin{align}\label{g35}
    &n\int_{\mbb{R}^{d}}n\int_{\mbb{R}^{d}}e^{-3\beta r_{n}^{(1)}(y,x)}(r_{6+p}^{*}(x)^{4}\vee 1)g(x)dxg(y)dy\nonumber\\&\lesssim k^{3\beta}n\int_{\mbb{R}^{d}}n\int_{\mbb{R}^{d}}\mathds{1}_{x\in\rec(x_{0,1},y)}e^{-3\beta (k+1)^{-1}n \int_{\rec(x_{0,1},y)}g(z)dz}(r_{6+p}^{*}(x)^{4}\vee 1)g(x)dxg(y)dy\nonumber\\&=(k+1)^{3\beta+2}\mcal{O}(\log^{d-1}((k+1)^{-1}n)).
\end{align}
Again using that $\zeta<3\beta$, \eqref{eq:rnbd} and \eqref{eq:cint} yield
\begin{align}\label{g36}
    &n\int_{\mbb{R}^{d}}n\int_{\mbb{R}^{d}}c_{\zeta(k+1)^{-1},n,x_{0,1}}(x)^{28/(6+p/2)}e^{-3\beta r_{n}^{(1)}(y,x)}g(x)dxg(y)dy\nonumber\\&\le n\int_{\mbb{R}^{d}}c_{\zeta(k+1)^{-1},n,x_{0,1}}(x)^{28/(6+p/2)}\left(n\int_{\mbb{R}^{d}}e^{-\zeta r_{n}^{(1)}(y,x)}g(y)dy\right)g(x)dx
\nonumber\\&\lesssim k^{3\beta}n\int_{\mbb{R}^{d}}c_{\zeta(k+1)^{-1},n,x_{0,1}}(x)^{28/(6+p/2)+1}g(x)dx\nonumber\\&=(k+1)^{3\beta+2+ 28/(6+p/2)}\mcal{O}(\log^{d-1}((k+1)^{-1}n)).
\end{align}
From \eqref{g35} and \eqref{g36}, we have
\begin{align}\label{G23}
    n\mbb{Q}f_{2,2,0,3\beta,2}^{(1,1,1,1)}=\Omega_{1,n}^{4}(k+1)^{3\beta+2+\lceil 28\varsigma\rceil}\mcal{O}(\log^{d-1}((k+1)^{-1}n)).
\end{align}
Combining \eqref{G21}, \eqref{G22} and \eqref{G23}, we obtain
\begin{align}\label{G25}
    n\mbb{Q}f_{2,2,0,3\beta}^{(1,1,1,1)}=\Omega_{1,n}^{4}(k+1)^{6\beta+2+\lceil 28\varsigma\rceil}\mcal{O}(\log^{d-1}((k+1)^{-1}n)).
\end{align}
Now \eqref{G24} and \eqref{G25} together with \eqref{rhoboundw}, \eqref{omeganexpress} (\eqref{rhobound} and  \eqref{Omegabound}, respectively, in the uniform case) and Lemma \ref{lemma11} yield
\begin{align}\label{gamma3ii}
&\left(\frac{n\mbb{Q}\left((\kappa_{n}^{(1)}+g_{n}^{(1)})^{6\beta}\left(G_{n}^{(1)}\right)^{4}\right)}{(\varrho_n^{(1)})^{4}}\right)^{\frac{1}{2}}+\left(\frac{n\mbb{Q}f_{2,2,0,3\beta}^{(1,1,1,1)}}{(\varrho_{n}^{(1)}\varrho_{n}^{(1)})^{2}}\right)^{\frac{1}{2}}\nonumber\\&=
\left\{
   \begin{aligned}
   &k^{3\zeta\beta+3\beta+1+\lceil 14\varsigma \rceil}W_{1}(n,k)\mcal{O}(\log^{(d-1)/2}n),\quad \text{for general weights},\\
   &k^{3\zeta\beta+3\beta+\lceil 14\varsigma\rceil}\mcal{O}(\log^{-(d-1)/2}n), \quad\text{for uniform weights}.
   \end{aligned}
   \right.
\end{align}

Finally, we are left to bound $\Gamma_{5}$ and $\Gamma_{6}$ defined at \eqref{GAMMA5} and \eqref{GAMMA6} respectively. By similar arguments as those used in bounding $\Gamma_{1}$ (with $i=j=l=t=1$ and $s=1$), one can show
\begin{align}\label{gamma4ii}
    &\left(\frac{n\mbb{Q}\left(f_{1,1,1/2,\beta}^{(1,1,1,1)}\right)^{2}}{\varrho_{n}^{(1)}(\varrho_{n}^{(1)}\varrho_{n}^{(1)})^{2}}\right)^{\frac{1}{3}}=
    \left\{
   \begin{aligned}
   & k^{4\beta/3+1+\lceil 35\varsigma/3\rceil}W_{1}(n,k)^{5/6}\mcal{O}(\log^{(d-1)/3}n),\\ &\hspace{2.5in} \text{for general weights},\\
   & k^{4\beta/3+1/6+\lceil 35\varsigma/3\rceil}\mcal{O}(\log^{-(d-1)/2}n),\hspace{.3in}\text{for uniform weights}.
   \end{aligned}
   \right.
\end{align}
and
\begin{align}\label{gamma5ii}
    &\left(\frac{n\mbb{Q}\left(f_{1,1,1,\beta}^{(1,1,1,1)}\right)^{2}}{(\varrho_{n}^{(1)}\varrho_{n}^{(1)}\varrho_{n}^{(1)})^{2}}\right)^{\frac{1}{4}}=
        \left\{
   \begin{aligned}
   & k^{\beta+3/4+\lceil 21\varsigma/2\rceil}W_{1}(n,k)^{3/4}\mcal{O}(\log^{(d-1)/4}n),\\ &\hspace{2.5in}  \text{for general weights},\\
   & k^{\beta+\lceil 21\varsigma/2\rceil}\mcal{O}(\log^{-(d-1)/2}n),\hspace{.5in} \text{for uniform weights}.
   \end{aligned}
   \right.
\end{align}

Therefore, combining \eqref{gamma1ii}, \eqref{G24}, \eqref{gamma3ii}, \eqref{gamma4ii} and \eqref{gamma5ii} and recalling that $\Gamma_0=0$, we conclude that the sums of all the contributions from $\Gamma_{s}$, $s\in\{0,1,3,4,5,6\}$ with $i=j=l=t$ (after writing the power of the sums as sums of powers, up to constants, using Lemma \ref{eq:basicanalysis}) is of the order
\begin{align}\label{maxrate}
\left\{
   \begin{aligned}
   &k^{6\zeta\beta+6\beta+3/2+\lceil 21\varsigma\rceil}\max_{j\in\{1,4\}}\left(W(n,k)^{1/2+1/j}\mcal{O}(\log^{(d-1)/j}n)\right),\\ &\hspace{2.8in} \text{for general weights},\\
   &k^{6\zeta\beta+6\beta+1/2+\lceil 21\varsigma\rceil}\mcal{O}(\log^{-(d-1)/2}n),\hspace{.5in} \text{for uniform weights},
   \end{aligned}
   \right.
\end{align}
where we recall from~\eqref{eq:wnk} that
\begin{align*}
W(n,k)\coloneqq  \frac{\underset{i=1,\ldots,m}{\max} \left(\sup_{(x,\eta): |\eta| \le 9} \left\|W_{nx}(x_{0,i},\mathcal{P}_{x,\eta})\right\|_{L_{6+p}}\right)^{2}}{\underset{i=1,\ldots,m}{\min}\mbb{E}\left[\sum_{\bm{x}}W_{n\bm{x}}(x_{0,i})^{2}\right]}.
\end{align*}
Next, we will consider the case when $i,j,l,t$ are not all equal. According to Theorems \ref{d2d3} and \ref{dconvex}, only $\Gamma_{1},\Gamma_{4},\Gamma_{5}$ and $\Gamma_{6}$ will change. We first focus on the case when $i\neq j$, and without loss of generality, take $i=1$ and $j=2$. 

Again, we start with $\Gamma_{1}$. Note that by \eqref{fij1}, \eqref{Grf} and Cauchy's inequality,
\begin{align*}
    &\left(f_{1,1,0,\beta,1}^{(1,2,1,1)}(y)\right)^2\\
    & \lesssim  (\Omega_{1,n}^{4}+ \Omega_{2,n}^{4})\left(n\int_{\mbb{R}^{d}}(r_{6+p}^{*}(x)^{2}\vee 1) e^{-\beta r_{n}^{(1)}(x,y)} g(x)dx\right)^{2}\\&\qquad+\Omega_{1,n}^{4}k^{28\zeta/(6+p/2)}\left(n\int_{\mbb{R}^{d}}c_{\zeta(k+1)^{-1},n,x_{0,1}}(x)^{14/(6+p/2)}e^{-\beta r_{n}^{(1)}(x,y)} g(x)dx\right)^{2}\\ &\qquad+\Omega_{2,n}^{4} k^{28\zeta/(6+p/2)}\left(n\int_{\mbb{R}^{d}}c_{\zeta(k+1)^{-1},n,x_{0,2}}(x)^{14/(6+p/2)}e^{-\beta r_{n}^{(1)}(x,y)} g(x)dx\right)^{2}.
\end{align*}
Therefore, it suffices to derive bounds for the final additional term compared to \eqref{equalgamma11}. Denote
\begin{align*}
    \mcal{I}_{\text{add}}:=n\int_{\mbb{R}^d}\left(n\int_{\mbb{R}^{d}}c_{\zeta(k+1)^{-1},n,x_{0,2}}(x)^{14/(6+p/2)}e^{-\beta r_{n}^{(1)}(x,y)}g(x)dx\right)^2g(y)dy.
\end{align*}

Fix $\delta \in (0,1/2 (g(x_{0,1})\wedge g(x_{0,2})))$ and $\epsilon$ as in Lemma \ref{lemma6} such that the conclusion therein holds for both $x_{0,1}, x_{0,2}$. 
We argue as in Lemma \ref{lemma6} using the partition \eqref{partC} for the inside integral in $\mcal{I}_{\text{add}}$. For any $\underline j_1 \subseteq [d]$ with $|\underline j_1|=j_1 \ge 1$, we have
\begin{align*}
    &n\int_{\mbb{R}^d}\Bigg(n\int_{\mcal{C}_{\underline j_1}^{x_{0,1}}}c_{\zeta(k+1)^{-1},n,x_{0,2}}(x)^{14/(6+p/2)}e^{-\beta r_{n}^{(1)}(x,y)}g(x)dx\Bigg)^2g(y)dy\\
    & \lesssim n\int_{\mbb{R}^d}\Bigg(n\int_{\mcal{C}_{\underline j_1}^{x_{0,1}}}c_{\zeta(k+1)^{-1},n,x_{0,2}}(x)^{14/(6+p/2)}e^{-\frac{\beta n}{k+1}(g(x_{0,1})-\delta)\epsilon^{d-j}\prod_{l \in \underline{j}_1}|x^{(l)}-x_{0,1}^{(l)}|} g(x)\\
    & \qquad \qquad \qquad \qquad \times \mathds{1}_{y^{\underline j_1} \in \rec(x^{\underline j_1}, x_{0,1}^{\underline j_1})} dx \Bigg)^2g(y)dy.
\end{align*}
Now, we further bound $c_{\zeta(k+1)^{-1},n,x_{0,2}}(x)$ using Lemma \ref{lemma6}. In the upper bound therein, the first term 
\begin{align*}
ne^{-\alpha (k+1)^{-1} n(g(x_{0,1})-\delta)\epsilon^{d}}\int_{\mbb{R}^{d}}\phi(x)g(x)dx,
\end{align*}
is exponentially small, so consider the summand corresponding to $\underline j_2$ with $|\underline j_2|=j_2 \le d-1$, i.e., $\mathds{1}_{x^{\underline{j}_2} \in \mcal{C}_{\bm{\epsilon}}(x_{0,2}^{\underline{j}_2})}=1$. 
First assume that $\underline j_1 \cap \underline j_2 = \emptyset$. Writing $a=14/(6+p/2)$ for ease, we have
\begin{align*}
&n\int_{\mbb{R}^d}\Bigg(n\int_{\mcal{C}_{\underline j_1}^{x_{0,1}}}e^{-\zeta \left(\frac{a n(g(x_{0,2})-\delta)\epsilon^{d-j_2}}{k+1}\right)|x^{\underline{j}_2}-x_{0,2}^{\underline{j}_2}|/2}\\
& \qquad \qquad \times \left(\left|\log \left(\zeta \left(\frac{n(g(x_{0,2})-\delta)\epsilon^{d-j_2}}{k+1}\right)|x^{\underline{j}_2}-x_{0,2}^{\underline{j}_2}|\right)\right|^{a(j_2-1)}+1\right)\nonumber\\& \qquad \qquad \qquad \times e^{-\frac{\beta n}{k+1}(g(x_{0,1})-\delta)\epsilon^{d-j}|x^{\underline{j}_1}-x_{0,1}^{\underline{j}_1}|} g(x)\mathds{1}_{y^{\underline j_1} \in \rec(x^{\underline j_1}, x_{0,1}^{\underline j_1})} dx\Bigg)^2g(y)dy\\
& \lesssim n\int_{\mcal{C}_{\underline j_1}^{x_{0,1}}}e^{-2\zeta \left(\frac{a n(g(x_{0,2})-\delta)\epsilon^{d-j_2}}{k+1}\right)|x^{\underline{j}_2}-x_{0,2}^{\underline{j}_2}|/2}\\
& \qquad \qquad \times \left(\left|\log \left(\zeta \left(\frac{n(g(x_{0,2})-\delta)\epsilon^{d-j_2}}{k+1}\right)|x^{\underline{j}_2}-x_{0,2}^{\underline{j}_2}|\right)\right|^{2a(j_2-1)}+1\right) g(x) dx\\
& \qquad \times n\int_{\mbb{R}^{j_1}} \Bigg(n \int_{\mcal{C}_{\underline j_1}^{x_{0,1}}} e^{-2\frac{\beta n}{k+1}(g(x_{0,1})-\delta)\epsilon^{d-j}|x^{\underline{j}_1}-x_{0,1}^{\underline{j}_1}|} g(x)\mathds{1}_{y^{\underline j_1} \in \rec(x^{\underline j_1}, x_{0,1}^{\underline j_1})} dx\Bigg) g^{\underline j_1}(y^{\underline j_1}) dy^{\underline j_1}\\
& \lesssim k^{3} \mathcal{O}(\log^{j_1+j_2-2} n),
\end{align*}
where we have used the Cauchy-Schwarz inequality in the penultimate step. 
Since $j_1+j_2 \le d$, this is at most of the order $k^3\log^{d-2} n$. 

Now, we are left with the case when $\underline j_1 \cap \underline j_2 \neq \emptyset$. Let us illustrate with the computation when $\underline j_2\backslash \underline j_1\neq \emptyset$. Let $\underline j_2'= \underline j_2 \setminus \underline j_1$ with $|\underline j_2'| = j_2'$.
Writing $a=14/(6+p/2)$ for ease, we have
\begin{align*}
&n\int_{\mbb{R}^d}\Bigg(n\int_{\mcal{C}_{\underline j_1}^{x_{0,1}}}e^{-\zeta \left(\frac{a n(g(x_{0,2})-\delta)\epsilon^{d-j_2}}{k+1}\right)|x^{\underline{j}_2}-x_{0,2}^{\underline{j}_2}|/2}\\&\qquad\qquad\left(\left|\log \left(\zeta \left(\frac{n(g(x_{0,2})-\delta)\epsilon^{d-j_2}}{k+1}\right)|x^{\underline{j}_2}-x_{0,2}^{\underline{j}_2}|\right)\right|^{a(j_2-1)}+1\right)\nonumber\\& \qquad \qquad\quad \times e^{-\frac{\beta n}{k+1}(g(x_{0,1})-\delta)\epsilon^{d-j}|x^{\underline{j}_1}-x_{0,1}^{\underline{j}_1}|} g(x)\mathds{1}_{y^{\underline j_1} \in \rec(x^{\underline j_1}, x_{0,1}^{\underline j_1})} dx\Bigg)^2g(y)dy\\
& \lesssim n\int_{\mcal{C}_{\underline j_1}^{x_{0,1}}}e^{-2\zeta \left(\frac{a n(g(x_{0,2})-\delta)\epsilon^{d-j_2}}{k+1}\right)|x^{\underline{j}_2'}-x_{0,2}^{\underline{j}_2'}|/2}\\&\qquad \qquad \left(\left|\log \left(\zeta \left(\frac{n(g(x_{0,2})-\delta)\epsilon^{d-j_2}}{k+1}\right)|x^{\underline{j}_2}-x_{0,2}^{\underline{j}_2}|\right)\right|^{2a(j_2-1)}+1\right) g(x) dx\\
& \qquad \qquad \quad\times n\int_{\mbb{R}^{j_1}} \Bigg(n \int_{\mcal{C}_{\underline j_1}^{x_{0,1}}} e^{-\frac{2a' n}{k+1}|x^{\underline{j}_1}-x_{0,1}^{\underline{j}_1}|} g(x)\mathds{1}_{y^{\underline j_1} \in \rec(x^{\underline j_1}, x_{0,1}^{\underline j_1})} dx\Bigg) g^{\underline j_1}(y^{\underline j_1}) dy^{\underline j_1}\\
& \lesssim k^{3} \mathcal{O}(\log^{j_1+j_2'-2} n),
\end{align*}
where we have used the Cauchy-Schwarz inequality, and 
$$
a'=\min\left\{\zeta \left(a (g(x_{0,2})-\delta)\epsilon^{d-j_2}/2\right), \beta (g(x_{0,1})-\delta)\epsilon^{d-j}\right\}.
$$
Since $j_1+j_2' \le d$, this is again at most of the order $k^3\log^{d-2} n$. The case when $\underline j_1 \cap \underline j_2 \neq \emptyset$, but $\underline j_2\backslash \underline j_1= \emptyset$ can be dealt with similarly as above, by putting the exponential factor corresponding to the common coordinates in one of the exponential factors, instead of in both. Finally, the case when $j_1=0$ is trivial, since in this case, 
\begin{equation*}\label{eq:R1c}
    e^{-\beta r_{n}^{(1)}(x,y)}\lesssim k^\beta e^{-\frac{\beta}{k+1}n\epsilon^{d}(g(x_{0,1})-\delta)}.
\end{equation*}
Combining all the above analysis, we conclude that the additional term appeared $\mcal{I}_{\text{add}}$ is of lower order compared to the corresponding bound for $f_{1,1,0,\beta,1}^{(1,1,1,1)}$.

Using similar arguments as for $\left(f_{1,1,0,\beta,2}^{(1,2,1,1)}(y)\right)^2$ and $\left(f_{1,1,0,\beta,3}^{(1,2,1,1)}(y)\right)^2$, combining the cases $i=j$ in \eqref{gamma1ii} and the case when $i\neq j$, from \eqref{GAMMA1} we finally obtain
\begin{align*}
    \Gamma_{1}=\left\{
   \begin{aligned}
   &mk^{2\beta+3/2+\lceil 14\varsigma\rceil}
   W(n,k)\mcal{O}(\log^{(d-1)/2}n),\qquad\qquad \text{for general weights},\\
   &mk^{2\beta+1/2+\lceil 14\varsigma\rceil}\mcal{O}(\log^{-(d-1)/2}n),\qquad\qquad\quad\quad\;\;\; \text{for uniform weights},
   \end{aligned}
   \right.
\end{align*}
where $\varsigma$ is given in \eqref{varsignmawnk} and $W(n,k)$ is defined in \eqref{eq:wnk}.

Similar arguments using \eqref{gamma3ii}, \eqref{gamma4ii} and \eqref{gamma5ii} as well as \eqref{GAMMA4}, \eqref{GAMMA5} and \eqref{GAMMA6} yield 
\begin{align*}
    \Gamma_{4}=
\left\{
   \begin{aligned}
   &mk^{3\zeta\beta+3\beta+1+\lceil 14\varsigma\rceil}
   W(n,k)\mcal{O}(\log^{(d-1)/2}n),\qquad \quad \text{for general weights},\\
   &mk^{3\zeta\beta+3\beta+\lceil 14\varsigma\rceil}\mcal{O}(\log^{-(d-1)/2}n),\qquad\qquad\text{for uniform weights},
   \end{aligned}
   \right.
\end{align*}
\begin{align*}
    \Gamma_{5}&=
    \left\{
   \begin{aligned}
   &m^{13/6} k^{4\beta/3+1+\lceil 35\varsigma/3\rceil}
   W(n,k)^{5/6}\mcal{O}(\log^{(d-1)/3}n), \text{ for general weights},\\
   &m^{13/6} k^{4\beta/3+1/6+\lceil 35\varsigma/3\rceil}\mcal{O}(\log^{-(d-1)/2}n),\quad\text{for uniform weights},
   \end{aligned}
   \right.
\end{align*}
and
\begin{align*}
    \Gamma_{6}&=\left\{
   \begin{aligned}
   &m^{2} k^{\beta+3/4+\lceil 21\varsigma/2\rceil}
   W(n,k)^{3/4}\mcal{O}(\log^{(d-1)/4}n),\qquad \quad  \text{for general weights},\\
   &m^{2} k^{\beta+\lceil 21\varsigma/2\rceil}\mcal{O}(\log^{-(d-1)/2}n),\qquad \qquad\text{for uniform weights}.
   \end{aligned}
   \right.
\end{align*}

\begin{proof}[Proofs of Theorem \ref{thmrfw} and Corollary \ref{thmrf}]
Putting together all bounds above on $\Gamma_{s}$ for $s\in\{0,1,3,4,5,6\}$ as well as a similar bound for $\Gamma_2$, and $\Gamma_1$ with $p_0=4$, 
we obtain the proof of Theorem \ref{thmrfw} and Corollary \ref{thmrf} by invoking Theorems \ref{d2d3} and \ref{dconvex}.
\end{proof}

\subsection{Proof of Proposition~\ref{varbias}}\label{sec:propvarbias}

\begin{proof}[Proof of Proposition \ref{varbias}]\label{varbiaspf} We start with studying the following general integral more carefully. By the Mecke formula, for any function $\phi(x)$ as considered in Lemma \ref{lemma11} and Remark \ref{rem:B7}, we have
\begin{align*}
    \mbb{E}\bigg(\sum_{\bm{x}\in\Png}\phi(\bm{x})\mathds{1}_{\bm{x}\in\mcal{L}_{n,k}(x_{0})}\bigg)=n\int_{\mbb{R}^{d}}\psi(n,k,x_{0},x)\phi(x)g(x)dx.
\end{align*}
When $g$ and $\phi$ are assumed to be continuous, by Remark \ref{rem:B7}, we in particular have that 
\begin{align*}
    \lim_{n\rightarrow\infty}\frac{n\int_{\mbb{R}^{d}}\psi(n,k,x_{0},x)\phi(x)g(x)dx}{k\log^{d-1}n}=\frac{2^{d}}{(d-1)!}\phi(x_{0}).
\end{align*}
The rate of convergence for this limit, which we estimate below in \eqref{nonphi}, plays a key role in analyzing the bias. 


Let $\phi$ and $g$ be a H\"{o}lder continuous functions at $x_{0}$ with parameters $L_{\phi},\gamma_{\phi}>0$ and $L_{g},\gamma_{g}>0$, respectively. Note for any $\delta\in(0,1/2\ g(x_0))$ and $\epsilon>0$ as in \eqref{eq:gpcont}, according to the proof of Lemma \ref{lemma11}, by \eqref{ej1}, \eqref{ej2} and \eqref{ej11}, we have
\begin{align}\label{choiceofed}
    &\frac{2^{d}}{(d-1)!}\frac{(\phi(x_{0})-\delta)(g(x_{0})-\delta)}{g(x_{0})+\delta}\left(k\log^{d-1}(n \Delta_{d}'^{d}\epsilon^{d})+k\mcal{O}\left(\log k \log^{d-2}(n\Delta_{d}'^{d}\epsilon^{d})\right)\right)\nonumber\\&\le n\int_{\mbb{R}^{d}}\psi(n,k,x_{0},x)\phi(x)g(x)dx\nonumber\\&\le \frac{2^{d}}{(d-1)!}\frac{(\phi(x_{0})+\delta)(g(x_{0})+\delta)}{g(x_{0})-\delta}\left(k\log^{d-1}(n \Delta_{d}^{d}\epsilon^{d})+k\mcal{O}\left(\log k \log^{d-2}(n\Delta_{d}^{d}\epsilon^{d})\right)\right)\nonumber\\&\quad+\sum_{j=0}^{k-1}ne^{-n(g(x_{0})-\delta)\epsilon^{d}}\frac{(n(g(x_{0})-\delta)\epsilon^d)^{j}}{j!}\int_{\mbb{R}^{d}}\phi(x) g(x)dx\nonumber\\&\qquad\qquad +\sum_{l=1}^{d-1}\frac{2^l\binom{d}{l}\Lambda(l,\epsilon)}{(g(x_0)-\delta)\epsilon^{d-l}}k\mcal{O}\left(\log k \log^{l-1}(n\Delta_{l}^{l}\epsilon^{d})\right).
\end{align}
Now to obtain a rate of convergence, instead of fixing the pair $(\delta,\epsilon)$ as in the proofs of Theorem \ref{thmrfw} and Corollary \ref{thmrf}, we use the assumed H\"{o}lder continuity to determine $(\delta,\epsilon)$ more explicitly as a function of $n$. 
We have for any $x\in\mcal{C}_{\bm{\epsilon}}(x_{0})$,
\begin{align}\label{deltag}
    |g(x_{0})-g(x)|\le L_{g}(d\epsilon^{2})^{\frac{\gamma_{g}}{2}}=L_{g}d^{\frac{\gamma_{g}}{2}}\epsilon^{\gamma_{g}}=:\delta_{g},
\end{align}
and 
\begin{align}\label{deltaphi}
    |\phi(x_{0})-\phi(x)|\le L_{\phi}(d\epsilon^{2})^{\frac{\gamma_{\phi}}{2}}=L_{\phi}d^{\frac{\gamma_{\phi}}{2}}\epsilon^{\gamma_{\phi}}=:\delta_{\phi}.
\end{align}
To make sure that $\epsilon^{-(d-l)} \log^{l-1} n = o(\log^{d-1} n)$ in \eqref{choiceofed}, we pick $\epsilon\equiv \epsilon(n)=\log^{-\zeta}n$ for some $0<\zeta<1$. Thus 
for $1\le l\le d-1$,
\begin{align*}
    \frac{1}{\epsilon^{d-l}}\log^{l-1}n=\log^{(d-l)\zeta+(l-1)}\le \log^{d-2+\zeta}n = o(\log^{d-1}n),
\end{align*}
while for $l=d-1$, we have
\begin{align*}
    \frac{1}{\epsilon^{d-l}}\log^{l-1}n=\log^{d-2 + \zeta}n>\log^{d-2}n.
\end{align*}
Now, from \eqref{deltag}, we have $\delta_{g}=L_{g}d^{\frac{\gamma_{g}}{2}}\log^{-\zeta\gamma_{g}}n$ so that choosing
$
n\ge \exp\left[\left(\frac{2L_{g}d^{\frac{\gamma_{g}}{2}}}{g(x_{0})}\right)^{\frac{1}{\zeta\gamma_{g}}} \right],
$ 
ensures that $\delta_{g}<1/2\ g(x_{0})$. Also, noting that $\delta_{\phi} =L_{\phi}d^{\frac{\gamma_{\phi}}{2}}\log^{-\zeta\gamma_{\phi}}n$, starting with $\delta \equiv \delta(n) = \delta_{g} \vee \delta_{\phi}$, we have that the above choice of $\epsilon\equiv \epsilon(n)$ satisfies \eqref{eq:gpcont}. Thus from \eqref{choiceofed} we obtain


\begin{align}\label{nonphi}
&\left|\frac{n\int_{\mbb{R}^{d}}\psi(n,k,x_{0},x)\phi(x)g(x)dx}{k\log^{d-1}n}-\frac{2^{d}}{(d-1)!}\phi(x_{0})\right|\nonumber\\
&=\mcal{O}\left((\log^{-\zeta(\gamma_{g}\wedge \gamma_{\phi})}n) 
\vee (\log k \log^{-(1-\zeta)}n)\right).
\end{align}
Moreover, when $\phi(x_{0})=0$, we have
\begin{align}\label{nonzero}
    \left|\frac{n\int_{\mbb{R}^{d}}\psi(n,k,x_{0},x)\phi(x)g(x)dx}{k\log^{d-1}n}\right|=\mcal{O}(\delta_{g}\vee \delta_{\phi})=\mcal{O}(\log^{-\zeta(\gamma_{g}\wedge \gamma_{\phi})}n).
\end{align}

The above is now being applied for the actual proof of Proposition \ref{varbias}. From \eqref{rnkunif}, by Fubini's theorem, we have
\begin{align*}
    &\mbb{E}r_{n,k}(x_{0})-r_0(x_{0})\\
    &=\mbb{E}\left(\sum_{(\bm{x},\bm{\varepsilon}_{\bm{x}})\in \Pngm}\frac{\mathds{1}_{\bm{x}\in\mcal{L}_{n,k}(x_{0})}}{L_{n,k}(x_0)}(\bm{y}_{\bm{x}}-r_0(x_{0}))\right)\nonumber\\&=\mbb{E}\left(\sum_{(\bm{x},\bm{\varepsilon}_{\bm{x}})\in \Pngm}\frac{\mathds{1}_{\bm{x}\in\mcal{L}_{n,k}(x_{0})}}{L_{n,k}(x_0)}[(\bm{y}_{\bm{x}}-r_0(\bm{x})) + (r_0(\bm{x}) - r_0(x_0))]\right) 
    \nonumber\\&=\mbb{E}\left(\sum_{(\bm{x},\bm{\varepsilon}_{\bm{x}})\in \Pngm}\frac{\mathds{1}_{\bm{x}\in\mcal{L}_{n,k}(x_{0})}}{L_{n,k}(x_0)}(r_0(\bm{x})-r_0(x_{0}))\right),
\end{align*}
where the final step is due to the fact that $\mbb{E}[y_{\bm{x}} | \bm{x}] = r_0(x)$. In order to apply \eqref{nonzero} with $\phi(x)=r_0(x)-r_0(x_0)$, we aim to substitute $L_{n,k}(x_0)$ with $\mbb{E}L_{n,k}(x_0)$ and bound the error term. To this end, by the triangle inequality we can write
\begin{align}\label{trij1}
    &\left|\mbb{E}r_{n,k}(x_{0})-r_0(x_{0})\right|\le  \left|\mbb{E}\left(\sum_{(\bm{x},\bm{\varepsilon}_{\bm{x}})\in \Pngm}\frac{\mathds{1}_{\bm{x}\in\mcal{L}_{n,k}(x_{0})}}{\mbb{E}L_{n,k}(x_0)}(r_0(\bm{x})-r_0(x_{0}))\right)\right|\nonumber\\&\quad +\mbb{E}\left(\sum_{(\bm{x},\bm{\varepsilon}_{\bm{x}})\in \Pngm}\left|\frac{1}{L_{n,k}(x_0)}-\frac{1}{\mbb{E}L_{n,k}(x_0)}\right|\mathds{1}_{\bm{x}\in\mcal{L}_{n,k}(x_{0})} (r_0(\bm{x})-r_0(x_{0}))\right)\nonumber\\&\quad =:\mcal{J}_1+\mcal{J}_2.
\end{align}
Plugging $\phi(x)=r_0(x)-r_0(x_0)$, which is H\"{o}lder continuous at $x_{0}$ with parameters $L_{1},\gamma_{1}>0$ by our assumption, in \eqref{nonzero} and using Lemma \ref{lemma11}, we conclude that for any $\zeta \in (0,1)$,
\begin{align*}
    \mcal{J}_{1}=\mcal{O}(\log^{-\zeta(\gamma_{g}\wedge \gamma_{1})}n).
\end{align*}
To estimate $\mcal{J}_{2}$, we proceed with the following two subcases.
First, on $\mcal{E}:=\{|L_{n,k}(x_0)-\mbb{E}L_{n,k}(x_0)|<(\mbb{E}L_{n,k}(x_0))^{3/4}\}$, we have
\begin{align*}
    \left|\frac{1}{\mbb{E}L_{n,k}(x_{0})}-\frac{1}{L_{n,k}(x_{0})}\right|= \frac{|L_{n,k}(x_0)-\mbb{E}L_{n,k}(x_0)|}{L_{n,k}(x_0) \mbb{E}L_{n,k}(x_0)} 
    < \frac{1}{L_{n,k}(x_0)(\mbb{E}L_{n,k}(x_0))^{\frac{1}{4}}}.
\end{align*}
Next, on $\mcal{E}^{c}=\{|L_{n,k}(x_0)-\mbb{E}L_{n,k}(x_0)|\ge (\mbb{E}L_{n,k}(x_0))^{3/4}\}$, we simply bound
\begin{align}\label{moments12}
    \left|\frac{1}{\mbb{E}L_{n,k}(x_{0})}-\frac{1}{L_{n,k}(x_{0})} \mathds{1}_{L_{n,k}(x_{0}) \ge 1}\right| 
    \lesssim 1,
\end{align}
which holds since $\mbb{E}L_{n,k}(x_0) = \mcal{O}(k \log^{d-1} n)$ by Remark \ref{rem:B7}. On the other hand, arguing as in the proof of Lemma \ref{concetrationoftailKPNN}, we have for any $r \ge 1$ that
\begin{align}\label{tails12}
    \mbb{P}(\mcal{E}^{c}) &= \mbb{P}(|L_{n,k}(x_0)-\mbb{E}L_{n,k}(x_0)|\ge (\mbb{E}L_{n,k}(x_0))^{3/4})\nonumber\\&\lesssim \frac{\mbb{E}(L_{n,k}(x_0)-\mbb{E}L_{n,k}(x_0))^{2r}}{(\mbb{E}L_{n,k}(x_0))^{3r/2}}\nonumber\\&\lesssim_{r} (k\log^{d-1}n)^{-3r/2}k^{r(\tau_{\text{conc}}+1)} \log^{r(d-1)} n\nonumber\\&=k^{r(\tau_{\text{conc}}+1)-3r/2}\log^{-(d-1)r/2} n,
\end{align}
where $\tau_{\text{conc}}$ are defined in the proof of Lemma \ref{concetrationoftailKPNN} and it satisfies $\tau_{\text{conc}}\le 10/r$. Taking $r=29$ so that $\tau_{\text{conc}}\le 10/29$, we thus have
\begin{align*}
    \mbb{P}(\mcal{E}^{c})\lesssim k^{-9/2}\log^{-29(d-1)/2} n.
\end{align*}
Using the above bounds with the fact that both $r_{0}(\bm{x})$ and $r_0(x_0)$ are uniformly bounded (almost surely) by our assumption, we now obtain
\begin{align*}
    \mcal{J}_{2}&\lesssim \left|\mbb{E}\bigg(\mathds{1}_{\mcal{E}}\left|\frac{1}{\mbb{E}L_{n,k}(x_{0}))}-\frac{1}{L_{n,k}(x_{0})}\right|\sum_{\bm{x}\in\Png}\mathds{1}_{\bm{x}\in\mcal{L}_{n,k}(x_{0})}\bigg)\right|\\&\quad +\left|\mbb{E}\bigg(\mathds{1}_{\mcal{E}^{c}}\left|\frac{1}{\mbb{E}L_{n,k}(x_{0})}-\frac{1}{L_{n,k}(x_{0})}\right|\sum_{\bm{x}\in\Png}\mathds{1}_{\bm{x}\in\mcal{L}_{n,k}(x_{0})}\bigg)\right|\\&\lesssim (\mbb{E}L_{n,k}(x_0))^{-1/4}+\mbb{E}\bigg(L_{n,k}(x_0)\mathds{1}_{\mcal{E}^{c}}\bigg)\\&\lesssim (k\log^{d-1}n)^{-1/4}+\sqrt{\mbb{E}L_{n,k}(x_0)^{2}\mbb{P}(\mcal{E}^{c}})\\&\lesssim (k\log^{d-1}n)^{-1/4}+k^{-5/4}\log^{-25(d-1)/4} n=\mcal{O}((k\log^{d-1}n)^{-1/4}),
\end{align*}
where in the penultimate step, we have also used Remark \ref{rmkonconcofkpnn}. Putting the bounds for $\mcal{J}_{1}$ and $\mcal{J}_{2}$ together in \eqref{trij1} now yields
\begin{align}\label{nonj1}
    \left|\mbb{E}r_{n,k}(x_{0})-r_0(x_{0})\right|=\mcal{O}\left((\log^{-\zeta(\gamma_{g}\wedge \gamma_{1})}n)\vee(k^{-1/4}\log^{-(d-1)/4}n)\right)
\end{align}
completing the proof.
\end{proof}

\section{Auxiliary Results}\label{Appendix D}

\begin{Lemma}\label{eq:basicanalysis}
The two inequalities below, follow by elementary analysis:
\begin{itemize}
    \item For $\{a_{i}\}_{i=1}^{\ell}\subset\mbb{R}$ and $\iota \ge 1$, $\left(\sum_{i=1}^{\ell}a_{i}\right)^{\iota}\le \ell^{\iota-1}\sum_{i=1}^{\ell}a_{i}^{\iota}$.
    \item For $\{a_{i}\}_{i=1}^{\iota}\subset \mbb{R}_{+}$ and $0<\iota<1$, we have that $\left(\sum_{i=1}^{\ell}a_{i}\right)^{\iota}\le \sum_{i=1}^{\ell}a_{i}^{\iota}$.
\end{itemize}
\end{Lemma}

\begin{Lemma}[Poisson concentration]\label{concforpoi}
    Let $h:[-1,\infty)\rightarrow \mbb{R}$ be given by $h(x):=\frac{2(1+x)\log (1+x)-x}{x^{2}}$, and let $\textup{Poi}(\lambda_0)$ be a Poisson random variable with parameter $\lambda_0>0$. Then, for any $x>0$, we have
    \begin{align*}
        \mbb{P}(\textup{Poi}(\lambda_0)\ge \lambda_0+x)\le e^{-\frac{x^2}{2\lambda_0}h\left(\frac{x}{\lambda_0}\right)},
    \end{align*}
    and, for any $0<x<\lambda_0$, 
    \begin{align*}
        \mbb{P}(\textup{Poi}(\lambda_0)\le \lambda_0-x)\le e^{-\frac{x^2}{2\lambda_0}h\left(-\frac{x}{\lambda_0}\right)}.
    \end{align*}
    In particular, this implies that for $x>0$,
    \begin{align*}
        \max\big[ \mbb{P}(\textup{Poi}(\lambda_0)\ge \lambda_0+x),\mbb{P}(\textup{Poi}(\lambda_0)\le \lambda_0-x)\big]\le e^{-\frac{x^{2}}{2(\lambda_0+x)}}.
    \end{align*}
\end{Lemma}
The proof of Lemma~\ref{concforpoi} follows by a standard use of Chernoff's bound, and is hence omitted.